\documentclass[11pt]{article}

\usepackage[margin=3cm]{geometry}

\usepackage[T1]{fontenc}
\usepackage[utf8]{inputenc}
\usepackage{hyperref}
\usepackage{amsmath,amsfonts}
\usepackage{amsthm}
\usepackage{amssymb}
\usepackage{mathtools}
\usepackage{verbatim}
\usepackage{placeins}
\usepackage{tikz}
\usetikzlibrary{calc}
\usepackage[round, comma]{natbib}
\usepackage{authblk}

\usepackage[title]{appendix}

\usepackage{titlesec}

\usepackage[draft,multiuser]{fixme}
\fxusetheme{colorsig}
\FXRegisterAuthor{peter}{Peter}{\color{blue}P}

\usepackage{bm}
\usepackage{todonotes}

\renewcommand{\textsc}[1]{\text{\it #1}}
\renewcommand{\P}{\ensuremath{\textsc{prop}}}
\newcommand{\Thetafull}{\ensuremath{\Theta}}
\newcommand{\Thetasmall}{\ensuremath{\Theta'}}

\newcommand{\grow}{\ensuremath{\text{GROW}}}
\newcommand{\gro}{\ensuremath{\text{GRO}}}

\newcommand{\commentout}[1]{}

\newcommand{\cE}{\ensuremath{\mathcal E}}
\newcommand{\cG}{\mathcal{G}}

\newcommand{\cH}{\mathcal{H}}
\newcommand{\cP}{\mathcal{P}}

\newcommand{\cT}{\mathcal{T}}
\newcommand{\cN}{\mathcal{N}}
\newcommand{\cW}{\mathcal{W}}
\newcommand{\cX}{\mathcal{X}}
\newcommand{\cY}{\mathcal{Y}}
\newcommand{\cZ}{\mathcal{Z}}

\newcommand{\cF}{\mathcal{F}}

\newcommand{\Yvec}{\ensuremath{\mathbf Y}}

\newcommand{\yvec}{\ensuremath{y}}
\newcommand{\Vvec}{\ensuremath{\mathbf V}}

\newcommand{\Ycomp}[1]{\ensuremath{Y}_{#1}}
\newcommand{\Vcomp}[1]{\ensuremath{V}_{#1}}
\newcommand{\Ydata}{\mathbf Y}
\newcommand{\Yb}{\Ydata}
\newcommand{\bX}{\mathbf X}
\newcommand{\bV}{\mathbf V}
\newcommand{\bx}{\mathbf x}

\newcommand{\by}{\mathbf y}
\newcommand{\bY}{\mathbf Y}

\newcommand{\bR}{\mathbf R}

\renewcommand{\S}{\ensuremath{\textsc e}}

\newcommand{\sval}{\ensuremath{E}}
\newcommand{\eval}{\ensuremath{E}}
\newcommand{\p}{\ensuremath{\textsc p}}
\newcommand{\pval}{\ensuremath{\textsc{p}}}
\newcommand{\reals}{\mathbb{R}}
\newcommand{\naturals}{\mathbb{N}}
\newcommand{\Exp}{\ensuremath{\mathbf E}}
\newcommand{\priordist}{\ensuremath{W}}
\newcommand{\priorset}{\ensuremath{\cW}}
\newcommand{\dif}{\,\textnormal{d}}

\newcommand{\bmprior}[1]{\ensuremath{{p}_{\text{\tiny $#1$}}}}
\newcommand{\bmdprior}[3]{\ensuremath{{p}_{{\text{\tiny $#3$}}}(#2)}}

\newcommand{\WCauchy}{\ensuremath{W^{\text{\sc c}}}}

\newtheorem{theorem}{Theorem}

\newtheorem{proposition}{Proposition}

\newtheorem{owndefinition}{Definition}
\newtheorem{corollary}{Corollary}

\newtheorem{ownexample}{Example}

\DeclareRobustCommand{\VANDER}[3]{#2}

\title{Safe Testing}
\date{\today}
\author[1]{Peter Gr{\"u}nwald}
\author[2]{Rianne {de Heide}}
\author[3]{Wouter M. Koolen}
\affil[1,3]{Centrum Wiskunde \& Informatica, Amsterdam, The Netherlands}
\affil[1]{Leiden University, Leiden, The Netherlands}
\affil[2]{Vrije Universiteit, Amsterdam, The Netherlands}
\affil[3]{University of Twente, Enschede, The Netherlands}
\affil[1,3]{ $\{$pdg, wmkoolen$\}$@cwi.nl} 
\affil[2]{r.de.heide@vu.nl}

\begin{document}

\maketitle

\begin{abstract}
We develop the theory of hypothesis testing based on 
the $\S$-value, a notion of evidence that, unlike the $\p$-value,
allows for effortlessly combining results from several studies in the common scenario 
where the decision to perform a new study may depend on previous outcomes. Tests based on
$\S$-values are safe, i.e.\ they preserve Type-I error guarantees, under such {\em optional continuation}. We define growth-rate optimality (GRO) as an analogue of power in an optional continuation context, and we show how to construct GRO $\S$-variables for general testing problems with composite null and alternative, emphasizing  models with nuisance parameters. GRO $\S$-values take the form of Bayes factors with special priors. We illustrate the
theory using several classic examples including a
one-sample safe $t$-test and the $2 \times 2$ contingency table. Sharing Fisherian, Neymanian and Jeffreys-Bayesian interpretations, $\S$-values may provide a methodology acceptable to  adherents of all three schools.

\end{abstract}

\section{Introduction and Overview}
\label{sec:introduction}
We wish to
test the veracity of a {\em null hypothesis\/} $\cH_0$, often in
contrast with some {\em alternative hypothesis\/} $\cH_1$, where both
$\cH_0$ and $\cH_1$ represent sets of distributions on some given
sample space. Our theory is based on {\em $\S$-test
	statistics}. These are simply  {\em nonnegative\/}
random variables  that satisfy the inequality:
\begin{equation}\label{eq:basic}
\text{for all $P \in \cH_0$:}\ \ {\bf E}_{P}[\sval] \leq 1.
\end{equation}
We refer to $\S$-test statistics as $\S$-variables, and to the value
they take on a given sample as the $\S$-{\em value}, emphasizing that
they are to be viewed as an alternative to, and in many cases an
improvement of, the classical $\p$-value. Note that {\em large\/}
$\S$-values correspond to evidence against the null: for given
$\S$-variable $E$ and $0 \leq \alpha \leq 1$, we define
the {\em threshold test corresponding to $E$ with significance level
	$\alpha$}, as the test that rejects $\cH_0$ iff $E \geq 1/\alpha$. We will see, in a sense to be made more precise, that this test is {\em safe
	under optional continuation with respect to Type-I error}. 

\paragraph{Motivation}  $\p$-values and standard null hypothesis testing have come under intense scrutiny in recent years \citep{wasserstein2016asa,benjamin2018redefine}. $\S$-variables and safe tests offer several advantages. Most importantly, in contrast to $\p$-values,
$\S$-variables behave excellently under {\em optional continuation}, the
highly common practice in which the decision to perform additional
tests partly depends on the outcome of previous tests.
They thus seem particularly promising when
used in meta-analysis \citep{terschure2019accumulation}.
A second
reason is their enhanced {\em interpretability\/}: 
they have a very concrete  (monetary) interpretation as `evidence against the null' which remains valid even if one dismisses concepts such as  `significance' altogether, as recently advocated by
\cite{amrhein2019scientists}. 
A third is their
flexibility: as we show in this paper, $\S$-variables can be based on Bayesian prior knowledge, on earlier data, but also on minimax performance considerations, in all cases  preserving frequentist Type I error
guarantees.

\paragraph{Overall Contribution and Contents}
Although the concept is much older (Section~\ref{sec:related}), the interest in $\S$-values and the related test martingales has exploded over the last three years \citep{WangR20,VovkW21,Shafer19,henzi2021valid,waudbysmith2021estimating,OrabonaJ21}. 
In this paper, we further develop the theory of $\S$-values, by providing general optimality criteria and show how to design $\S$-variables that satisfy them. We 
do this on the basis of four ever more general versions of a single novel theorem, Theorem~\ref{thm:maina}. In its first incarnation, in Section~\ref{sec:GRO}, Theorem~\ref{thm:maina} already tells us that one can design nontrivial, useful $\S$-variables for a wide 
class of  testing problems with composite null and alternative. 
This first instance relies on using a prior $W_1$ on the alternative $\cH_1$. The ensuing $\S$-variables, while guaranteeing frequentist Type-I error control, will have a GRO ({\em growth-rate optimality}) property under $W_1$. 
This GRO $\S$-variable will be a Bayes factor with a special prior on the null. 
More general versions of the theorem allow us to  construct $\S$-variables when no  prior on $\cH_1$ is available. These satisfy either a direct worst-case optimality criterion (GROW) or a relative one (REGROW). In 
our example applications we restrict ourselves to classical testing scenarios such as 1-dimensional exponential families, the $2 \times 2$ contingency table, and the $t$-test. Importantly, the latter two have {\em nuisance\/} parameters and the GRO approach provides a generic methodology for dealing with them. For the $t$-test setting, GRO 
$\S$-variables turn out to be  Bayes factors based on the right Haar prior, as known from objective Bayes analyses \citep{berger1998bayes}. For the $2\times 2$-setting, GRO $\S$-values do not correspond to standard Bayes factors. 

We then, in Section~\ref{sec:GROanalysis} and~\ref{sec:competitive}, investigate optional continuation, stopping and GRO in more detail, and we assess how competitive the $\S$-variables we designed are compared to classical methods in terms of the amount of data needed before a certain desired power or growth rate can be reached. The final three sections put our work in context. We provide a historical overview of $\S$-value related work in Section~\ref{sec:related}, critically discuss GRO in Section~\ref{sec:GROdiscussion},  and then, in Section~\ref{sec:theory}, taking a step back, we come to the inescapable conclusion that $\S$-variables unify and correct ideas from the three main paradigms of testing: Fisherian, Neyman-Pearsonian and Jeffreysian.
But first, in the remainder of this introduction, we explain the  three main interpretations of $\S$-variables
(Section~\ref{sec:three}), 
we briefly introduce our main theorem (Section~\ref{sec:thmprelude}) and, in Section~\ref{sec:os}, we explain the main advantage of $\S$-variables over 
$\p$-values in terms of optional continuation. We claim no technical novelty for this part, which mainly restates and reinterprets existing results\footnote{
Since the first version of the present paper appeared on arXiv, various subsets of these results have been widely discussed in various recent papers, but we still re-state them here to keep the paper self-contained.
}. 
We defer to the appendices all longer proofs and details that would distract from the main story. 

\subsection{The three main interpretations of $\S$-variables}
\label{sec:three}
\paragraph{1.\ First Interpretation: Gambling}
The first and foremost interpretation of $\S$-variables is in terms of {\em money},
or, more precisely, {\em  Kelly {\rm (\citeyear{Kelly56})} gambling\/}. Imagine
a ticket (contract, gamble, investment) that one can buy for 1\$, and
that, after realisation of the data, pays $\eval \, \$$; one may buy several
and positive fractional amounts of tickets. \eqref{eq:basic} says that, if the
null hypothesis is true, then one expects not to gain any money by
buying such tickets: for any $r \in \reals^+$, upon buying $r$ tickets
one expects to end up with $r \Exp[\eval] \leq r \, \$ $.  Therefore, if the
observed value of $\eval$ is large, say $20=1/0.05$, one would have gained a lot
of money after all, indicating that something might be wrong about the
null.
\paragraph{2.\ Second Interpretation:
	Conservative $\p$-Value, Type I Error Probability}
Recall that a (strict) $\p$-value is a random variable $\pval$ such that
for all $0 \leq \alpha \leq 1$, all $P_0 \in \cH_0$,
\begin{equation}\label{eq:pval} P_0(\pval \leq \alpha) = \alpha.
\end{equation}
A {\em conservative\/} $\p$-value is a random variable for which
\eqref{eq:pval} holds with `$=$' replaced by `$\leq$'.
There is a close connection between (small) $\p$- and (large) $\S$-values:
\begin{proposition}\label{prop:pval}
	For any given $\S$-variable $\eval$, define $\pval_{[\S]} \coloneqq 1/\eval$. Then $\pval_{[\S]}$ is a conservative $\p$-value. As a consequence, for every $\S$-variable $\eval$, any $0 \leq \alpha \leq 1$, the corresponding threshold-based test has Type-I
	error guarantee $\alpha$, i.e.\ for all $P \in \cH_0$, 
	\begin{equation}\label{eq:reject}
	P(\eval \geq 1/\alpha ) \leq \alpha.
	\end{equation}
\end{proposition}
\begin{proof}{\bf (of Proposition~\ref{prop:pval})}
	Markov's inequality gives $P(\eval \geq 1/\alpha) \leq \alpha
	\Exp_P[\eval] \leq \alpha$. 
\end{proof}
While reciprocals of $\S$-variables thus give a special type of conservative $\p$-values, reciprocals of standard
$\p$-values satisfying \eqref{eq:pval} are by no means $\S$-variables; if
$\eval$ is an $\S$-variable and $\pval$
is a standard $\p$-value, and they are
calculated on the same data, then we will usually observe $\pval \ll 1/\eval$
so with e-values we need more extreme data in order to  reject the null (see Section~\ref{sec:competitive} for a nuanced analysis and Section~\ref{sec:related} for more on e-p conversions).

\paragraph{Combining 1.\ and 2.: Optional Continuation}
Proposition~\ref{prop:optional.stopping}
below shows that {\em multiplying\/}
$\S$-variables $\eval_{(1)}, \eval_{(2)}, \ldots$ for tests based on respective independent samples
$\Ydata_{(1)}$, $\Ydata_{(2)}, \ldots$ (with each $\Ydata_{(j)}$ being the batch of outcomes for
the $j$-th test), gives rise to new $\S$-variables, even if the decision
whether or not to perform the test resulting in $\eval_{(j)}$ was based on
the value of earlier test outcomes $\eval_{(j-1)}, \eval_{(j-2)}, \ldots$
As a
result (Corollary~\ref{cor:villerobbins}), {\em the Type I-Error Guarantee
	\eqref{eq:reject} remains valid even under this `optional
	continuation' of testing}. Just as importantly, in contrast to $\p$-values, $\S$-variables satisfy an `optional continuation principle': 
whether an observed $\S$-value is valid or not does not depend on whether or not you would have performed an additional study and gathered additional evidence in situations that did not occur.

An indication that something like this might be true is immediate from our gambling interpretation: if we start by investing \$1 in $\eval_{(1)}$ and,
after observing $\eval_{(1)}$, reinvest all our new capital $\$\eval_{(1)}$ into
$\eval_{(2)}$, then after observing $\eval_{(2)}$ our new capital will obviously be
$\$ \eval_{(1)} \cdot \eval_{(2)}$, and so on. If, under the null, we do not expect to
gain any money for any of the individual gambles $\eval_{(j)}$, then,
intuitively, we should not expect to gain any money under whichever
strategy we employ for deciding whether or not to reinvest (just as
you would not expect to gain any money in a casino irrespective of
your rule for re-investing and/or stopping and going home).

\paragraph{3.\ Third Interpretation: Bayes Factors}
For convenience, from now on we write the models $\cH_0$ and $\cH_1$
as
$
\cH_0 = \{ P_{\theta} : \theta \in \Theta_0\} \ \ ; \ \ 
\cH_1 = \{ P_{\theta} : \theta \in \Theta_1\},
$ where $\Theta_0, \Theta_1\subset \Thetafull$, and $\{P_{\theta}: \theta \in \Thetafull\}$ represents a general family of distributions or random processes, defined relative to some given sample space and $\sigma$-algebra or filtration. $\Ydata
= Y^N = (Y_1, \ldots, Y_N)$, a vector of $N$ outcomes, represents our data.
$N$ may be a fixed sample size
$n$ but can also be a random stopping time. We assume that, under every $P_{\theta}$ with $\theta \in \Thetafull$, $\Ydata$ has a  probability density $p_{\theta}$ relative to some fixed underlying measure $\mu$.  
In the
Bayes factor approach to testing, one associates both $\cH_j$ with a
{\em prior\/} $\priordist_j$, which is simply a probability
distribution on $\Theta_j$, and a {\em Bayes marginal probability
	distribution\/} ${P}_{\priordist_j}$, with density  (or mass) function given
by\label{page:firstprior}
\begin{equation}\label{eq:bayesmarginal}
\bmdprior{j}{\Ydata}{\priordist_j} \coloneqq \int_{\Theta_j} p_{\theta}(\Ydata) \dif \priordist_{j}(\theta).
 \end{equation}
The {\em Bayes factor\/} is then given as:
\begin{equation}\label{eq:bayesfactor}
\mathsf{BF} \coloneqq   \frac{\bmdprior{1}{\Ydata}{\priordist_1}}{\bmdprior{0}{\Ydata}{\priordist_0}}.
\end{equation}
Whenever $\cH_0 = \{P_0 \}$ is {\em simple}, i.e., a singleton, then
the Bayes factor is also a (sharp, i.e.\ with expectation exactly 1) $\S$-variable, since we must then
have that $\priordist_0$ is degenerate, putting all mass on $0$, and
$\bmprior{\priordist_0} = p_0$, and then for all $P \in \cH_0$,
i.e.\ for $P_0$, we have, assuming $P_0$ has strictly positive density,
\begin{equation}\label{eq:itisone}
\Exp_{P}[\mathsf{BF}] = \int p_0(y) \cdot \frac{\bmdprior{1}{y}{\priordist_1}}{{p}_0(y)} \dif \mu(y)  = 1.  
\end{equation}
For such $\S$-variables that are really
simple-$\cH_0$-based Bayes factors, Proposition~\ref{prop:pval}
reduces to the well-known {\em universal bound\/} for likelihood ratios
\citep{royall1997statistical}.
When $\cH_0$ is itself composite, most Bayes factors $\mathsf{BF} =
p_{W_1}/p_{W_0}$ will {\em not\/} be $\S$-variables any more, since
for $\mathsf{BF}$ to be an $\S$-variable we require \eqref{eq:itisone} to hold
for {\em all\/} $P_{\theta}, \theta \in \Theta_0$, whereas in general it only
holds for $P=P_{W_0}$. Nevertheless, Theorem~\ref{thm:maina} (in its first, simplest version in Section~\ref{sec:GRO})
implies that, under weak conditions, 
for every prior $W_1$ on $\Theta_1$ there always exists a corresponding prior $W_0$ on $\Theta_0$, for which $\mathsf{BF} = p_{W_1}/p_{W_0}$ is an $\S$-variable after all. More generally, in all our examples $\S$-variables invariably take on a Bayesian form, though sometimes with highly unusual (e.g.\ degenerate) priors.

\subsection{This Paper: Beyond Simple Nulls, Beyond Available Priors}
\label{sec:thmprelude}
In this paper, we focus on general, composite $\cH_0$. The only assumption on $\cH_0$ is the existence of densities as above --- we make this assumption because it allows for a completely general characterisation of GRO (`growth-rate optimal') $\S$-variables as in Theorem~\ref{thm:maina}. Still, useful $\S$-variables for nonparametric settings without densities do exist \citep{waudbysmith2021estimating,OrabonaJ21}. 

Theorem~\ref{thm:maina} in its first form in Section~\ref{sec:GRO} tells us how to choose an $\S$-variable that is optimal in the GRO sense if a prior $W_1$ on $\cH_1$ is given. Roughly speaking, GRO means that the $\S$-variable tends to grow fast under $\cH_1$ as more data come in, thereby generating substantial evidence against $\cH_0$. The generalisations of Section~\ref{sec:grow}--\ref{sec:regrow}, extend the GRO idea to  $\S$-variables when no such prior is available. Section~\ref{sec:grow} deals with a basic maximin optimality approach, which is appropriate if there is a single parameter of interest, a minimum relevant effect size, and no nuisance parameter. Section~\ref{sec:regrow} describes  {\em relative\/} maximin optimal $\S$-variables, appropriate if there is no minimal effect size and/or nuisance parameters are present. This culminates in the fully general version of Theorem~\ref{thm:maina} in Section~\ref{sec:fulltheorem} which is also applicable to hypotheses with nuisance parameters that satisfy a group invariance, such as in the $t$-test. To show that the $\S$-variable we propose for the $t$-test  is indeed optimal we need a special case of an additional result, Theorem~4.2 of the paper \citep{perez2022estatistics}, which, for convenience, we restate. 
But before embarking on these results, we explain the benefits of $\S$-variable based tests in detail.

\subsection{Optional Continuation}\label{sec:os} We defined $\S$-variables for a single experiment. We now discuss sequential experimentation and how $\S$-values can be combined to accumulate evidence against the null. To this end, let us imagine a sequence of random variables 
$\bY_{(1)}, \bY_{(2)}, \ldots$ representing the outcomes of experiments/studies. We will not (except for illustration purposes later on) make use of any internal structure of the $\bY_{(j)}$, which in particular may come to us as batches of varying lengths.
\begin{owndefinition}\label{def:cond.safe}
Let $\cH_0$ be a collection of distributions for a sample space equipped with filtration $(\mathcal F_{(m)})_m$. We say that 
	$\eval_{(m)}$ is a  \emph{$\cF_{(m-1)}$-conditional $\S$-variable} (relative to null hypothesis $\mathcal H_0$) if it is a nonnegative random variable that is $\cF_{(m)}$-measurable and $\text{for all $P \in \cH_0$:}\ \ {\bf E}_{P}[\eval_{(m)} \mid \mathcal F_{(m-1)} ]  \leq 1 \; \text{a.s.}$. If, for each $m$,  $\eval_{(m)}$ is an {$\cF_{(m-1)}$-conditional $\S$-variable}, we will call $\{\eval_{(m)}\}_m$ a conditional $\S$-variable collection relative to $(\mathcal F_{(m)})_{m}$ .
\end{owndefinition}
In standard cases, $\cF_{(m)}$ represent all that is known to us at time $m$ (that is, after having observed the $m$-th study). 
Then it
is simply $\sigma(\bY^{(m)})$, with $\bY^{(m)} = (\bY_{(1)}, \ldots, \bY_{(m)})$ the sequence of outcomes of previous studies,
and we could then rewrite the expectation  elementarily as ${\bf E}_{P}[\eval_{(m)} \mid  {\bf Y}^{(m-1)} ]$. 
More generally though, $\cF_{(m)}$
is allowed to be a coarser filtration as well: as long as for all $m$, $E_{(m)}$ is $\cF_{(m)}$-measurable, we can safely engage in optional continuation in the sense of Corollary~\ref{cor:villerobbins} below, as explained in Section~\ref{sec:GROanalysis}.  
On the other hand, $\cF_{(m)}$ could also be finer, 
including nonstochastic side information such as `there is money to do an additional study' or covariates; we briefly describe such extensions, as well as subtleties that may arise, in Appendix~\ref{app:first}. 

Intuitively, $\cF_{(m-1)}$-conditional $\S$-values measure the conditional evidence in round $m$ (representing the $m$-th study) against $\mathcal H_0$, and hence their running product measures the total evidence (such a running product would then be a {\em test super-martingale\/} \citep{shafer2011test}, i.e.\ a nonnegative super-martingale with starting value $\leq 1$, under every element of the null). We may turn this running product into one quantity by adding a stopping rule. The following result, both parts of which are a direct implication of Doob's optional stopping theorem \citep{Williams91}  
states that, irrespective of the stopping rule/time, we obtain a fair measure of evidence.
\begin{proposition}\label{prop:optional.stopping}
\begin{enumerate}\item 	Let $\{E_{(m)}\}_m$ be a collection of conditional $\S$-variables relative to filtration $\mathcal (\mathcal F_{(m)})_m$. Then the running product $(E^{(m)})_{m}$ with $E^{(m)} := \prod_{j=1}^m E_{(j)}$ is a test super-martingale w.r.t.\ each  $P \in \mathcal H_0$.
\item Any process $(E^{(m)})_m$ that is a test super-martingale w.r.t.\ each  $P \in \mathcal H_0$ is also an {\em e-process\/} \citep{fork-convex} w.r.t.\ each $P \in \mathcal H_0$, which by definition means that for any stopping time $\tau$ (not necessarily finite), the stopped value $E^{(\tau)}$ is a (standard non-conditional) $\S$-value for $\mathcal H_0$.
	\end{enumerate}
\end{proposition}
Proposition~\ref{prop:optional.stopping} says that, no matter when we stop collecting batches of data, the resulting product is an $\S$-variable and therefore a test based on it preserves Type-I error guarantees by Proposition~\ref{prop:pval}.
We note that, after the first version of the present paper appeared on arxiv, it was found that for some composite $\cH_0$ there exist useful e-processes that are not test-martingales
\citep{fork-convex}; we cannot build these as products of the conditional e-variables that our main theorem, presented in the next section, generates. On the other hand, as a referee suggested, in our optional continuation context the use of conditional e-variables as basic building blocks may also have advantages over general e-processes (for example, it is easier to switch \citep{erven2012catching} to a different type of e-variable from one study to the next); sorting this out in detail remains a question for future work. 
\paragraph{Just-in-Time Conditional $\S$-variables: Optional Continuation}
As we can see from Proposition~\ref{prop:optional.stopping}, the stopped running product $E^{(\tau)}$ of a sequence of conditional $\S$-variables only evaluates each member variable $\eval_{(m)}$ after $m$ rounds, and only on the data ${\bf Y}_{(1)},\ldots,{\bf Y}_{(m)}$ that actually happened. It is therefore perfectly fine (for the Type-I error safety guaranteed by combining Propositions~\ref{prop:optional.stopping} and~\ref{prop:pval}) for us to construct $\eval_{(m)}$ on demand just before round $m$, as a function of all available information so far (including possibly both stochastically modeled and arbitrary variables), as long as we ensure the conditional safety property in Definition~\ref{def:cond.safe}.
This simple observation gives us tremendous flexibility for testing, much in contrast to traditional $\p$-values where the sampling plan needs to be fixed up front. In particular, it allows \emph{optional continuation}: the practice of deciding \emph{after} an initial series of experiments whether to output the current accumulated evidence, or perform yet more experiments.

Proposition~\ref{prop:optional.stopping} already indicates that we can safely engage in such optional continuation, assuming that we stop performing studies at some stopping time $\tau$ relative to the filtration $(\cF_{(m)})_{(m)}$. The following proposition expresses that we have Type-I error safety under optional continuation in an even stronger filtration-independent sense:  
\begin{corollary}\label{cor:villerobbins}{\bf (of Proposition~\ref{prop:optional.stopping}): ``Ville-Robbins''} 
	\begin{equation}\label{eq:villerobbins}
	\text{For all $P \in \cH_0$, all $0 < \alpha \leq 1$:}\ \ 
	P\left( 
	\text{there exists $m$ such that\ }  \eval^{(m)} \geq \frac{1}{\alpha} \right) \leq \alpha. 
	\end{equation}
\end{corollary}
\begin{proof} Proposition~\ref{prop:optional.stopping} expresses that $E^{(1)}, E^{(2)}, \ldots$ is a super-martingale with starting value $\leq 1$. Ville's inequality \citep{Ville39} (also known as Ville-Robbins or Ville-Robbins-Wald inequality) then implies (\ref{eq:villerobbins}). 
\end{proof}
For use later on, we formally  define a {\em threshold test based on non-negative process $(\eval^{(m)})_{m}$\/}  to be the random function that, when input $m$ and level $\alpha$, outputs \textsc{reject} if $\eval^{(m)} \geq 1/\alpha$ and outputs \textsc{accept} otherwise (in this general definition, $\eval^{(m)}$ can but does not need to be defined as a product of conditional $\S$-variables). We say that a threshold test
is {\em safe under optional continuation  with respect to Type I error\/} if (\ref{eq:villerobbins}) holds.
Thus, no matter when the data collecting and combination process is stopped, the Type-I error probability  is preserved. We relate optional continuation to the more common notion of `optional stopping' and discuss filtration-related subtleties in Section~\ref{sec:GROanalysis}.   

Whereas Ville-Robbins stresses that we may greedily `keep combining studies until we can reject or resources run out', it is just as important that our $\S$-values keep providing valid Type-I error guarantees if the continuation rule is externally imposed or unknowable.

\begin{ownexample}\label{ex:oc}{ \normalfont
	Consider the simple scenario with a single underlying data stream $Y_1, Y_2, \ldots$ with  $Y_i$ i.i.d.\ according to both  $\cH_0$ and $\cH_1$. Assume for simplicity simple $\cH_0 = \{P_0\}$  so that Bayes factors provide $\S$-variables. 
	For arbitrary prior $W$ on $\Theta_1$, define $e_{n,W}(Y^n) =
	p_W(Y^n)/p_0(Y^n)$ to be the Bayes factor as in (\ref{eq:bayesfactor}) with prior $W$ for
	$\Theta_1$ applied to data $Y^n$.
	
	Suppose we perform an initial study on sample $\Ydata_{(1)} \coloneqq Y^{N_{(1)}} \coloneqq (Y_1, \ldots,
	Y_{N_{(1)}})$ and we equip $\Theta_1$ with prior $W_{(1)}$. We can use as our $\S$-variable $\eval_{(1)}$ the Bayes factor 
	$\eval_{(1)} := e_{N_{(1)},W_{(1)}}(\bY_{(1)})$. 
	Suppose this leads to a first $\S$-value  $\eval_{(1)} =
	18$ --- promising enough for us to invest our resources into a
	subsequent study. We decide to gather $N_{(2)}$ data points leading to
	data $\Ydata_{(2)} = (Y_{N_{(1)}+1}, \ldots, Y_{N_{(1)}+ N_{(2)}})$. For this second data batch, we will use an $\S$-variable 
	$\eval_{(2)}  \coloneqq e_{N_{(2)},W_{(2)} }({\bf Y_{(2)}})$ 
	for a new prior $W_{(2)}$. Crucially, we are allowed to choose both $N_{(2)}$ and $W_{(2)}$  as a function of past data $\Ydata^{(1)}$: clearly, because the underlying data stream was assumed i.i.d., $\Exp_{P_0}[\eval_{(2)} \mid \bY_{(1)} ] \leq 1$ irrespective of our choice (here we use (\ref{eq:itisone})), and this allows us to use Proposition~\ref{prop:optional.stopping}.
	If we want to stick to the Bayesian paradigm, we can choose 
	$W_{(2)} \coloneqq {\priordist_{(1)}(\cdot \mid \Ydata_{(1)})}$, as  the Bayes posterior for $\theta_1$ based on data $\Yb_{(1)}$ and prior $W_{(1)}$.
	Bayes' theorem
	shows that multiplying $\eval^{(2)}\coloneqq \eval_{(1)} \cdot \eval_{(2)}$ (which gives a new $\S$-variable by Proposition~\ref{prop:optional.stopping}), satisfies
	\begin{equation}\label{eq:coherent}
	\eval^{(2)} = \eval_{(1)} \cdot \eval_{(2)} = \frac{{p}_{W_{(1)}}(\Yb_{(1)}) \cdot
		{p}_{W_{(1)}(\cdot \mid \Yb_{(1)})}(\Yb_{(2)})}{p_0(\Yb_{(1)})\cdot p_0(\Yb_{(2)})} =
	\frac{{p}_{W_{(1)}}(Y_1, \ldots, Y_{N_{(2)}})}{p_0(Y_{1}, \ldots, Y_{N_{(2)}})},
	\end{equation}
	which is exactly what one would get by Bayesian updating. This illustrates
	that, for simple $\cH_0$, combining $\S$-variables by multiplication can
	be done consistently with Bayesian updating.} 
\end{ownexample}

\paragraph{The Local Perspective} It might also be the case that it is not us who get the additional funding to obtain extra data, but rather some research
group at a different location.  If the question is, say, whether a
medication works, the null hypothesis would still be $\cH_0 = \{P_0\}$
but, if it works, its effectiveness might be slightly different due to
slight differences in population. In that case, the research group
might decide to use a different test statistic $\eval'_{(2)}$ which is
again a Bayes factor, but now with an alternative prior $\priordist$
on $\theta_1$ (for example, the original prior $W_{(1)}$ might be re-used
rather than replaced by $W_{(1)}(\cdot \mid \Yb_{(1)})$) --- one might call this the {\em local\/} perspective. Even though not standard Bayesian, $\eval_{(1)} \cdot \eval'_{(2)}$
still gives  a valid $\S$-variable, and Type-I error guarantees
are preserved --- and the same will hold even if the new
research group would use an entirely different prior on $\Theta_1$.
And, after the second batch of data
$\Ydata_{(2)}$, one might consider obtaining even
more samples, each time using a different $W_{(j)}$, that is always allowed to
depend on the past in arbitrary ways. 

Finally, it is important to note that, when combining studies, we do require all the data batches $\bY_{(1)}, \bY_{(2)}, \ldots$ to refer to separate data: obviously it is not allowed to `borrow' some data from $\bY_{(1)}$ and reuse it as part of $\bY_{(2)}$. E-values based partly on the same data can still be validly combined (e.g.\ by averaging \citep{VovkW21}) but not by multiplication as we do here.

\section{The GRO $\S$-Variable}\label{sec:GRO}
\paragraph{Mathematical Preliminaries} 
In this and the coming sections we present our main result, Theorem~\ref{thm:maina}. We first list all required mathematical notations and definitions. We invariably assume that some family $\{ P_{\theta}: \theta \in \Thetafull\}$ of probability distributions for $\cY$ has been fixed and  all $P_{\theta}$ with $\theta \in \Thetafull$ have  densities relative to some underlying measure $\mu$. When we write `$p$ is a (sub-) probability density', we mean it is a (sub-) probability density relative to $\mu$, i.e.\ $p \geq 0$ and  $\int p(\bY) d\mu = 1$ for a density and $\int p(\bY) d\mu \leq 1$ for a sub-density. In the latter case we call the measure $P$ with density $p$ a sub-probability distribution. We use $D(Q \| P)$ to denote the
Kullback-Leibler (KL) Divergence between distributions $Q$ and $P$ \citep{CoverT91}. We allow $P$ (but not $Q$) to be a sub-probability distribution, with $D(Q \|P) = {\bf E}_{\bY \sim Q}[\ln q(\bY)/p(\bY)]$. We say that random variables $U^*$ and $U^{\circ}$ are {\em essentially equal\/} if, for all
$\theta \in \Thetafull$, $P_{\theta}(U^* = U^{\circ}) = 1$.
We say that $U^*$ {\em essentially uniquely\/} satisfies property $\P$  if all other random variables satisfying property $\P$ are essentially equal to $U^*$. When we write `$P$ has full support', we mean that its density $p$ satisfies $p({\bf Y}) > 0$  $\mu$-almost everywhere.
We assume some suitable $\sigma$-algebra including all singleton sets on $\Thetafull$ has been defined, and for $\Thetasmall \subset \Thetafull$ we let $\priorset(\Thetasmall)$ be the set of all probability distributions
(i.e., `proper priors') on $\Thetasmall$ with this $\sigma$-algebra. Notably, $\priorset(\Thetasmall)$  includes, for each $\theta \in \Thetasmall$,
the degenerate distribution $W$ which puts all mass on $\theta$. We say that $W$ {\em essentially uniquely\/} satisfies property $\P$ among $\cW(\Thetasmall)$ if for all other distributions $W' \in \cW(\Thetasmall)$ that satisfy $\P$ and all $\theta \in \Thetafull$, we have $P_{\theta}(p_W = p'_W) = 1$, where $p_W$ and $p'_W$ are as in (\ref{eq:bayesmarginal}). 
$\cE(\Theta_0)$ is defined as the set of all $\S$-variables that can be defined on $\Ydata$ for $\Theta_0$, i.e.\ all random variables satisfying (\ref{eq:basic}).
We frequently use the fact
that if  $\Theta_0 = \{0 \}$ is a singleton so that $\cH_0$ is simple, then the class of $\S$-variables corresponds exactly to the set of likelihood ratios relative to $p_0$: 

\begin{align}\label{eq:simpleEisLR}
\cE(\{0 \}) = \left\{ 
\frac{q(\bY)}{p_0(\bY)}\; : 
\text{$q$ is a sub-probability density for $\bY$}\;
\right\}.
\end{align}
To see this, note that for every $\S$-variable $\eval = e({\bf Y)}$ we can define  $q({\bf Y}) := e ({\bf Y}) \cdot  p_0({\bf Y})$ and then  $\int q d \mu = \int p_0({\bf Y}) e({\bf Y}) d \mu \leq 1$; conversely every sub-density $q$ defines an $\S$-variable by setting $\eval = q({\bf Y})/p_0({\bf Y})$ which gives $\Exp_{P_0}[\eval] \leq 1$. 
 
Our main theorem (proof in Appendix~\ref{app:theorem1simpleproof}) implies that nontrivial $\S$-variables exist without any further conditions: 
\begin{theorem}\label{thm:maina}
	Suppose $Q$ is a probability distribution with full support and with density $q$, and assume  $\inf_{W_0 \in \priorset(\Theta_0)} D(Q  \| P_{W_0}) < \infty$.  Then
	there exists a (potentially sub-) distribution $P^*_0$ with density $p^*_0$ such that 
	\begin{equation}\label{eq:firsteval}
	\eval^{*} := \frac{q(\bY)}{p^*_0(\bY)}
	\end{equation} 

	is an $\S$-variable. Moreover, $\eval^*$ satisfies, essentially uniquely,  
	\begin{align}\label{eq:firstgro}
	\sup_{\eval\in \cE(\Theta_0)}  \Exp_{\Ydata \sim Q}[\log \eval]
	=  
	\Exp_{\Ydata \sim Q}[\log \eval^*] =
	\inf_{W_0 \in \priorset(\Theta_0)}  D(Q \| P_{W_0})
	= D(Q \| P^*_0).  
	\end{align} 
	%
	If the minimum is achieved by some ${W}^*_0$, i.e.\ $D(Q \| P^*_0) = D(Q \|
	P_{W^*_0})$, then $P^*_0 = P_{W^*_0}$. 
\end{theorem}
The full support condition is natural and discussed further in Appendix~\ref{app:regularityconditions}. Following Barron and \cite{Li99} (see also \citep{csiszar2003information}), we call $P^*_0$  the {\em Reverse Information Projection (RIPr)\/} of $Q$ on ${\cal P}(\Theta_0)= \{P_{W_0}: W_0 \in \cW(\Theta_0) \}$. In all examples in this paper, we have $P^*_0 = P_{W^*_0}$: the minimum is achieved and its density integrates to $1$ (one can construct special $\cH_0$ for which $p^*_0$ integrates to strictly less than $1$ \citep{Harremoes21}, but we do not know whether this happens for any practically relevant $\cH_0$). The following corollary (see Appendix~\ref{app:theorem1simpleproof} for details) is useful in applications:  
\begin{corollary}\label{cor:justone}
	$\eval^*$ is the only  $\S$-variable of Bayes factor/likelihood ratio form  with $q$ in the numerator. That is, for all $W_0 \in \cW(\Theta_0)$: if $P_{W_0}$ is not essentially equal to $P^*_0$ then $q(\Ydata)/p_{W_0}(\Ydata)$ is not an $\S$-variable. In particular this implies: (a)  if $P^*_0 = P_{W^*_0}$, then $W^*_0$ achieves $\min_{W_0 \in \cW(\Theta_0)} D(Q \| P_{W_0})$ essentially uniquely; and (b) if we have found an $\S$-variable of form $q(\Ydata)/p_{W_0}(\Ydata)$ then  $W_0$ must be essentially equal to $W^*_0$. 
\end{corollary}
Theorem~\ref{thm:maina} leaves open the question of how to calculate $W^*_0$, if it exists. In the examples we encounter below, we can either show that $W^*_0$ is degenerate, putting all its mass on a single distribution $P_{\theta^*_0}$, and $\theta^*_0$ can be determined analytically, or, as in the t-test example, we can analytically find it by other means. `Easy' $W^*_0$ occur  in surprisingly many other situations as well (see e.g.\ \citep{TerschurePLG21,KoolenG21}), but by no means always \citep{Adams20}. More generally, even if $W^*_0$ is not easy to determine analytically, as long as $\cY$ is finite then, using
Carath\'eodory's theorem we can still show that $W^*_0$ must exist and has finite support. By strict convexity of KL divergence in its second argument it can therefore in principle be found by numerical methods, but more research is needed to see existing methods are fast enough in practice. If $\cY$ is infinite, one can still try to approximate $W^*_0$ numerically but it may be hard to determine the accuracy of such approximations. 
\subsection{The GRO criterion when \texorpdfstring{$\mathcal{H}_1$}{the alternative hypothesis} is simple}
\label{sec:grosimple}
We now focus on the case with a given alternative $\cH_1 = \{P_{\theta}: \theta \in \Theta_1\}$, and for now assume $\Theta_1 = \{\theta_1\}$ is a singleton. 
Applying Theorem~\ref{thm:maina} above with $Q= P_{\theta_1}$, we call the resulting $\eval^*$ (or any essentially equal version of it) the $\theta_1$-GRO $\S$-variable, GRO standing for {\em growth-rate optimal}. We define the {\em growth rate achievable with $\theta_1$\/} as  
\begin{equation}\label{eq:grodef}
\gro(\theta_1) :=  
\sup_{\eval \in \cE(\Theta_0)} {\bf E}_{Y \sim P_{\theta_1}} [ \log \eval] = D(P_{\theta_1} \| P^*_0),
\end{equation}
\noindent 
with the equality following from Theorem~\ref{thm:maina}
(we omit $\Theta_0$ in the notations since, in contrast to $\Theta_1$ or $\theta_1$, $\Theta_0$ will always be clear from context). 
In general, there exist many nontrivial $\S$-variables for a given $\cH_0$. The $\theta_1$-GRO $\S$-variable is a  special one that is optimal in a natural sense for the given $\cH_1$: whereas in the Neyman-Pearson paradigm, one measures the quality of a test at a given significance level $\alpha$ by its
power, i.e.\ the
probability of correct decision under $\theta_1$, we will measure it by the  {\em expected capital growth rate\/} under $\theta_1$. This is different from power, yet there are close connections to which we return in 
Section~\ref{sec:theory}. 

To explain, we return to the monetary interpretation of $\S$-values.  The
definition of $\S$-variable ensures that we expect them to stay under
$1$ (one does not gain money) under any $P \in \cH_0$.
Analogously, one would like them to be constructed such that they can
be expected to grow large as fast as possible (one gets rich, gets
evidence against $\cH_0$) under $\cH_1$. 
Assuming for now that $\cH_1 = \{P_{\theta_1}\}$ is simple, this suggests to define the optimal $\S$-variable $\eval^*$ as the one that maximises $\Exp_{P_{\theta_1}}[f(\eval^*)]$ for some function that is increasing in $\eval^*$. At first sight it may seem natural to pick $f$ the identity,
but this can lead to adoption of an $\S$-variable $\eval^*$ such that
$P_{\theta_1}(\eval^* = 0) > 0$. This choice, however, does not go together well with preserving evidence (capital) under optional continuation: if $\eval^*_{(1)}$ is $0$ with positive probability, then it may happen that the evidence $\eval^{(m)} = \prod_{j=1}^m \eval_{(j)}$ obtained so far remains $0$, no matter how large $\eval_{(j)}$ for $j \geq 1$ --- akin to losing all one's money in the first round at a roulette table. A similar objection applies to any polynomial $f$, but it does not apply to the logarithm,
which is also the asymptotically optimal choice for $f$ if samples are independent: by Kolmogorov's strong law of large numbers, any
sequence of $\S$-variables $\eval_{(1)}, \eval_{(2)}, \ldots$ based on independent $\Ydata_{(1)}, \Ydata_{(2)},
\ldots$ with  
$\sup_{j \in \naturals} \Exp_{\Ydata_{(j)} \sim P_{\theta_1}}[(\log \eval_{(j)})^2] / (j / \log^2 (j+1))  < \infty$ (in particular this holds if the variances are uniformly bounded), will
a.s.\ satisfy $
(1/m) \sum _{j=1}^m \left(\log \eval_{(j)}\right) - (1/m) \sum_{j=1}^m   \Exp_{\Ydata_{(j)} \sim P_{\theta_1}}[\log \eval_{(j)}]  \rightarrow 0$. Thus, 
$\eval^{( m)} \coloneqq \prod_{j=1}^m \eval_{(j)}$ will grow exponentially fast if all $\Exp_{\Ydata_{(j)} \sim P_{\theta_1}}[\log \eval_{(j)}]  > 0$, with maximal growth rate attained if the $\eval_{(j)}$ are chosen to maximize $\Exp_{\Ydata_{(j)} \sim P_{\theta_1}}[\log \eval_{(j)}]$ --- a quantity which, for $\log$ taken to the base $2$, is known as the \emph{doubling rate} \citep{CoverT91,Kelly56}. This provides a powerful reason for choosing the logarithm; see also the extensive exposition by 
\cite{Shafer19}. We return to GRO's motivation in Section~\ref{sec:GROdiscussion}. 

\begin{ownexample}{\bf\ [$2 \times 2$ Contingency Tables]}\label{ex:2x2}{ \normalfont Let $\cY^n = \{0,1\}^n$  and let $\cX = \{a,b\}$ 
	represent two categories. We start with a multinomial model $\cG$ on $
	\cZ = \cX \times \cY$, extended to $n$ outcomes by independence. We want 
	to test whether the $Y_i$ are dependent on the $X_i$. To this end, 
	we condition every distribution in $\cG$ on a fixed, given,
	$X^n = \bx$ with $\bx = (x_1, \ldots, x_n)$, and we let $\cH_1$ be the set of (or a subset of the) conditional 
	distributions on $\cZ$ that thus result. 
	We thus assume the design 
	of $\cX^n$ to be set in advance, but $N_1$, the number of ones, to 
	be random; alternative choices are possible and would lead to a 
	different analysis. Conditioned on $X^n= \bx$, the likelihood of an individual sequence $\by \mid \bx$ 
	becomes:
	\begin{align}\label{eq:apple}
	p_{\mu_{1|a}, \mu_{1|b}}(\by \mid \bx) = 
	 \mu_{1|a}^{n_{a1}} (1-\mu_{1|a})^{n_{a0}} \cdot \mu_{1|b}^{n_{b1}} (1-\mu_{1|b})^{n_{b0}},
	\end{align}
%
	where $n_{ji}$ is the number of times outcome $i$ was observed to fall in category $j$ and $\mu_{1|j}$ is the probability of observing a $1$ given category $j$.
	These densities define the full model $\{ P_{\mu_{1|a},
		\mu_{1|b}} : (\mu_{1|a}, \mu_{1|b}) \in \Thetafull \}$ with $\Thetafull
	= [0,1]^2$.  $\cH_0$, the null model, simply has $ (X_1, \ldots,
	X_n)$ and $\Ydata = (Y_1,\ldots, Y_n)$ independent, with $Y_i, \ldots, Y_n$
	i.i.d.\ $\text{Ber} (\mu)$ distributed, $\mu \in \Theta_0 \coloneqq
	[0,1]$, i.e.\ $p_{\mu}(\by \mid \bx) = p_{\mu}(\by) =
	\mu^{n_1}(1-\mu)^{n-n_1}
	$  with $n_{1} = n_{a1} + n_{b1}$. 
	We defer description of the test of  the full alternative $\{P_{\theta}: \theta \in \Theta_1 \}$ with $\Theta_1 =(0,1)^2$ against $\cH_0$ to Section~\ref{sec:noeffectsize}. For now, we assume a simple $\cH_1 = 
	\{P_{\theta_1 } \}$ for a specific 
	$\theta_1 = (\mu_{1|a},\mu_{1|b})$ with $\mu_{1|a} \neq \mu_{1|b}$. 
	\cite{TurnerLG21} shows that the RIPr for $P_{\theta_1}$, achieving the infimum in (\ref{eq:firstgro}) is given by $P^*_0 = P_{W^*_0}$ where $W^*_0$ is the degenerate prior that puts all its mass on the single point 
	$\mu^{\circ}  = (n_a \mu_{1|a} + n_b \mu_{1|b})/(n_a+ n_b)$. 
	Thus, the  $\theta_1$-GRO $\S$-variable has an intuitive form here, being  given by 
	\begin{align}
	\label{eq:eval2x2}
	\eval^* = \frac{p_{\mu_{1|a},\mu_{1|b}}({\bf Y} \mid {\bf x})}{p_{\mu^{\circ}}( {\bf Y} \mid {\bf x})}.
	\end{align}
%
	The fact that the  RIPr is achieved by a point prior is quite specific to contingency tables. 
	We also note that, while the expectation of  $\eval^*$ is bounded by $1$ under  all $P_\mu \in \cH_0$, its actual distribution function varies with $P_{\mu}$. This is in contrast to the t-test example, in which the GRO $\eval^*$ turns out to have the same distribution under all distributions in the null. 
}\end{ownexample}

\subsection{GRO when prior on \texorpdfstring{$\mathcal{H}_1$}{the alternative hypothesis} is available}
\label{sec:BayesGRO}
We now take a Bayesian stance regarding $\cH_1$ and, conditioned
on $\cH_1$, are prepared to represent our uncertainty by prior
distribution $W_1$ on $\Theta_1$.
The marginal distribution of $\bf Y$ is then $P_{W_1}(\bf Y)$. 
Applying Theorem~\ref{thm:maina} with $P_{W_1}$ as $Q$ then leads to the ``$W_1$-GRO $\S$-variable'' --- it would be optimal in the GRO sense under prior $W_1$. This $\S$-variable  is a Bayes factor, but gives only a quasi-Bayesian notion of evidence since  any prior $W_1$ on $\cH_1$ that we wish to adopt forces us to adopt a particular corresponding prior
$W^*_0 \in \cH_0$. 
One may perhaps consider this a small price to pay for creating a Bayes
factor that, by its Type-I error safety under optional continuation, should be acceptable to frequentists as well. Moreover it is often recognised that 
priors on $\Theta_0$ and $\Theta_1$ should somehow be `matched' to each
other \citep{berger1998bayes}; we may view the RIPr construction as
providing a reasonable (from a frequentist stance) matching.

\begin{ownexample}{\bf [Gaussian Location with Gaussian prior (z-test)]}\label{ex:normal}{ \normalfont
	Now consider  $\cH_1$ according to which the $Y_i$ are i.i.d.\ $\sim N(\mu,1)$ for some $\mu \in \Theta_1 = \reals$, so that  $p_{\mu}(\Ydata) = p_{\mu}(Y_1, \ldots, Y_n)  \propto \exp(-\sum_{i=1}^n (Y_i - \mu)^2/2)$. We let $\cH_0= \{P_0 \}$. 
	We perform a Bayes factor test using $\eval := p_W(\bY)/p_0(\bY)$ where
	we
	take the prior $W$  to have Gaussian density
	$w(\mu) \propto \exp(-\mu^2/2)$. 
	By \eqref{eq:itisone} we know that $\eval$ is not just a Bayes factor but also an  $\S$-variable.
	By straightforward calculation:
	\begin{align*}
	-\frac{1}{2} \log (n+1) +
	\frac{1}{2} (n+1) \cdot  \breve{\mu}_n^2,
 \end{align*}

	where $\breve\mu_n = (\sum_{i=1} Y_i)/(n+1)$ is the Bayes MAP estimator, which only differs from the ML estimator by $O(1/n^2)$: $\breve\mu_n -  \hat\mu_n  = \hat\mu_n/(n(n+1))$. If we were to reject $\Theta_0$ when
	$\eval \geq 20$ (giving, by Proposition~\ref{prop:pval} a Type-I error guarantee of $0.05$), we would thus 
	reject if
	\begin{equation}\label{eq:bayesthreshold}
	|\breve\mu_n| \geq \sqrt{\frac{5.99 + \log (n+1)}{n+1}}
	,\text{\ i.e.\ } |\hat\mu_n| \succeq \sqrt{(\log n)/n},
	\end{equation}
	where we used
	$2 \log 20 \approx 5.99$. Contrast this with the
	standard two-sided Neyman-Pearson (NP) test, which would reject (with $\alpha = 0.05$)
	if $|\hat\mu_n| \geq 1.96/\sqrt{n}$, or the one-sided test which would reject if $\hat\mu_n \geq 1.645/\sqrt{n}$ or the $\S$-value based tests of the next section: the standard Bayesian test is significantly more conservative, requiring more data to conclude rejection. In Section~\ref{sec:competitive} we investigate this further.
}\end{ownexample}

\section{The GROW $\S$-Variable}\label{sec:grow}
We now show how to construct good $\S$-variables if $\cH_1$ is composite and no prior on $\Theta_1$ is available. We focus on  variations of worst-case (maximin) growth optimality, but this is certainly not the only criterion that might be useful or valuable; see the discussion in Section~\ref{sec:GROdiscussion}.
In the case of simple $\cH_1 = \{P_{\theta_1}\}$, we aimed for $\S$-variables that could be expected to grow as fast as possible under $P_{\theta_1}$. 
Analogously, we would now like them to be constructed such that they can
be expected to grow large as fast as possible (one gets rich, gets
evidence against $\cH_0$) under {\em all\ } $P_1 \in \cH_1$. We call 
$\S$-variables satisfying this property GROW: {\em
	growth-rate optimal in worst-case}. We now discuss  the simplest, `raw' form of this criterion --- in some settings a modification of this criterion, REGROW, which we discuss in the next section, is more suitable.  
GROW  tells us to pick,
among all $\S$-variables relative to $\cH_0$, the  $\eval^*$ that achieves the {\em worst-case optimal expected capital growth rate}
\begin{equation}\label{eq:growdef}
\grow(\Theta_1) := 
\sup_{\eval: \eval \in \cE(\Theta_0)}
\inf_{\theta \in \Theta_1} \Exp_{P_{\theta}} \left[ \log \eval \right].
\end{equation}

\noindent 
{\bf Theorem~\ref{thm:maina}, First Generalisation} \textit{ Suppose all $P_{\theta}, P_{\theta'}$ with $\theta,\theta' \in \Theta_1$ satisfy $D(P_{\theta} \| P_{\theta'}) < \infty$, and have full support. If $
	\inf_{W_1 \in \priorset_{1}, W_0 \in \cW(\Theta_0)}
	D(P_{W_1} \| P_{W_0}) =  \inf_{W_0 \in \cW(\Theta_0)} D(P_{W^*_1} \|
	P_{W_0})   < \infty,$ (i.e.\ the minimum on the left over $\cW_1$ is achieved by $W^*_1$) 
	then
	then there exists an $\S$-variable 
	\begin{equation}
	\eval^*  := \frac{p_{W^*_1}(\bY)}{p^{*}_0(\bY)},
	\end{equation}	
	where $p_0^*$ is the density of $P^{*}_0$, a (potentially sub-) distribution satisfying $\inf_{W_0 \in \cW(\Theta_0)} D(P_{W^*_1} \|
	P_{W_0})=  D(P_{W_1^*} \| P^*_0)$, and $\eval^*$ achieves (\ref{eq:growdef}), satisfying, essentially uniquely:
	$
	\inf_{\theta \in \Theta_1}  \Exp_{\Ydata \sim P_\theta}[\log \eval^*]  = 
	\sup_{\eval\in \cE(\Theta_0)} \inf_{\theta \in \Theta_1} \Exp_{\Ydata \sim P_\theta}[\log \eval] =
	D(P_{W^*_1} \| P^{*}_0).
	$
	If further $D(P_{W_1^*} \| P^*_0) = D(P_{W_1^*} \| P_{W_0^*})$ for some $W_0^* \in \cW(\Theta_0)$, then $P^*_0 = P_{W_0^*}$.} 
\\ \ \\ \noindent
The earlier version of Theorem~\ref{thm:maina} is the special case we get if we set $\Theta_1 =\{\theta_1\}$ 
a singleton and $Q:= P_{\theta_1}$. 
This generalized version expresses that the GROW $\S$-variable is once again  a
Bayes factor --- a special one in fact, between the components of the {\em joint information projection\/} $(P_{W_1^*}, P^*_0)$  \citep{CsiszarT84}; see Figure~\ref{fig:jipr}.
As to computing $W^*_1$ in practice, the same remarks apply as we already made (underneath Corollary~\ref{cor:justone}) regarding computing $W^*_0$. 
\vspace*{-.3em}   
\begin{figure} 
	\center 
	\scalebox{0.5}{  \begin{tikzpicture}
\draw[rotate=0, line width=0.5mm] (0,0) ellipse (200pt and 100pt);
\draw[rotate=10, line width=0.5mm, blue] (-4, 0.2) ellipse (40pt and 60pt);
\node[draw,fill=blue,circle,text=white, minimum size=0.2cm, label=-45:{\huge $P_{W^*_0}$}] (A) at ($(-4.6,-0.4)+(0:2 and 1)$) {
};
\draw[rotate=0, line width=0.5mm, red] (3, 0) ellipse (80pt and 70pt);
\node[draw,fill=red,circle,text=black, minimum size=0.2cm, label=50:{\huge
  $P^*_{W_1}$}] (B) at ($(2.2,0)+(180:2 and 1)$) {
};
\draw[line width=0.5mm, dashed] (A) -- (B);
\node[text width=2cm] at ($(-0.5, 2.7)$) {\huge $\bm{\mathcal{P}\left(\Theta_1\right)}$};
\node[text width=2cm, red] at ($(3.5, 0.2)$) {\huge $\bm{\mathcal{P}\left(\Theta\left(\delta\right)\right)}$};
\node[text width=2cm, blue] at ($(-4.2, -0.2)$) {\huge $\bm{\mathcal{P}\left(\Theta_0\right)}$};
\end{tikzpicture}}
	\caption{\label{fig:jipr} Joint Information Projection (JIPr). $\Theta_0, \Theta_1$ represent non-overlapping models, ${\cal W}(\Thetasmall)$ is the set of all priors over $\Thetasmall$, and $\cP(\Thetasmall) = \{P_{W}: W \in {\cal W}(\Thetasmall) \}$. Theorem~\ref{thm:maina} implies that the GROW $\S$-variable between $\Theta_0$ and $\Theta_1$ is given by $p_{W^*_1}/p_{W^*_0}$, the Bayes factor between the two Bayes marginals that minimise KL divergence $D(P_{W_1} \| P_{W_0})$, assuming the minima are achieved.}
\end{figure}

\subsection{One-parameter models with minimum relevant effect size}\label{sec:effectsize}
Let $\Thetafull$ be a connected subset of $\reals$ indexing a 1-parameter parametric model $\{P_{\theta}: \theta \in \Thetafull\}$ with $\theta$ indicating the size of some effect.
If, as is standard practice in e.g.\ medical statistics, 
we have a {\em minimum clinically relevant effect size\/} $\delta^+$ and a status quo $\delta^- < \delta^+$ in mind, we want to test 
\begin{equation}\label{eq:newthetas}
\Theta_0 =  \{\theta \in \Thetafull: \theta \leq \delta^{-} \} \text{\ vs.\ }
\Theta_1 = \{\theta \in \Thetafull: \theta \geq \delta^+ \}.
\end{equation}
In standard cases, often $\delta^-=0$ and $\Theta_0 := \{0 \}$.
\begin{proposition}\label{prop:dominance}
	Suppose there exists a 1-dimensional statistic  $T = t({\bf Y})$ such that the family of densities $\{p_{\theta}: \theta \in \Theta\}$ has a monotone likelihood ratio in $T$. 
	Then for all $\delta^- < \delta^+$ with $\delta^-,\delta^+ \in \Thetafull$,  the GROW $\S$-variable  relative to $\Theta_1$ and $\Theta_0$ as in (\ref{eq:newthetas}),  
	is given by $\eval^* = p_{\delta^+}({\bf Y})/p_{\delta^-}({\bf Y})$:
	it sets $W^*_1$ and $W^*_0$ 
	to be degenerate priors, putting all mass on $\delta^+$ and $\delta^-$, respectively.
\end{proposition}
We now illustrate Proposition~\ref{prop:dominance} for 1-dimensional exponential families, but stress that it can be applied to some other families (e.g.\ location families or the t-test setting in Section~\ref{sec:fulltheorem}) as well.  
\begin{ownexample}{\bf \ [GROW for 1-dimensional exponential families]}\label{ex:ponential}
{ \normalfont 	Let $\{ P_{\theta} \mid \theta \in \Thetafull\}$  represent a 1-parameter exponential family for sample space
	$\cY$, given in its mean-value
	parameterisation, where $\Thetafull$ is a connected subset of (and possibly equal to) the full mean-value parameter space.  
	Let $\delta^- < \delta^+$ with $\delta^-, \delta^+$ both in $\Thetafull$.  Both $\cH_0 = \{P_{\theta} :\theta \in \Theta_0 \}$ and
	$\cH_1 = \{ P_{\theta}: \theta \in \Theta_1 \}$ with $\Theta_0, \Theta_1$ as in (\ref{eq:newthetas}) are extended to outcomes in $\Ydata = (Y_1, \ldots, Y_n)$ by independence.
	Let $T = t({\bf Y})$ be the sufficient statistic of the exponential family under consideration, i.e.\ $\Exp_{{\bf Y} \sim P_{\theta}} [t({\bf Y})] = \theta$. 
	It is well-known that the monotone likelihood property holds in the statistic $T$.
	It thus follows from Proposition~\ref{prop:dominance} above 
	that the GROW
	$\S$-variable relative to $\Theta_1$ and $\Theta_0$  can be calculated as a likelihood ratio
	$\eval^* = p_{\delta^+}(\Ydata)/p_{\delta^-}(\Ydata)$ between two point hypotheses, even though
	$\Theta_1$ and/or $\Theta_0$ may be composite.
	Comparison of the ensuing test to the Neyman-Pearson and Bayes factor tests are given in Section~\ref{sec:competitive}.}
\end{ownexample}

\section{The REGROW $\S$-variable: general composite \texorpdfstring{$\mathcal{H}_1$}{alternative hypothesis} case}\label{sec:regrow}
\noindent {\bf Theorem~\ref{thm:maina}, Further Generalisation}
\textit{ Let $f(\theta)$ be a function that is bounded on $\Theta_1$; we abbreviate $f(W) := \Exp_{\theta \sim W} [f(\theta)]$. Suppose all $P_{\theta}, P_{\theta'}$ with $\theta,\theta' \in \Theta_1$ satisfy $D(P_{\theta} \| P_{\theta'}) < \infty$, and have full support.  If $
	\inf_{W_1 \in \cW(\Theta_1), W_0 \in \cW(\Theta_0)}
	(D(P_{W_1} \| P_{W_0}) - f(W_1) )=  \inf_{W_0 \in \cW(\Theta_0)} D(P_{W^*_1} \|
	P_{W_0}) - f(W^*_1)  < \infty$ 
	then there exists an $\S$-variable $\eval^f$ given by 
	\begin{equation}\label{eq:regroweval}
	\eval^f := \frac{p_{W^*_1}(\bY)}{p^*_0(\bY)}
	\end{equation}
	where $p^*_0$ is the density of $P^*_0$, a (potentially sub-) distribution such that $\inf_{W_0 \in \cW(\Theta_0)} D(P_{W^*_1} \|
	P_{W_0})=  D(P_{W_1^*} \| P^*_0)$, and $\eval^f$ satisfies, essentially uniquely:
	\begin{equation}\label{eq:jiprb}
	\inf_{\theta \in \Theta_1} (\ \Exp_{\Ydata \sim P_\theta}[\log \eval^f] - f(\theta) 
	\ ) =
	\sup_{\eval\in \cE(\Theta_0)} \inf_{\theta \in \Theta_1} (\ \Exp_{\Ydata \sim P_\theta}[\log \eval] - f(\theta) \ )
	=  
	D(P_{W^*_1} \| P^*_0)  - f(W^*_1).
	\end{equation}  
	If further $D(P_{W_1^*} \| P^*_0) = D(P_{W_1^*} \| P_{W_0^*})$ for some $W_0^* \in \cW(\Theta_0)$, then $P^*_0 = P_{W_0^*}$.} \\ \ \\ \noindent
We call $\eval^f$ the REGROW (standing for {\em relative\/} growth in the worst-case) $\S$-variable relative to {\em offset\/}  $f$. The previous version of Theorem~\ref{thm:maina} is  the special case with $f$ constant. 
The offset $f$ will be useful when $\Theta_0$ and $\Theta_1$ are nested and no effect size can be stated in advance  (Section~\ref{sec:noeffectsize}) and/or when nuisance parameters are present (Section~\ref{sec:nuisance} and~\ref{sec:fulltheorem}). All these cases can be handled essentially the same way (and we may in fact think of the case of nested models as a situation in which {\em all\/} parameters in $\Theta_0$ are viewed as nuisance): we first consider a modified problem in which 
$\Theta_1$ is reduced to a singleton. That is, we imagine that  some oracle tells us ``if $\cH_1$ is the case, then the data are sampled from this specific $\theta^*_1$''. 
We then consider the corresponding  $\gro(\theta^*_1) = \grow(
\{\theta^*_1\})$ and view this as the desirable but unobtainable expected growth rate --- the one we could have obtained if we had known $\theta^*_1$.  
We may now aim for the $\S$-variable such that,  no matter what $\theta^*_1$ turns out to be, our expected growth is close to the optimum we could have obtained had we known $\theta^*_1$. 
Thus, we want to be worst-case growth optimal {\em relative\/} to $f(\theta_1) := \gro(\theta_1) = 
\Exp_{P_{\theta_1}}[\log E^*_{\theta_1}] = 
\inf_{W_0 \in \cW(\Theta_0)} D(P_{\theta_1} \| P_{W_0})$ (where we write $\eval^*_{\theta_1}$ for the GRO $\S$-variable for testing $\{\theta_1\}$ vs. $\Theta_0$ and the second equality follows by (\ref{eq:grodef})). Plugging in this $f$  and taking negatives on both sides, (\ref{eq:jiprb}) now becomes:
\begin{multline}\label{eq:jipre}
\sup_{\theta_1 \in \Theta_1} \Exp_{\Ydata \sim P_{\theta_1}}\left[\log \eval^*_{\theta_1} - \log \eval^{f}
\right]= 
\inf_{\eval\in \cE(\Theta_0)} \sup_{\theta_1  \in \Theta_1} \Exp_{\Ydata \sim P_{\theta_1}}[\log \eval^*_{\theta_1} - \log \eval]
= \\  \Exp_{\theta_1 \sim W^*_1} \left[
\inf_{W_0 \in \cW(\Theta_0)} D(P_{\theta_1} \| P_{W_0} ) 
\right]
- D(P_{W^*_1} \| P^*_{0}),
\end{multline} 

\noindent 
an expression that is always nonnegative, since, by definition of $\eval^*_{\theta_1}$, for any $\S$-variable $E$, 
$\Exp_{\Ydata \sim P_{\theta_1}}[\log \eval^*_{\theta_1}] \geq  \Exp_{\Ydata \sim P_{\theta_1}}[\log \eval]$. This shows that $\eval^f$ can be thought of as a {\em minimax pseudo-regret\/} $\S$-variable, regret being the loss of expected capital growth under $\cH_1$ due to not knowing the underlying $\theta_1$ in advance.

\subsection{Composite \texorpdfstring{$\mathcal{H}_1$}{alternative hypothesis}, no effect size known}\label{sec:noeffectsize}
Suppose we are interested in detecting whether there is any deviation at all from the null. There is no pre-stated effect size, and $\Theta_0 \subset \Theta_1 = \Thetafull$ are nested, or more generally, for all $\theta_1\in \Theta_1$, $\inf_{\theta_0 \in \Theta_0} D(P_{\theta_1}\| P_{\theta_0}) = 0$.
In this case, $\grow(\Theta_1) = 0$ and the GROW $\S$-variable that achieves it is equal to $\eval^*=1$, which will never give any evidence against $\cH_0$, so clearly, the raw GROW approach is not useful. Instead, in this setting, the REGROW approach is a sensible generalisation of successful existing approaches. We first establish this for simple nulls:
\paragraph{Simple Nulls} If $\Theta_0 = \{0 \}$ is simple,  we have $\inf_{W_0} D(P_{\theta_1} \| P_{W_0}) = D(P_{\theta_1} \| P_{0} )$ and 
$D(P_{W_1^*} \| P_{W_0^*}) = D(P_{W_1^*} \| P_{0})$,
and terms in (\ref{eq:jipre}) involving $- \log p_0(\bY)$ cancel. 
Further using the 1-to-1 mapping (\ref{eq:simpleEisLR}) 
between probability densities and $\S$-variables for the case of point $0$'s to 
rewrite  (\ref{eq:jipre}) and using $\eval^f  = 
p_{W^*_1}({\bf Y})/p_0({\bf Y})$, (\ref{eq:jipre}) simplifies to:  
\begin{align}\label{eq:jiprc}
\sup_{\theta \in \Theta_1}  \Exp_{\Ydata \sim P_\theta}\left[- \log \frac{p_{W^*_1}(\bY)}{p_\theta(\bY)} \right] 
=
\inf_{q} \sup_{\theta \in \Theta_1} \Exp_{\Ydata \sim P_\theta}\left[- \log \frac{q(\bY)}{p_\theta(\bY)}\right] 
= \sup_{W_1 \in \cW(\Theta_1)} \  \Exp_{\theta \sim W_1}
\left[ D(P_{\theta} \| P_{W_1}) \right],
\end{align}

\noindent 
where the infimum is over all sub-probability densities $q$ over $\bY$. 
(\ref{eq:jiprc}) is just the redundancy-capacity theorem \citep{CoverT91} of information theory, and it has a data-compression interpretation.
In a nutshell, for any $\S$-variable of the form $p_{W_1}(\bY)/p_0(\bY)$, the $\log$ evidence 
$\log p_{W_1}(\bY)/p_0(\bY)$ is thought of as a difference between the code length needed to code the data using two lossless codes, one with lengths $-\log p_{W_1}$, associated with $\cH_1$, and one with lengths $- \log p_0$, associated with $\cH_0$. (\ref{eq:jiprc}) expresses that when choosing $W_1 = W^*_1$, one associates $\cH_1$ with the code that minimises worst-case redundancy (the additional expected number of bits needed compared to an encoder that knows $\theta^*_1$). 
This is in accordance with the {\em MDL\/} (Minimum Description Length) Principle, in which  code length difference between the same two codes is used to measure evidence \citep{BarronRY98,GrunwaldR20}.

\begin{ownexample}{\bf\ [Exponential Families with a point null: Jeffreys' Prior on $\Theta_1$]}{ \normalfont \label{ex:jeffreys}
	To make this more concrete, let $\{P_{\theta}: \theta \in \Thetafull\}$ represent a $d$-dimensional exponential family given in either the mean or the canonical parameterisation. 
	We restrict the parameter set to $\Theta_1$ that is a compact subset of the interior of $\Thetafull$ and let $\Theta_0$ be a singleton subset in the interior of $\Theta_1$. By standard properties of exponential families, the finite KL condition of Theorem~\ref{thm:maina} applies  and the problem reduces to finding the prior $W^*_1$ on $\Theta_1$ that satisfies (\ref{eq:jiprc}). 
	\citep{ClarkeB94} showed that, for large $n$, this prior converges in an $L_1$-sense to Jeffreys' prior (`least favourable under entropy loss') , which is the main reason for adopting it in MDL model selection. They also showed that (\ref{eq:jiprc}) and hence (\ref{eq:jipre}) is of size $(d/2) \log n + O(1)$.
	Thus, for point nulls and suitably truncated parameter spaces, this approach is consistent with the MDL Principle and with objective Bayes approaches based on Jeffreys prior.} 
\end{ownexample}
\begin{ownexample}{\bf\ [$2 \times 2$ Tables, Continued]}
	\label{ex:2x2c} { \normalfont If $\Theta_0$ is not a singleton  
	then the simplification of (\ref{eq:jipre}) to (\ref{eq:jiprc}) is not possible,  and numerical simulation can be used
	to determine (\ref{eq:jipre}) and the priors appearing therein. Consider for example the $2 \times 2$ model, but now with unrestricted $\Theta_1 = (0,1)^2$. This does satisfy the regularity conditions needed for Theorem~\ref{thm:maina} to be applicable (see Appendix~\ref{app:regularityconditions}), but it has $\Theta_1$  2-dimensional and $\Theta_0$ 1-dimensional. We saw in the previous example that for 1-vs. 0-dimensional exponential family models, (\ref{eq:jipre}) would take on value $ (1/2) \log n + O(1)$, which suggests that it is the same here, for $2$- vs. $1$-dimensional. This is confirmed by numerical simulations \citep{TurnerLG21}.}
\end{ownexample}

\subsection{Composite \texorpdfstring{$\mathcal{H}_1$}{alternative hypothesis},  nuisance parameters present}\label{sec:nuisance}
We now consider the common situation of models that can be parameterised by $\Thetafull
=\{(\delta,\gamma): \delta \in \Delta, \gamma \in \Gamma\}$ where
$\delta$ is a single parameter of interest (for simplicity taken to be  a scalar) and $\gamma$ represents a nuisance parameter (scalar or vector).
As in Section~\ref{sec:effectsize}, we  want to test whether $\delta \geq \delta^+$ or $\delta \leq \delta^-$ for some $\delta^- < \delta^+$.
We thus let 
\begin{align}\label{eq:nuisanceH0H1}
\Theta_0 = \{(\delta,\gamma): \delta \leq \delta^-, \gamma \in \Gamma \}, \ \ \text{vs.}\ \  \Theta_1 = \{ (\delta,\gamma): \delta \geq \delta^+,\gamma \in \Gamma\}. \vspace{0em}
\end{align}
We will first consider the simplified problem in which we test $\Theta_{0,\delta^-} := \{(\delta^-,\gamma) : \gamma \in \Gamma \}$  vs. $\Theta_{1,\delta^+} := \{(\delta^+,\gamma) : \gamma \in \Gamma \}$, and later extend to (\ref{eq:nuisanceH0H1}). 
This simplified problem can be handled via a REGROW $\S$-variable just like in the previous subsection, with now $\theta = (\delta^+,\gamma)$: 
we take $f((\delta^+,\gamma)) = \gro((\delta^+,\gamma))$ 
and apply Theorem~\ref{thm:maina} as in (\ref{eq:jipre}). This gives an $\S$-variable $\eval^*_{\delta^+} := p_{W^*_1}({\bf Y})/p^*_0({\bf Y})$ with $W^*_1$ a prior on $\{(\delta^+,\gamma): \gamma \in \Gamma\}$. Using this REGROW rather than GROW approach reflects a particular interpretation of what it means for a parameter $\gamma$ to be nuisance: we have no idea of what the true $\gamma$ might be and are therefore prepared to incur the same expected loss of growth for not knowing $\gamma$, irrespective of what $\gamma$ is. 
Solving this problem for all $\delta^+ \geq \delta^-$ gives us a collection of $\S$-variables $\cE_{\geq \delta^-} := \{ \eval^*_{\delta}: \delta \geq \delta^-\}$. Now suppose there exists another $\S$-variable $\eval^*$ such that 
\begin{equation}\label{eq:above}
\sup_{\eval \in \cE_{\geq \delta^-} } \inf_{\theta \in \Theta_1 } \Exp_{P_\theta} [\log \eval] =
\inf_{\theta \in \Theta_1 } \Exp_{P_\theta} [\log \eval^*]
\end{equation}
That is, we pick the worst-case optimal $\S$-variable among $\cE_{\geq \delta^-}$, thereby applying GROW on a meta-level as it were, after restricting ourselves to $\S$-variables that are themselves REGROW for fixed $\delta$ and unknown $\gamma$. This $\eval^*$ is then our choice for solving the original problem (\ref{eq:nuisanceH0H1}).

\begin{figure}[]
	\centering
	\includegraphics[width=0.25\textwidth]{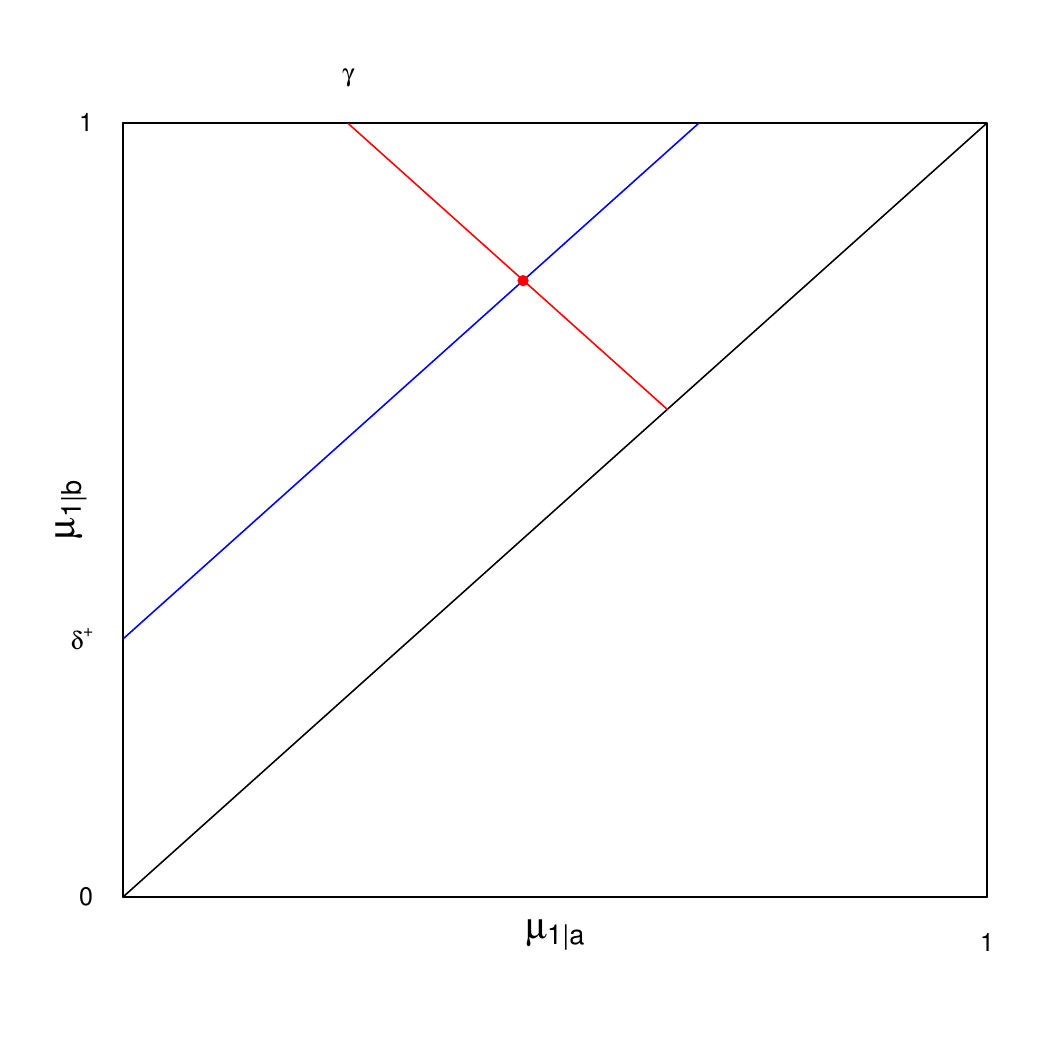}
	\caption{The $2 \times 2$ model. The diagonal represents the null, the decreasing line the set of parameters with nuisance parameter $\gamma = 1/4$ and the blue increasing line is $\Theta_{1,\delta^+}$ for $\delta^+ = 1/3$.}\label{fig:balk}
\end{figure}
\begin{ownexample}{\bf\ [$2 \times 2$ Tables, Continued]}
	\label{ex:2x2b}{ \normalfont We can reparameterise $\{ P_{\mu_{1|a},\mu_{1|b}} : (\mu_{1|a},\mu_{1|b}) \in [0,1]^2 \}$ as $\Theta = \{(\delta,\gamma): \delta \in \Delta$, $\gamma \in [0,1]$ \}
	using $\gamma = (n_a \mu_{1|a} + n_b \mu_{1|b})/(n_a+n_b)$ as a nuisance parameter: the marginal probability of observing a $1$. For $\delta$ we can take, for example, $\delta = \mu_{1|b} - \mu_{1|a}$ to be our notion of effect size, the {\em substantive difference}, with $\Delta = [-1,1]$.  Another popular choice, like substantive difference considered by \cite{Adams20,TurnerLG21} is $\delta = \log ( \mu_{1|b}/(1- \mu_{1|b})) \cdot (1- \mu_{1|a})/\mu_{1|a}$, i.e.\ the log-odds ratio, but for simplicity we stick to the substantive difference here. We take a $\Theta_1$ and $\Theta_0$ relative to some effect size $\delta^+$ and $\delta^-$ as in (\ref{eq:nuisanceH0H1}) above, where for simplicity we will take $\delta^- = 0$ and $\Delta = [0,1]$
	and also $n_a = n_b$ so that $\gamma = (\mu_{1|a}+\mu_{1|b})/2$.  The situation is depicted in Figure~\ref{fig:balk}, where we took, as an example, $\Theta_1$ and $\Theta_0$ defined relative to $\delta^+ = 1/3$ and $\delta^- = 0$. 
	
	Now, assume first that the true value of  $\gamma$ were given in advance. We would then  be dealing with a one-parameter exponential family model represented by a straight decreasing line in Figure~\ref{fig:balk}; the Figure illustrates this for $\gamma = 2/3$. We would then be in the situation of Section~\ref{sec:effectsize}, Example~\ref{ex:ponential}, and find, for any $\delta^+$, that  $\gro((\delta^+,\gamma)) =$ \\ $\inf_{W_0 \in \cW(\Theta_0)} D(P_{\delta^+,\gamma} \| P_{W_0}) = 
	D(P_{\delta^+,\gamma} \|P_{0,\gamma})$ where the latter equality was already stated as (\ref{eq:eval2x2}) in Example~\ref{ex:2x2}.
	
Now we look at unknown $\gamma$. As suggested above, we first set $\delta := \delta^+$ and test $\Theta_{1,\delta}$  vs. $\Theta_{1,0}$ (the increasing lines in Figure~\ref{fig:balk}), taking the REGROW $\S$-variable relative to $f((\delta,\gamma)) = \gro((\delta,\gamma))= D(P_{\delta,\gamma} \| P_{0,\gamma})$.
	Then the minimum $W^*_1$ on $\Theta_{1,\delta}$ and $W^*_0$ (with  $W^*_1$ putting mass $1$ on $\delta$ and spreading its mass over $\gamma$, and $P^*_0 = P_{W^*_0}$)  as in (\ref{eq:jipre}) are achieved and have finite support, and the finite KL condition of the theorem applies (Appendix~\ref{app:regularityconditions}). 
	This solves the problem for testing $\Theta_{1,\delta}$  vs. $\Theta_{1,0}$ for $\delta = \delta^+$; by varying $\delta$ we get a collection of  $\S$-variables $\cE_{\geq 0}$ containing an $\eval^*_{\delta}$ for each fixed $\delta \geq 0$. We then pick the $\eval^*$ among $\cE_{\geq 0}$ satisfying (\ref{eq:above}), which turns out equal to $\eval^*_{\delta^+}$: it has a point mass on $\delta^+$ again.  }
\end{ownexample} 

\paragraph{Discussion}
In our examples, we have used (and will keep using in the next section) the REGROW approach to first eliminate nuisance parameters, if they are present, followed by a GROW approach for the parameters of interest.\citep{TurnerLG21,TerschurePLG21}
find that this gives intuitive $\S$-variables that perform well in practical applications, not just directly in the GRO sense but also in terms of secondary measures such as a power analysis (Section~\ref{sec:competitive}) or as the basis of anytime-valid confidence intervals \citep{TurnerLG21}. 
Still, it may not always be the best way to go. For example, a REGROW approach for the parameter of interest when a minimum effect size is given may sometimes be sensible as well. Let us consider this a bit further for (for simplicity) the case with a minimum effect size $\delta^+$ but without nuisance parameters, as in Example~\ref{ex:ponential}. REGROW would amount here to using $E^*_{W_1} = p_{W_1}(\bY)/p_{\delta^-}(\bY)$ with some prior $W_1$ spread out on the set $\Theta_1 = \{\delta: \delta \geq \delta^+ \}$. Then we would get
$
\Exp_{P_{\delta}}[\log E^*_{W_1}]
< \grow(\Theta_1) \text{\ for $\delta$ close to $\delta^+$\ }
; 
\Exp_{P_{\delta}}[\log E^*_{W_1}]
\gg \grow(\Theta_1) \text{\ for $\delta \gg \delta^+$\ }
$
so we would win if we are `lucky' and it turns out that $\delta \gg \delta^+$. However, in practice we often deal with small sample sizes, and  $\delta$'s that may very well  be very close to $\delta^+$. Then (as our experiments done for the papers above indicate) the difference in `$<$' above is non-negligible, and the GROW approach seems safer, since for the GROW $\S$-variable we automatically have $\Exp_{P_{\delta}}[\log E^*_{W_1}]
\geq \grow(\Theta_1)$ for all $\delta \in \Theta_1$. 

\subsection{Theorem~\ref{thm:maina} in Full: Application to Bayesian and Sequential $t$-test}\label{sec:fulltheorem} 
If we try to apply Theorem~\ref{thm:maina} as above to the $t$-test, a prototypical nuisance setting, we run into the issue that the minimum KL  is not achieved. This problem can be solved by extending the theorem further, allowing for densities on a {\em coarsening\/} of $\Ydata$. This is
any random variable $\bV$ that can be written as a function of $\Ydata$,
i.e.\ $\bV=g(\Ydata)$ for some function $g$; we retrieve the previous version of Theorem~\ref{thm:maina} if we take $g$ the identity and $\bV=\Ydata$. We now present Theorem~\ref{thm:maina} in full generality, allowing for such coarsening and additionally for considering the best $\S$-variable on a modified $\cH_1$, consisting of any convex set of Bayes marginal distributions with priors on $\Theta_1$. This is needed for accommodating the $t$-test. It also allows us to incorporate {\em robust Bayesian\/} settings \citep{Berger85,GrunwaldD04}, but we will not further pursue those here.
In the theorem we use the following notation: for (sub-) distribution $P$ for $\bY$,  $P^{[\bV]}$ denotes the marginal (sub-) distribution of $P$ for  $\bV$, and $p'$ denotes its density.  \\

\noindent {\bf Theorem~\ref{thm:maina}, Full Generality}
\textit{ Let $f(\theta)$ be a function that is bounded on $\Theta_1$. Suppose all $P_{\theta}, P_{\theta'}$ with $\theta,\theta' \in \Theta_1$ satisfy $D(P_{\theta} \| P_{\theta'}) < \infty$, and have full support. Let $\cW_1 \subseteq \cW(\Theta_1)$ be convex. If 
	for some coarsening $\bV$
	of $\Ydata$ we have: 
	\begin{multline}\label{eq:tuinslang}
	\inf_{W_1
		\in \priorset_{1} } \inf_{W_0 \in \cW(\Theta_0)} \left(\ D(P_{W_1} \|
	P_{W_0}) - f(W_1)  \right)  = \\
	\min_{W_1 \in \priorset_{1} } \inf_{W_0 \in \cW(\Theta_0)}
	\left(\   D(P^{[\bV]}_{W_1} \| P^{[\bV]}_{W_0}) 
	- f(W_1) \right) 
	= \inf_{W_0 \in \cW_0(\Theta_0)} D(P^{[\bV]}_{W^*_1} \|
	P^{[\bV]}_{W_0})  - f(W^*_1) < \infty,\vspace*{-1 em} \end{multline} then there exists an $\S$-variable 
	\begin{equation}\label{eq:finale}
	\eval^f := \frac{p'_{W^*_1}(\bV)}{p^{*'}_0(\bV)},
	\end{equation}
		$p^{*'}_0$ being the density of $P^{*[\bV]}_0$, a (potentially sub-) distribution for $\bV$  that satisfies 
	$\inf_{W_0 \in \cW(\Theta_0)} D(P_{W^*_1} \|
	P_{W_0})=  D(P^{[\bV]}_{W_1^*} \| P^{*[\bV]}_0)$. 
	$\eval^f$ satisfies, essentially uniquely:
	\begin{equation}\label{eq:jipr}
	\inf_{W \in \priorset_1} (\ \Exp_{\Ydata \sim P_W}[\log \eval^f] - f(W) 
	\ ) =
	\sup_{\eval\in \cE(\Theta_0)} \inf_{W \in \priorset_1} (\ \Exp_{\Ydata \sim P_W}[\log \eval] - f(W) \ )=  
	D(P^{[\bV]}_{W^*_1} \| P^{*[\bV]}_0)  - f(W^*_1).
	\end{equation}  
	If further $D(P_{W_1^*}^{[\bV]} \| P^{* [\bV] }_0) = D(P_{W_1^*}^{[\bV]} \| P_{W_0^*}^{[\bV]})$ for some $W_0^* \in \cW(\Theta_0)$, then $P^{* [\bV]}_0 = P_{W_0^*}^{[\bV]}$.}  \\ \ \\ \noindent 
The previous version of Theorem~\ref{thm:maina} is the  special case obtained by setting $\priorset_1 = \priorset(\Theta_1)$, $\bY = \bV$ and using linearity of 
expectation. 
We call $\eval^f$ as in (\ref{eq:jipr}) the REGROW $\S$-variable  relative to offset $f$ and set of priors  $\priorset_1$. If $f$ is constant (no offset), we call it {\em $\cW_1$-GROW}, noting that it gives worst-case optimal growth rate under all priors in $\cW_1$.

\paragraph{The $t$-test Setting}
\label{sec:group}
We return to the setting with a nuisance parameter with notation as in Section~\ref{sec:nuisance}. 
\cite{Jeffreys61} proposed a Bayesian version of the
$t$-test; see also \citep{rouder-2009-bayes}. We start with the models
$\cH_0$ and $\cH_1$ for data $\Ydata = (Y_1, \ldots, Y_n)$ given as $\cH_0
= \{ P_{0,\sigma}(\Ydata) \mid \sigma \in \Gamma \}$; $\cH_1 = \{
P_{\delta,\sigma}(\Ydata) \mid (\delta,\sigma) \in \Theta_1  \}$, where $\Delta = \reals,\Gamma = \reals^+$, $\Theta_1 \coloneqq \Delta
\times \Gamma$ and $\Theta_0 = \{(0,\sigma) : \sigma \in \Gamma \}$, and
$P_{\delta,\sigma}$ has density (with $\overline{y} = \frac{1}{n} \sum_{i=1}^n y_i$)
\begin{align*}
p_{\delta,\sigma}(y)  & = \frac{1}{(2\pi \sigma ^2)^{n/2}} \cdot  \exp\left( -\frac{n}{2} \left[ \left(
\frac{\overline{y}}{\sigma} - \delta \right)^2
+ \left( \frac{\frac{1}{n}\sum_{i=1}^n (y_i -
	\overline{y})^2}{\sigma^2}
\right) \right] \right),
\end{align*}
Jeffreys proposed
to equip $\cH_1$ with a Cauchy prior $\WCauchy[\delta]$ on the {\em
	effect size\/} $\delta$, and both $\cH_1$ and $\cH_0$ with the
scale-invariant prior measure with density $w^H(\sigma) \propto
1/\sigma$ on the variance.
The same formula with the same prior on $\sigma$ but  other priors on $\delta$ was suggested by \cite{lai1976confidence} with a non-Bayesian, martingale interpretation.
Below we will see that, even though 
$w^H(\sigma)$ is improper (whereas the priors appearing in
Theorem~\ref{thm:maina} are invariably proper), the resulting Bayes
factor $E^*$ is an $\S$-variable. We then present Theorem~\ref{thm:particular} which shows that, for priors $W[\delta]$
with more than 2 moments, $E^*$ in fact even is $\cW_1$-GROW with $\cW_1$ the set of all product priors on $\delta \times \sigma$ with marginal $W[\delta]$ on $\delta$, i.e.\ it has  a worst-case optimal growth rate  property
relative to all distributions in $\cH_1$ compatible with
$W[\delta]$. Thus, a form of GROW-optimality holds for most priors $W[\delta]$ 
one might want to use, including standard choices (such as a standard normal) or  the point prior we will suggest further below --- but we do not know if it holds for the moment-less Cauchy proposed by Jeffreys.

\paragraph{Almost Bayesian Case: prior on $\delta$ available}
For any proper prior distribution $\priordist[\delta]$ on $\delta$ and
any proper prior distribution $\priordist[\sigma]$ on $\sigma$, we
define
$
p_{W[\delta],W[\sigma]}(\yvec)   = \int_{\delta \in \Delta} \int_{\sigma \in \Gamma} 
p_{\delta,\sigma}(\yvec)    \dif \priordist {[\delta]} \dif \priordist{[\sigma]}, 
$
as the Bayes marginal density under the product prior $W[\delta] \times W[\sigma]$.

For convenience later on we set the sample space to be
$\cY^n = (\reals \setminus \{0\}) \times \reals^{n-1}$, assuming
beforehand that the first outcome will not be 0.
Now
we define $\bV \coloneqq (\Vcomp{1}, \ldots, \Vcomp{n})$ with
$\Vcomp{i} = \Ycomp{i} /|\Ycomp{1}|$. We have that $\Ydata$ determines
$\bV$, and $(\bV,Y_1)$ determines
$\Ydata = (\Ycomp{1},\Ycomp{2},\ldots,\Ycomp{n})$.  
The distributions in
$\cH_0 \cup \cH_1$ can thus alternatively be thought of as
distributions on the pair $(\Vvec,Y_1)$. $\Vvec$ is ``$\Yvec$ with the
scale divided out'': as is well-known \citep{lai1976confidence,berger1998bayes} and easily shown (Appendix~\ref{app:remainingproofsttest}), under all $P \in \cH_0$, i.e.\ all
$P_{0,\sigma}$ with $\sigma > 0$, $\Vvec$ has the same distribution
$P_0[\bV]$ with density $p'_{0}$.
In the same way, one shows that under
all ${P}_{\priordist[\delta],\sigma}$ with $\sigma > 0$, $\Vvec$ has the same pdf $p'_{W[\delta]}$ (which
therefore does not depend on the prior on $\sigma$). 
We now
get that, with 
 \begin{equation}\label{eq:marginalS}
\eval^*
\coloneqq \frac{p'_{W[\delta]}(\Vvec)}{p'_0(\Vvec)}, \end{equation}
we must have $\Exp_{\Vvec \sim P}[\eval^*]= 1$
for all $P \in \cH_0$, hence it is an $\S$ variable.
Remarkably, this `scale-free' $\S$-variable coincides
with the Bayes factor one gets if one uses, for
$\sigma$, the prior $w^H(\sigma) = 1/\sigma$ suggested by Jeffreys, and
treats $\sigma$ and $\delta$ as independent. That is \cite[page 273]{lai1976confidence} (a full derivation is in Appendix~\ref{app:remainingproofsttest}), 
we have
\begin{equation}\label{eq:suffu}
\frac{\int_{\sigma} p_{W[\delta],\sigma}(\Yvec)
	w^H(\sigma) \dif \sigma}
{\int_{\sigma} {p}_{0,\sigma}(\Yvec) w^H(\sigma) \dif \sigma}
=  \frac{p'_{W[\delta]}(\Vvec)}{p'_{0}(\Vvec)} = \eval^*.
\end{equation}
Despite its improperness, $w^H$ induces a valid $\S$-variable when
used in the Bayes factor. The equivalence of this Bayes factor to
$\eval^*$
simply means that
it manages to ignore the `nuisance' part of the model and models the
likelihood of the scale-free $\Vvec$ instead. The reason this is
possible is that $w^H $ coincides with the right-Haar prior for this
problem \citep{Eaton1989,berger1998bayes}, about which we will say
more below.  Amazingly, it turns out that the $\S$-variable
\eqref{eq:suffu} has a GROW property (among {\em all\/} $\S$-variables for data $\Ydata$,
not just the coarsened $\bV$)  under the weak condition that the
prior $W[\delta]$ has a $(2+\epsilon)$th moment.  This follows from a special case of Theorem 4.2. of \citep{perez2022estatistics} (for the case that $W[\delta]$ puts all its mass on a single $\delta$) and Corollary 8.3. (for general $W[\delta]$. For convenience we re-state this special case here. 
%
%
Let, for priors $W[\delta], W[\sigma]$,
$P_{W[\delta],W[\sigma]}^{[\bV]}$ be the marginal distribution 
with density $p'_{W[\delta],W[\sigma]}$. We have: 
\begin{theorem}{\bf [Special case of Theorem 4.2./Corollary 8.3. of \cite{perez2022estatistics}]}
	\label{thm:particular}
	Let $W[\delta]$ be a distribution on $\delta$ such that ${\bf E}_{\delta \sim
		W[\delta]}[|\delta|^{2+\epsilon}] < \infty$ for some $\epsilon >
	0$ (in particular this includes all degenerate priors with mass 1
	on a single ${\delta}$).  Let $\cW[\Gamma]$ be the set of all probability
	distributions $W[\sigma]$ on the variance $\sigma$. 
	Let
	$\priorset_1$ be the set of all product distributions on $\delta
	\times \sigma$ such that, for each 
	$W'\in \cW_1$, $\delta$ and $\sigma$ are independent and its marginal on $\delta$, i.e.\ $W'[\delta]$, coincides with $W[\delta]$.
	We have: 
	\begin{align}
	\inf_{W \in \priorset_1} \inf_{W[\sigma] \in \cW[\Gamma]}
	D(P_{W} \| P_{0,W[\sigma]}) =    
	\inf_{W[\sigma], W'[\sigma] \in
		\priorset(\Gamma)} D(P_{W[\delta], W[\sigma]} \| P_{0,W'[\sigma]})\label{eq:newdelta}
	&=
	D(P_{W[\delta]}^{[\bV]} \| P_0^{[\bV]}).
	\end{align}
	\end{theorem}
The theorem allows us to use 
Theorem~\ref{thm:maina} as above with constant $f(\delta,\sigma) = 0$ (note that $\cW_1$ is convex) to conclude that  $\eval^*$ as in (\ref{eq:marginalS})
is equal to $\eval^f$ as in (\ref{eq:jipr}): the Bayes factor based on the right
Haar prior, is not just an $\S$-variable, but is even  
GROW
relative to the set of all priors on $\delta \times \sigma$  that are  compatible with $W[\delta]$.

\paragraph{REGROW-GROW safe $t$-test with minimum effect sizes}
Suppose 
we want to test $\Theta_1$ vs. $\Theta_0$ as in (\ref{eq:nuisanceH0H1}) with fixed effect sizes $\delta^+$ and $\delta^-$ and with $\sigma^2$ in the role of $\gamma$. 
We proceed exactly as we did underneath (\ref{eq:nuisanceH0H1}): we first consider the test $\{(\delta^+,\sigma^2): \sigma^2 > 0 \}$ vs. 
$\{(\delta^-,\sigma^2): \sigma^2 > 0 \}$ for the fixed given $\delta^+$ using the REGROW criterion with $f((\delta,\sigma)) = \gro(\delta,\sigma)$. We have \citep[Section 4.3]{KoolenG21} that $f(\delta^+,\sigma)
= (n/2) \log (1+\delta^{+2})$ is constant on $\sigma$. Therefore we can use Theorem~\ref{thm:maina} in its most general form above in combination with Theorem~\ref{thm:particular} (applied with point prior $W[\delta]$ on $\delta^+$) to conclude that (both the GROW and) the REGROW $\S$-variable are given by  $\eval^*_{\delta^+} := p'_{\delta^+}(\bV)/p'_{\delta^-}(\bV)$. Since Proposition~\ref{prop:dominance} is applicable to sets of distributions defined on $\bV$ rather than $\bY$ 
(details in Appendix~\ref{app:remainingproofsttest}),
we find that, with $\cE_{\geq \delta^+} := \{ \eval^*_{\delta} : \delta \geq \delta^- \}$ that
$
\sup_{E \in \cE_{\geq \delta^+}} \inf_{\sigma > 0, \delta \geq \delta^+} \Exp_{{\bf Y} \sim P_{\delta,\sigma}}[\log \eval] =
\inf _{\sigma > 0, \delta \geq \delta^+} \Exp_{{\bf Y} \sim P_{\delta,\sigma}}[\log \eval^*_{\delta^+}]
$
so $\eval^*_{\delta^+}$ may be thought of as first applying REGROW, to get rid of the nuisance parameter, and then applying GROW --- just like in the $2 \times 2 $ Example~\ref{ex:2x2b}. 

\paragraph{Extension to General Group Invariant Bayes Factors}
%
%
In a series of papers
\citep{berger1998bayes,dass-2003-unified-condit,bayarri2012criteria},
Berger and collaborators developed a theory of Bayes factors for
$\cH_0= \{P_{0,\gamma} : \gamma \in \Gamma\}$ and $\cH_1 = \{
P_{\delta,\gamma}: \delta \in \Delta, \gamma \in \Gamma\}$ with a
nuisance parameter (vector) $\gamma$ that appears in both models and
that satisfies a group invariance; the Bayesian $t$-test is the
special case with $\gamma = \sigma, \Gamma = \reals^+$ and with the
scalar multiplication group and $\delta$ an `effect size'. Other
examples include regression based on mixtures of $ g$-priors
\citep{liang2008mixtures},  testing a
Weibull vs.\ the log-normal and many more 
\citep{dass-2003-unified-condit}.
The reasoning of the first part of this section straightforwardly
generalises to all such cases:  the Bayes factor based on using the right Haar measure on
$\gamma$ in both models gives rise to an $\S$-variable.
Theorem 4.2. of 
\cite{perez2022estatistics} shows that, if the underlying group satisfies a condition called {\em amenability\/} (which holds, e.g., for scaling as in the t-test, but also for e.g.\ rotations and affine transformations as in parametric linear regression models), then  the resulting
Bayes factor is  GROW relative to a suitably defined set $\cW_1$. Theorem~\ref{thm:particular} above is the very special case of their result when instantiated with $\gamma$ instantiated to the variance in the t-test (scaling). Although its proof is quite different, the general result may be viewed as the `e-variant' of the classical Hunt-Stein theorem \citep[Section 8.5]{lehmann2005testing}, with `power' in that theorem replaced by `GROW'. 
\cite[Proposition 4.4]{perez2022estatistics} then implies that in all such cases, this GROW Bayes factor is in fact also REGROW. Remarkably therefore, with parameters representing group transformations, unlike e.g.\ for the $2 \times 2$ case, GROW and REGROW e-variables generally coincide.

\section{(RE)GRO(W), Optional Continuation and Stopping}
\label{sec:GROanalysis}
\newcommand{\cU}{\ensuremath{\mathcal{U}}}
We now address two related questions:
\begin{enumerate}
\item We focused on Type-I error safety under optional continuation (OC). Can we also get safety under {\em optional stopping\/} (OS), and what is the difference? 
	\item The GRO-criteria were chosen to optimise expected capital (logarithmic) growth `locally', within a study. How well do GRO-criteria go together with OC over several studies?
	\end{enumerate}
To make the questions concrete, we consider a specific set-up with data stream 
$Y_{j,1}, Y_{j,2}, \ldots$  corresponding to the $j$-th study to be performed. 
In the first study we observe batch of outcomes $\bY_{(1)} = (Y_{1,1}, \ldots, Y_{1, N_{(1)}})$; in the second study (if it is performed at all) $\bY_{(2)} = (Y_{2,1}, \ldots, Y_{2,N_{(2)}})$; and so on. 
For further simplicity we will assume that all $Y_{j,i}$ are i.i.d.\
The set-up slightly differs from  Example~\ref{ex:normal} in which the second study's data was part of the same stream as the first; at the expense of additional notation, everything that follows can be formalized  in that setting as well. 

The conditional $\S$-variables determining our test martingale are now determined by a sequence of  stopping times $N_{(1)}, N_{(2)}, \ldots$.  
The first stopping time $N_{(1)}$ is defined as a stopping time on the first sequence $Y_{1,1}, Y_{1,2}, \ldots$ relative to filtration $(\sigma(Y_1^n))_n$ and defines a stopped $\sigma$-algebra $\cF_{(1)} :=\sigma(Y_1^{N_{(1)}})= \sigma(\bY_{(1)})$.
$N_{(2)}$ is defined on the second sequence $Y_{2,1}, Y_{2,2},\ldots$, but it is also allowed to depend on previous data $\bY_{(1)}$, i.e it is a stopping time relative to filtration $(\sigma({\bf Y}_{(1)}, Y_2^n))_n$, and defines a stopped $\sigma$-algebra $\cF_{(2)} :=\sigma(\bY_{(1)},\bY_{(2)})$. In general,
$N_{(m)}$ is defined relative to filtration $\sigma(\bY_{(1)},\ldots,\bY_{(m-1)}, Y_m^n)_n$, and defines a stopped $\sigma$-algebra $\cF_{(m)} :=\sigma(\bY_{(1)},\ldots,\bY_{(m)})$. We let $\cN_{(m)}$ be the collection of all stopping times for the $m$-th study, i.e.\ relative to
$\sigma(\bY_{(1)},\ldots,\bY_{(m-1)}, Y_{m,1}, \ldots, Y_{m,n})_n$.
The sequence of stopped $\sigma$-algebras thus defines a new filtration $(\cF_{(m)})_m$, which we call 
the filtration at the {\em study level},
since $\cF_{(m)}$ denotes all information available after $m$ studies or trials have been completed --- it is the filtration referred to in Proposition~\ref{prop:optional.stopping} (for the t-test we need to extend this set-up a little, see Example~\ref{ex:t}).  

\paragraph{OC vs. OS}
In the optional continuation setting, we assume that, after having observed and analyzed $m-1$ studies, we either stop, or continue to the next study. In the latter case we need to specify a stopping time $N_{(m)} \in \cN_{(m)}$ and a conditional e-variable $E_{(m)} \in \cE_{(m)}$ where $\cE_{(m)}$ is the set of all $\cF_{(m-1)}$-conditional e-variables.
For example, in the setting of Example~\ref{ex:oc}, with a simple null $\cH_0 = \{P_0\}$ and composite alternative $\cH_1 = \{P_{\theta}: \theta \in \Theta_1\}$ and $N_{(m)} \in \cN_{(m)}$ an arbitrary stopping time for the $m$-th study, 
$e_{N_{(m)},W}(Y_{m,1}, \ldots, Y_{m,N_{(m)}}) := p_{W}(Y_{m,1}, \ldots, Y_{m,N_{(m)}})/p_0(Y_{m,1}, \ldots, Y_{m,N_{(m)}})$ is an $\S$-variable for every `prior' distribution $W$ on $\Theta_1$ (here we generalize the notation of Example~\ref{ex:oc} to allow for data-dependent stopping times). 
In this setting the analyst can freely choose (or somebody else can impose) any $N_{(m)}$ and any $W$ and use the corresponding $\S$-variable $e_{N_{(m)},W}$; all such $\S$-variables are contained in the set $\cE_{(m)}$.  As explained in Example~\ref{ex:oc}, this includes the choice to set $W_{(m)}:= W_{(1)}$ (re-use the original prior) and the choice to set $W_{(m)}:= W_{(1)}(\cdot \mid \bY^{(m-1)})$, i.e.\ use the Bayes posterior based on the previous studies.


We can now interpret our various GRO criteria
as each providing a {\em prescription\/} to choose specific elements of $\cE_{N_{(m)}}$, for all stopping times $N_{(m)}$ that are constant given the outcomes of previous studies $\bY^{(m-1)}$, i.e.\ that are $\cF_{(m-1)}$-measurable.
To see this, note that each GRO criterion in combination with $\cH_0$ and $\cH_1$ defines, for each $n$,  an $\S$-variable $S^{[n]}$ for a single sequence $Y_1, \ldots, Y_n$. In all cases, this can be written as $S^{[n]} = s^{[n]}(Y^n)$ for some function  $s^{[n]}$;  the function $s^{[n]}$ is what is really specified; we call $s^{[1]}, s^{[2]}, \ldots$ an e-specification. If $N_{(m)}:= n_m$, we can thus simply set $E_{(m)} := s^{[n_m]}(Y_{m,1}, \ldots, Y_{m,n_m})$. 
We call this the {\em plug-in\/} method for constructing $E_{(m)}$ from specification $s^{[1]}, s^{[2]}, \ldots$ (such specification may correspond to one of our GRO criteria, but in general it could be arrived at in different ways as well) . 
The running product of $E_{(m)}$ thus constructed provides, via Proposition~\ref{prop:optional.stopping}, a test martingale at the study level.

In contrast, {\em optional stopping\/} scenarios  usually concern only a single {\em data-level process\/} $Y_1, Y_2, \ldots$, without any super-structure in terms of subsequent studies. The scenario is  completely determined by a sequence of  conditional $\S$-variables $S_1, S_2, \ldots$, i.e., applying Definition~\ref{def:cond.safe} at the data level, such that 
for all $n \in \naturals$,
	$\text{for all $P \in \cH_0$:}\ \ {\bf E}_{P}[S_n \mid \mathcal F_{n-1} ]  \leq 1  \text{\ a.s.}
	$, with $\cF_{n} = \sigma(Y^{n})$.  Their running product $S^{[1]}, S^{[2]}, \ldots$ (with $S^{[n]} = \prod_{i=1}^n S_i$) then forms a test martingale. 
Recall from Corollary~\ref{cor:villerobbins} in Section~\ref{sec:os} that for any nonnegative random process $E^{(1)}, E^{(2)}, \ldots$ at the study level (adapted to $(\cF_{(m)})_{m}$) we say that the corresponding threshold test is {\em safe under OC\/} (with respect to Type-I error) if the Ville-Robbins inequality (\ref{eq:villerobbins}) holds. Extending this definition in the natural way, we say, for a data-level process of nonnegative random variables $S^{[1]}, S^{[2]}, \ldots$, i.e.\ with $S^{[j]}$ adapted to $\cF_j$, that the corresponding threshold test is safe under OS
(with respect to Type-I error) if again the Ville-Robbins inequality holds (with $E^{(n)} := S^{[n]}$).
\begin{figure} 
	\center 
	\includegraphics[width=0.9\textwidth]{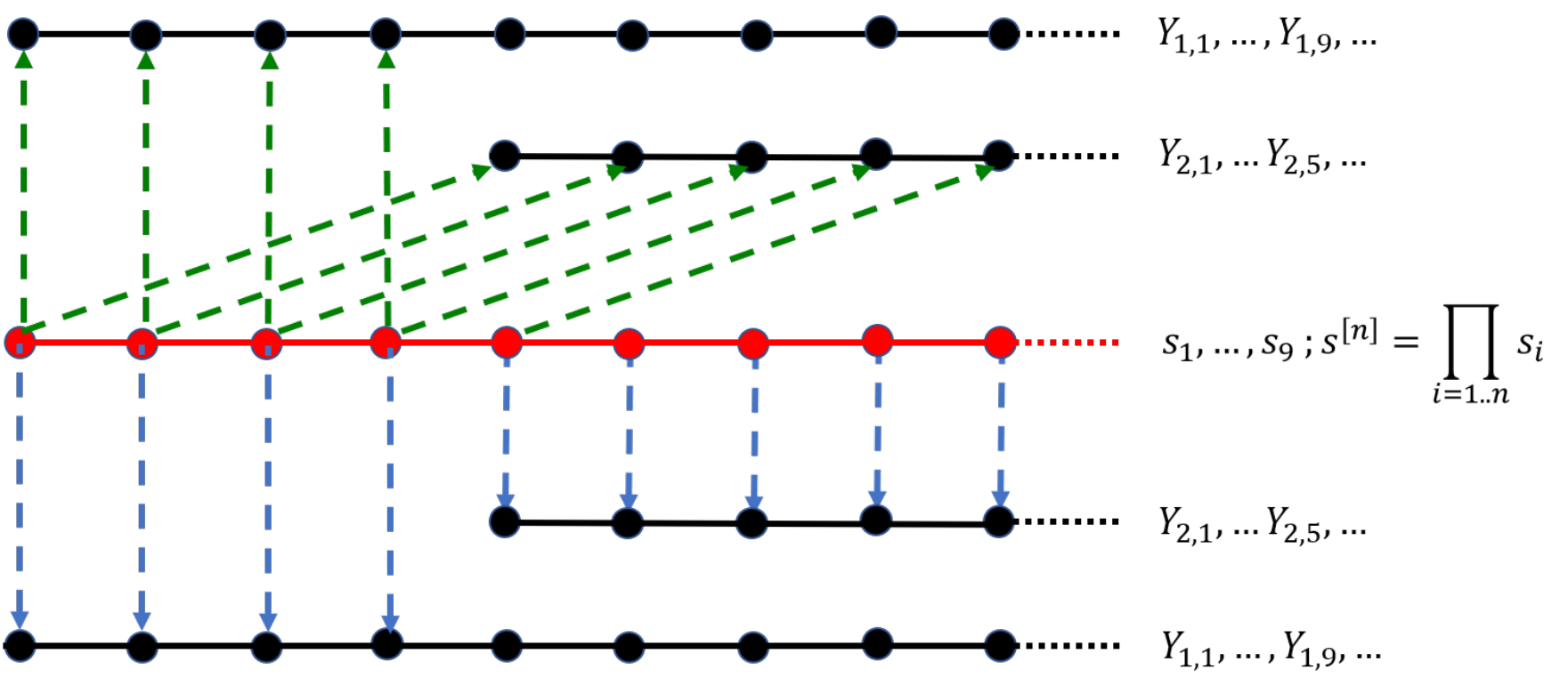}
	\caption{\label{fig:seqdec} Two ways of applying {\em seqdec\/} e-variable specifications to subsequent studies. We observe $N_{(1)}= 4$ data points in the first study (represented by the top and bottom-most line), and $N_{(2)} = 5$ data points in the second (represented by the 2nd and 4th line above). Since the specification is seqdec, it provides a sequence of functions $s_i$ on $Y^i$, represented by the dots on the red-line. The plug-in application uses $s^{[4]}$ for the first batch and  $s^{[5]}$ applied to the second data batch, as depicted in the top two lines. The sequential application uses $s^{[4]}$ for the first and $\prod_{i=1}^5 s_{N_{(1)}+i}$ for the second, as depicted on the bottom two lines. 
	If the specification is not seqdec, then $s^{[n]}$ does not decompose into a product of $s_i$ and in general only the plug-in application can be used.}
\end{figure}

\paragraph{Sequentially Decomposable $\S$-Specifications}
In many (not all) cases, the GRO-specification $s^{[1]}, s^{[2]}, \ldots$ forms itself a test martingale relative to  some filtration $(\cG_n)_n$: there exist a sequence of functions $s_1, s_2, \ldots$, with $s_i$ a function on $\cY^i$, such that, for all $n$,  $s^{[n]}(Y^n) = \prod_{i=1}^n S_i$ with $S_i := s_i(Y^i)$ and
$\{S_i\}_i$ is a conditional e-variable collection relative to filtration $(\cG_n)_n$.   
We will say that such an $\S$-variable specification is {\em sequentially decomposable}, or {\em seqdec\/} for short, relative to filtration $(\cG_n)_n$; in all our examples except the t-test (Example~\ref{ex:t}) we can take $\cG_n = \sigma(Y^n)$.  Seq\-dec specifications have a direct link with the OS setting, resulting in three remarkable properties: first, assuming still that data are i.i.d.,
any study-level test martingale process we can construct via the plug-in method (see above) based on a seq\-dec specification also defines a sequence of conditional $\S$-variables and hence a test martingale at the corresponding  {\em concatenated data level\/}
$Y'_1, Y'_2, \ldots$ where $Y'_i$ is arrived at by relabeling $Y_{1,1}, Y_{1,2}, \ldots, Y_{1,N_{(1)}}, Y_{2,1}, Y_{2,2}, \ldots, Y_{2,N_{(2)}}, Y_{3,1}, \ldots$ in order (so that e.g.\ $Y'_{N_{(1)}+N_{(2)}} = Y_{2,N_{(2)}}$). This defines a concatenated data-level filtration   $(\cG'_t)_t$ with
 $\cG'_t =  \sigma(Y'_1,\ldots, Y'_t)$. 
The corresponding sequence of conditional $\S$-variables is then given by $S'_1, S'_2, \ldots$ where, for $T_{(0)} := 0$, $m \in \naturals$, $1 \leq i < N_{(m)}$, $T_{(m)} := \sum_{j=1}^{m} N_{(j)}$, we set $S'_{T_{(m)} + i} := s_i(Y_{m,1}, \ldots, Y_{m,i})$.
This means that besides engaging in optional continuation, we can also safely do optional stopping at this concatenated-data level, since the Ville-Robbins inequality holds at this level by Corollary~\ref{cor:villerobbins}.

Second, we can use seqdec specifications to extend the plug-in method  (which required constant stopping times $N_{(m)}$) to prescribe conditional E-variables for stopping times $N_{(m)}$ that are not constant given $\bY^{(m-1)}$: 
for arbitrary $N_{(m)} \in \cN_{(m)}$, 
we set $E_{(m)} :=  \prod_{i=1}^{N_{(m)}} s_i(Y_{m,1}, \ldots, Y_{m,i})$. By construction, this reduces to the plug-in method whenever $N_{(m)}$ is constant given $\bY^{(m-1)}$, so it is a proper extension, and it follows as a direct corollary of the fact that $E_{(m)}$ can be rewritten as 
$\prod_{i=1}^{N_{(m)}} S'_{T_{(m-1)}+i}$, i.e.\ a product of factors in a test martingale, that any $E_{(m)}$ constructed in this manner for any seqdec specification is a $\cF_{(m-1)}$-conditional $\S$-variable.

Third, seqdec specifications allow for alternative ways  to create study-level processes from $\S$-specifications beyond the plug-in method used thus far. For example, we can set 
$E_{(m)} := \prod_{i=1}^{N_{(m)}} s_{T_{(m-1)} +i}(\bY^{(m-1)},Y_{m,1}, \ldots, Y_{m,i})$ 
for arbitrary stopping times $N_{(1)}, \ldots, N_{(m)}$. 
Once again, this also defines a martingale at the concatenated data-level  --- 
it is simply the martingale that arises if we view the $m$ studies as one single, long sequence of $T_{(m)}$ data points.  We call this the {\em sequential\/} application of the $\S$-variable specification --- see Figure~\ref{fig:seqdec}.

\begin{ownexample}\label{ex:t}
{ \normalfont All GRO-type specifications based on a simple $\cH_0 = \{P_0\}$ are likelihood ratios $s^{[i]} = q(Y^i)/p_0(Y^i)$, and hence will be seqdec and can thus be combined with optional stopping. In Example~\ref{ex:oc}, choosing $W_{(m)} := W_{(1)} \mid \bY^{(m-1)}$ to be the Bayesian posterior corresponds to the sequential application of the $W_{(1)}$-GRO specification of Section~\ref{sec:BayesGRO}; using $W_{(m)} := W_{(1)}$ corresponds to the plug-in application of the $W_{(1)}$-GRO application. 
This illustrates that we may think of our GRO criteria not as prescribing a single choice $E_{(m)} \in \cE_{N_{(m)}}$, but rather as suggesting to choose $E_{(m)}$
from a preferred subset $\cE'_{N_{(m)}}$ of $\cE_{N_{(m)}}$; 
the end-user may then pick any e-variable in $\cE'_{N_{(m)}}$. For example, in the case of Example~\ref{ex:oc} with simple $\cH_0$, we may further specify a set of distributions $\cW_1 \mid \Yb^{(m-1)}$ on 
$\Theta_1$ that we deem `reasonable' given previous outcomes $\Yb^{(m-1)}$, which may include the full Bayesian posterior, the originally used prior, combinations of these, tempered posteriors and so on; and we may then suggest the set $\cE'_{N_{(m)}}$ of all e-variables for the $m$-th study based on a prior in $\cW_1 \mid \Yb^{(m-1)}$.

For the (RE)GRO(W)-specifications with composite null, one immediately verifies that those of Proposition~\ref{prop:dominance} and Example~\ref{ex:ponential}  
also are seqdec, since they equal the likelihood ratio between the same two distributions $P_{\delta^+}$ and $P_{\delta^{-}}$ irrespective of $n$. The t-test GROW/REGROW e-variables for arbitrary prior $W[\delta]$ on $\delta$ as in (\ref{eq:marginalS}) are  seqdec as well, but to formalize this statement we have to extend the setting. In general, the seqdec definition makes sense for every filtration $(\cG_n)_n$ with $\cG_n = \sigma(V^n)$ where $V_n = v_n(Y^n)$ for some sequence of functions $v_1, v_2, \ldots$ defined on $\cY^1, \cY^2, \ldots$ respectively. The previous definition is the special case with $v_n(Y^n) = Y_n$.  In the t-test example we can take $V_n$ as in Section~\ref{sec:group} such that $v_n(Y_{1}, \ldots, Y_{n}) = Y_{n}/|Y_{1}|$.
Let us illustrate how, with this coarser filtration, we can still apply the plug-in method for non-constant stopping times $N_{(m)}$. 
For this, we also have to coarsen the filtrations relative to which the $N_{(m)}$ are defined : the set of allowed stopping times  $\cN_{(m)}$ is now restricted
to lie in $(\cF_{(m-1)} \cup  \sigma(V_{m,1}, \ldots, V_{m,n}))_{n}$ with $V_{j,n} = v_{j,n}(Y_{j,1}, \ldots, Y_{j,n})$ for some collection of functions $(v_{j,n})_n$ (recall that before they were members of 
$(\cF_{(m-1)} \cup  \sigma(Y_{m,1}, \ldots, Y_{m,n}))_n$)
. For the $t$-test example we set $v_{j,n}(Y_{j,1}, \ldots, Y_{j,n}) = Y_{j,n}/|Y_{j,1}|$.
The study-level filtrations $\cF_{(1)} :=\sigma(Y_{(1)}), \cF_{(2)} := \sigma(Y_{(1)},Y_{(2)}), \ldots$ remain unchanged and do not hide any information in the $Y_{(j)}$. 
 In practice the restriction of $\cN_{(m)}$ will not be of much concern since `most' stopping times are still allowed, including  
 the most aggressive stopping rule: stop the $m$-th study at the smallest $n'$ such that $\prod_{i=1}^{n'} s_i(V_{m,i}) \geq 1/\alpha_{(m)}$, where $\alpha_{(m-1)}$ is some threshold that is allowed to depend on $\bY^{(m-1)}$. 
We can also allow for the sequential (rather than plug-in) application of the t-test e-variable specification so that effectively we view all studies as subsequent outcomes of a single study,  by restricting the filtrations in a slightly different way; we omit the details. We can even let the choice between a sequential or plug-in choice for $E_{(m)}$ depend on past data, but this requires further generalizations of the $\cF_m$ and $(v_{j,n})_n$ definitions that we shall not pursue here. 
}\end{ownexample}
Summarizing, the practical setting  we have in mind when we speak about OS and OC respectively is quite different: OC concerns study-level martingales constructed by deciding,  on the fly , after the $m-1$-st study, whether to continue to the $m$-th study and if so, what new $\cF_{(m-1)}$-conditional $\S$-variable to take from the set  $\cE_{(m)}$ of possible $\S$-variables  of use. OS is about  data-level martingales with only a stop/continue choice. Nevertheless, the formal definitions of (Type-I error) `safety under OS' and  `safety under OC' only differ in that `study-level' is  replaced by `data-level'. We may say that {\em combining e-variables $E_{(m)}$ by multiplication is always Type-I error safe under OC. If the e-variable prescription used to construct  $E_{(m)}$ has the seqdec property, then the stopping times $N_{(m)}$ used in each study do not need to be specified before the study starts and can even be externally imposed, so that we have Type-I error safety not just under OC but also under OS within each individual study.}

There is one final subtlety to consider: in the OS setting, with a single stream of data $Y_1, Y_2, \ldots$ and conditional e-variables $S_{1}, S_2, \ldots$ and test martingale $S^{[1]}, S^{[2]}, \ldots$, the Ville-Robbins inequality (\ref{eq:villerobbins}) implies that our Type-I error bound $\alpha$ is guaranteed no matter when we stop --- in particular, the actual stopping time does not have to be taken relative to the filtration $(\cG_n)_n$ --- we may even peek into the future to decide whether to stop now. This suggests that our care in specifying the correct filtrations for the t-test was unnecessary --- it seems we can use any stopping rule we like! 
But this becomes incorrect once we move from OS at the data-level to OC at the study-level: if, in the t-test setting with the plug-in construction of the $\S$-variable $E_{(m)}$ for the $m$-th study, we were to set the $N_{(m)}$ so that they are not stopping times relative to $(\cF_{(m-1)} \cup \sigma(V_{m,1}, \ldots, V_{m,n}))_n$ but only relative to the more refined $(\cF_{(m-1)} \cup \sigma(Y_{m,1}, \ldots, Y_{m,n}))_n$  , we could end up creating fake conditional $\S$-variables at the study-level, i.e.\ so that $\Exp_{P_0}[E_{(m)} \mid \cF_{(m-1)}] > 1$ for all $m$ (\cite[Appendix B]{perez2022estatistics} constructs such a random variable for the t-test). And then the Ville-Robbins inequality may {\em not\/} hold any more at the study-level, and we loose the Type-I error guarantee under optional continuation.
\paragraph{Local vs. Global GRO}
Now consider  two data streams, $Y_{1,1}, \ldots, Y_{1,n_1}$ and $Y_{2,1}, \ldots, Y_{2,n_2}$ of fixed lengths $n_1$ and $n_2$. We may alternatively model these two streams as a single stream  $Y'_{1}, \ldots, Y'_{n_1+n_2}$ of length $n_1+ n_2$.
If we use an $\S$-variable specification with the seqdec property to generate a study-level test martingale in the sequential way (i.e.\ not the plug-in way) described above, we will get that, $E^{(2)} = E_{(1)} \cdot  E_{(2)}$ constructed for the first two data streams
(with $E_{(1)} = \prod_{i=1}^{n_1} s_{i}(Y^i)$ and $E_{(2)} = \prod_{i=1}^{n_2} s_{n_1+i}(Y'^{n_1+i})$) coincides with $E'_{(1)} = \prod_{i=1}^{n_1 + n_2} s^{[i]}(Y'_i)$ constructed for the single alternative stream. We may thus say that the {\em sequential application\/}  of seqdec $\S$-specifications is always {\em coherent\/}: 
applying the specification sequentially-`locally' (separately for both studies) or sequentially-globally (for the concatenated data viewed as one study) gives the same result. For example, the Bayesian $W$-GRO specification of Section~\ref{sec:BayesGRO}, the GROW specification of Proposition~\ref{prop:dominance} and Example~\ref{ex:ponential} and the GROW and REGROW specifications in the t-test example are all seqdec and hence all have this coherence property when applied sequentially.
Sometimes the sequential application of an E-prescription is not feasible or desirable; for example, not all details of previous data may be known. We may then prefer the plug-in application of the E-prescription. Unfortunately, 
the seqdec property is not sufficient to get coherence for the plug-in method: clearly, if we use the same prior $W_{(1)}$ as prior in the Bayesian Example~\ref{ex:normal} for the first and the second study, this leads to different $E^{(2)}$ and $E'_{(1)}$, the latter being equivalent to using the posterior of the first study as prior in the second. A sufficient condition for a plug-in application to satisfy coherence after all is that it satisfies both the seqdec property, and further that $s_i$, with $S_i = s_i(Y^i)$ as in the definition of seqdec, can be rewritten as $s_i(Y^i)= s'(Y_i)$ for a single function $s'$, for all $i$. Then in fact the plug-in and the sequential application of the $\S$-prescription will coincide, and coherence is guaranteed. This  happens in the subset of our examples in which $s^{[i]}(Y^i)$ takes the form $q(Y^i)/p(Y^i)$  for the same $p$ and $q$, for all $i$, as happens e.g.\ in Example~\ref{ex:ponential}.
\paragraph{An Open Question concerning GRO} 
In practice we may very well be in a situation in which OS at the data-level is desirable (see the next section for why it would be), so we want to use a seqdec specification, yet the GRO criterion we are interested in does not give one --- Example~\ref{ex:regrow} illustrates this for the $2 \times 2$ case. We may then try the following approach, which for simplicity we only describe for the REGROW criterion: let $\eval^f_n$ be the REGROW $\S$-variable (\ref{eq:regroweval}) achieving (\ref{eq:jiprb}) with $f(\theta) = \gro(\theta)$ for samples of size $n$. We try to find a sequence of $\S$-variables $E_1, E_2, \ldots$  such that $E_i$ is a $\sigma(Y^{i-1})$-conditional $\S$-variable for $Y_i$ and the product $\S$-variable $E^{[n]} := \prod_{i=1}^n E_i$ achieves (\ref{eq:jiprb}) to within some fixed  $\epsilon$ for all $n$ larger than some minimal $n_0$, i.e.\
\begin{equation}\label{eq:rug}
\inf_{\theta \in \Theta_1} 
(\ 
\Exp_{\Ydata \sim P_\theta}[\log \eval^{[n]}]- \gro(\theta) \ ) 
\geq 
\inf_{\theta \in \Theta_1} (\ \Exp_{\Ydata \sim P_\theta}[\log \eval^f_n] - \gro(\theta) \ ) - \epsilon.
\end{equation}
By construction, the sequence  $E^{[1]}, E^{[2]}, \ldots$ is seqdec and allows for optional stopping, and if we can find $E_1, E_2, \ldots$ such that $\epsilon$ is small for all $n$ larger than or equal to the $n_0$ corresponding to the smallest sample we'd ever be interested in analyzing, we can say that the full process (and not just an instance at a fixed $n$) is `almost' REGROW in the desired sense.
\begin{ownexample}\label{ex:regrow} { \normalfont
\cite{TurnerLG21} 
successfully use this idea for the $2 \times 2$ model. We illustrate this confining ourselves for simplicity to a stream of paired data, i.e.\ $X_1 =a, X_2 = b, X_3=a, X_4= b,\ldots$. First, we note that directly applying the idea above will not work. To see this, consider the simple alternative $\Theta_1 = \{(\mu_{1|a}, \mu_{1|b}) \}$. According to the composite null, the $Y_i$ are i.i.d.\ Bernoulli with parameter $\mu\in [0,1]$, but the $\grow(\Theta_1) = \gro((\mu_{1|a},\mu_{1|b}))$- $\S$-variable for such a $\cH_0$ and single data point $Y_i$ is the trivial $E \equiv 1$. Thus we would get all $E_i$ equal to $1$, and zero growth. However, if we analyze the data in batches of size $2$, so set
	$Y'_1 = (Y_1,Y_2), Y'_2 = (Y_3,Y_4)$ and take as $E'_i = {\rm e}(Y'_i)$ the (nontrivial) $(\mu_{1|a},\mu_{1|b})$-GRO-$\S$-variable for $Y'_i = (Y_{2i-1},Y_{2i})$ then  $E'^{[n]} = \prod_{i=1}^n E'_i$
	is
	the $(\mu_{1|a},\mu_{1|b})$-GRO-$\S$-variable for all $n$: we have the seqdec property and coherence, both for the sequential and for the plug-in method of applying the $(\mu_{1|a},\mu_{1|b})$-GRO-prescription.  
Now in practice we want to consider a composite alternative --- say we consider the full  alternative $\Theta_1 = [0,1]^2$. Then the REGROW prescription will not be seqdec - the prior $W^*_1$ in (\ref{eq:regroweval}) depends on the sample size $n$. However it turns out that for a particular choice of prior $W$ (\citet{TurnerLG21} find it to be the beta-prior with parameters $\alpha=\beta=0.18$) we have the following: if, for all $i$, we take $E'_i$ the $W \mid Y'^{i-1}$-GRO $\S$-variable, $W \mid Y'^{i-1}$ being the posterior based on $Y'^{i-1}$, then
	we numerically find that $E'^{[n]} = \prod_{i=1}^n E'_i$ is, 
	for all but the smallest $n$, very close to the REGROW $\S$-variable $\eval^f_n$ for that $n$, i.e.\ it achieves (\ref{eq:rug}) for small $\epsilon$.  }
\end{ownexample}
The example raises an important question: under what conditions (on model, minimal batch sizes and the like) can we create a seqdec specification that behaves optimally for our desired GRO criterion (as in Section~\ref{sec:BayesGRO} with $W$-GRO, and in Example~\ref{ex:ponential}, with the GROW criterion) or almost optimally (as in the $2 \times 2$ example above with batch size 2, with the REGROW criterion)? 

\section{Competitiveness: GRO and Power}
\label{sec:competitive}
What sample size should we minimally plan for in a study so that we may expect a useful result? The answer depends on whether one looks at $\S$-values purely as measures of evidence, without an accept/reject decision attached, or whether one considers such decisions after all. In the latter case, we can ask, more generally, how competitive e-value based tests are, in terms of required sample size, compared to the standard fixed-sample size Neyman-Pearson approach. We consider the cases without and with accept/reject decisions in turn.
\paragraph{$\S$-Values as Evidence} $\S$-values may be viewed simply as a measure of evidence, extending the evidential interpretation of likelihood ratios \citep{Royall00}. They are then certainly competitive in every sense: for simple $\cH_0$  they coincide with likelihood ratios and Bayes factors, and will give thus as much evidence as these notions do; for composite $\cH_0$, GRO(W) $\S$-variables are designed to give as much expected log-evidence against $\cH_0$ as possible without violating the optional continuation requirement --- in practice in some cases giving a bit more, and in some cases a bit less evidence against the null than standard Bayes factors
(see \cite{TurnerLG21} for a practical example).

Now suppose we have a minimal effect size $\delta$ in mind, and we plan a study in which obtaining data is expensive. What sample size should we plan for? One option is to
pick a certain {\em target growth\/} $L$ (essentially the logarithm of  \cite{Shafer19}'s notion of ``implied target'') and determine the sample size at which we expect to gain $L$. To illustrate, consider the 1-dimensional exponential family case of Example~\ref{ex:ponential} with $\Theta_0 = \{0 \}$. We know that, for a sample of size $n$ , under all $\theta_1 \in \Theta_1$ with $\Theta_1 = \{\theta_1: \theta_1 \geq \delta \}$, we have $\grow(\Theta_1)=  n D(P_{\delta} \| P_0) $ where $D(P_{\delta} \| P_0)$ is the KL divergence for 1 outcome. We then calculate $n_{\grow}$ as the smallest $n$ such that $n D(P_{\delta}\| P_0)\geq L$, i.e.\ $n_{\grow} = \lceil L/D(P_{\delta} \| P_0) \rceil$. In the Gaussian location model, $D(P_{\delta}\| P_0) = \delta^2/2$, so $n = \lceil 2 L/\delta^2 \rceil$. We return to the question of choosing $L$ below. 

\paragraph{$\S$-Values for Decisions}
We can also use $\S$-values in the traditional setting, in which a study ends with an accept/reject decision --- with the proviso that any decision is provisional, since there always is an option to continue and combine the results with a new study. 
For better or worse, this is the paradigm that researchers often have to work in, and within this paradigm they will inevitably be interested in the power for the experiment ahead. They will then plan for a certain sample size $n$ to achieve such power, with a minimal relevant effect size $\delta$ in mind. As long as the $\S$-variables themselves are chosen according to a GRO criterion, such a use of power as a `secondary' criterion used merely to determine sample size is consistent with the GRO approach.
In order for $\S$-variables to be embraced by practitioners, we would hope that the sample sizes required to achieve a certain desired power  with GRO-$\S$-variables would be competitive with the standard approach based on Neyman-Pearson tests. We now study whether this is the case. 
For simplicity we only consider the Gaussian location model of Example~\ref{ex:normal}, where $\cH_0$ is the standard normal $N(0,1)$ and $\cH_1$ the set  $\{p_{\mu} : \mu \in \Theta\}$ of normals with variance $1$. All results readily generalise to  1-dimensional exponential families.

\paragraph{Power: planning for a Fixed $n$}
For comparison, recall that a standard one-sided NP test at level $\alpha$ would reject if $\hat\mu \geq z_{\alpha}/\sqrt{n}$ with $z_{\alpha}$ the $(1-\alpha)$-quantile of the standard normal with  $z_{0.05} = 1.645, z_{0.01} = 2.33$.  
By standard calculation (see Appendix~\ref{app:brown}), under an alternative with mean  $\geq \delta$, the sample size needed with this test to get power at least $1-\beta$ satisfies $n_{\textsc{np}} =  C_{\textsc{np}}/\delta^2$ with $C_{\textsc{np}} =(z_{\alpha} + z_{\beta})^2$; for $\alpha = 0.05, \beta = 0.2$ we get 
$C_{\textsc{np}} \approx 6.180$. 
For the same $\Theta_1 = \{\mu: \mu \geq \delta\}$,   we can also calculate the sample size needed to get power $1-\beta$ using the GROW $\S$-variable of Example~\ref{ex:ponential}. If we use a fixed sample size $n$, we reject if $\log p_{\delta}(Y^n)/p_0(Y^n) \geq -\log \alpha$. By a simple calculation, for $\alpha > 1/2$, the smallest $n$ at which we have power  at least $1-\beta$, is given by setting
\begin{equation}
\label{eq:competition}
n_{\textsc{grow-fixed}} =  2 \cdot\frac{ - \log\alpha}{\delta^2} \cdot 
\left(1 + \frac{z_{\beta}}{
	z_{\alpha}}
\right)^2 = 
c_{\alpha} n_{\textsc{np}} \text{\ with\ } c_{\alpha} = \frac{2 \cdot ( - \log\alpha)}{z_{\alpha}^2}.
\end{equation}

\noindent
We have $c_{0.05} \approx 2.2$; $c_{0.01} \approx  1.7$ and $c_{\alpha}$ very slowly converges to $1$ in the limit $\alpha \downarrow 0$: 
up to a constant factor of about two we need the same amount of data as in a classical approach, and the width of the induced confidence interval is of the same order. 
We can therefore choose a GROW $\eval^*$ that is
qualitatively more similar to a standard NP test than a standard Bayes
factor approach. Using instead a standard Bayesian prior $W_1$ on $\Theta_1$ with the $W_1$-GRO $\S$-variable has the advantage of not needing to specify any $\delta$ in advance, but  the number of samples to get power $1-\beta$ is larger by a logarithmic factor (Appendix~\ref{app:brown}).
\begin{figure}[]
	\centering
	\includegraphics[width=0.4\textwidth]{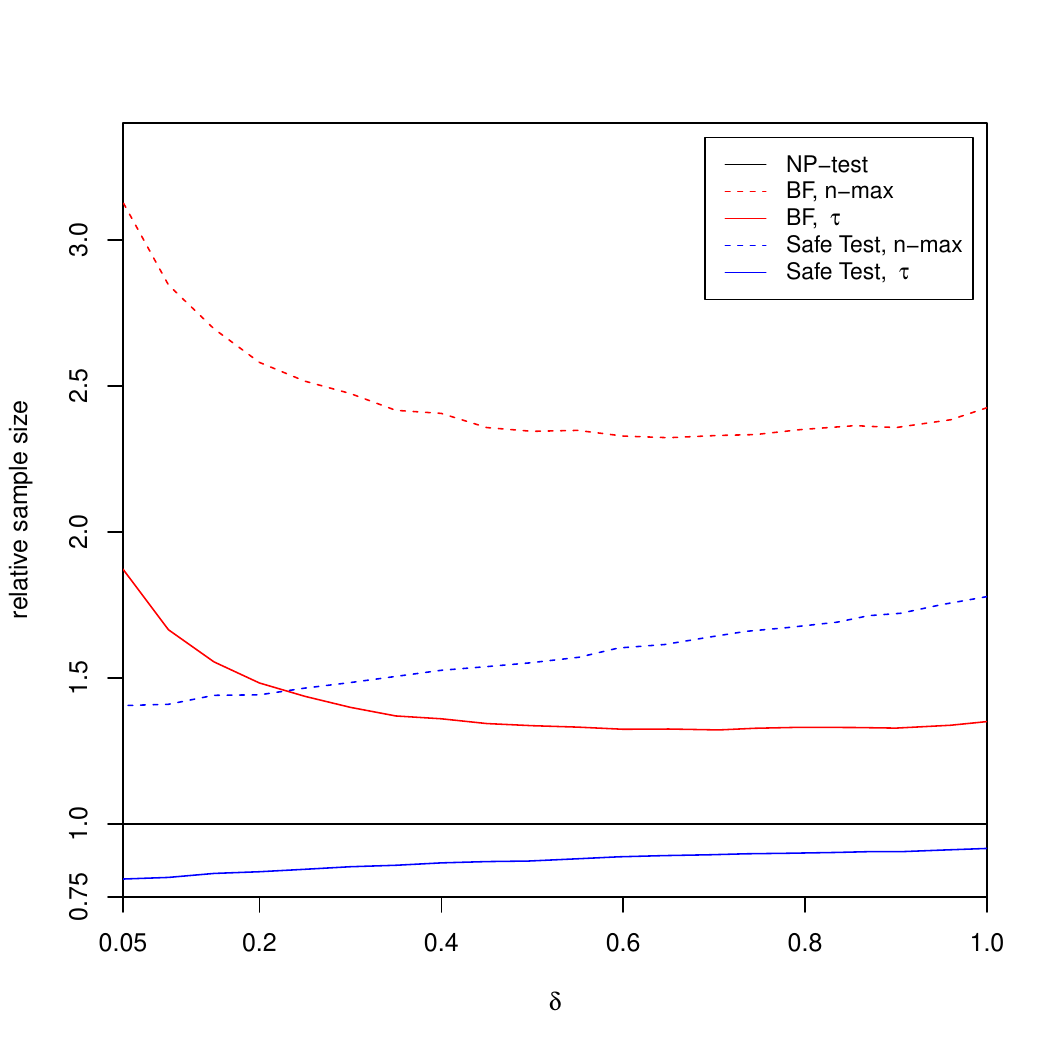}
	\caption{Factor of additional data needed to obtain a power of $1- \beta = 80\%$ compared to a Neyman-Pearson $z$-test, as a function of effect size (mean)  for the Bayesian $\S$-variable as in Example~\ref{ex:normal} with normal prior and the GROW $\S$-variable with minimum relevant effect size $\delta$ as in Example~\ref{ex:ponential} under optional stopping,  both in expectation ($\tau_{0.8}$) and in  worst-case ($n_{\max}(0.8)$) (for very small and large $\delta$, the normal prior we used for the Bayesian $\S$-variable may not be the prior of choice, but the results are representative for other priors one might use at $\delta \approx 0.5$).}\label{fig:samplesizes}
\end{figure}

The evidential target growth and the maximal power approach are not contradictory: for any particular choice of $\alpha$ and $\beta$, there is a choice of $L$ such that the planned-for sample sizes become the same function of $\delta$ (but of course the $L$ resulting from $\alpha = 0.05, \beta = 1-0.8$ will be just as arbitrary as these choices were in the first place). 

\paragraph{Power with Optional Stopping --- a Tragedy of the Commons?}
No matter the considerable advantages of being safe under optional continuation, the factor of about $2$ of extra data needed to get a desired power might scare away practitioners from adopting the $\S$-variable approach. 
The situation changes completely once one adopts optional stopping. As we saw in the previous section, many testing problems allow us to use $\S$-variables that remain safe under optional stopping --- and we can use the most aggressive stopping rule  that stops as soon as either $E_n \geq - \log \alpha$ (and we reject) or a pre-set maximum $n_{\max}$ is achieved (and we reject if $E_{n_{\max}} \geq - \log \alpha$ and otherwise accept). A simple but quite accurate approximation of the resulting stopping time $\tau_1$ for i.i.d.\ data in the GROW setting of Example~\ref{ex:ponential}, when setting $n_{\max}$ to $\infty$ is given by using Wald's equality in a manner first set out by \cite{breiman1961optimal}; \cite{TerschurePLG21} give details.
It gives for data  $Y_1, Y_2, \ldots \sim P_{\delta}$, that  $\Exp_{P_{\delta}}[\tau_1] \approx (- \log \alpha)/D(P_{\delta}\| P_0)$ with $D(P_{\delta} \| P_0)$ the KL divergence for a single outcome. For the Gaussian location family $D(P_{\delta} \| P_{0}) = \delta^2/2$, and we get
$\Exp_{P_{\delta}}[\tau_{1}] \approx 2 (- \log \alpha)/\delta^2$. 
Comparing to (\ref{eq:competition}), this gives that with $n_{\max} = \infty$ (so that the power of our test is $1$), the expected stopping time will  be in fact already smaller than the fixed stopping time we get with the Neyman-Pearson approach at power $1-\beta$ set to $0.8$. In practice we will choose $n_{\max}(\beta)$ to be the smallest number so that the overall procedure has power $1-\beta$, resulting in a stopping time $\tau_{\beta} = \min \{ \tau_1, n_{\max}(\beta)\}$. 
The expected stopping time is then really even smaller. Figure~\ref{fig:samplesizes} demonstrates this for the Gaussian location family. As the figure illustrates for the case $\beta = 0.8$, we have  $n_{\max}(\beta) = C_{\beta,\delta}/\delta^2$ and $\Exp[\tau_{\beta}] = C'_{\beta,\delta}/\delta^2$ for $C_{\beta,\delta}$ and $C'_{\beta,\delta}$ that remain within constant bounds as $\delta$ varies.  
We can in fact heuristically derive analytic expressions (integrals) for the limits $C_{\beta,0}$ and $C'_{\beta,0}$ for $\delta\downarrow 0$ by rescaling the log-likelihood process $(\log p_{\delta}(X^n)/p_0(X^n))_{n}$ to become compatible with a Brownian motion with drift, see Appendix~\ref{app:brown}. These give, for $\beta = 0.2$, $C_{\beta,0}= 8.5936$ and $C'_{\beta,0} = 4.971$ in accordance with Figure~\ref{fig:samplesizes}. 
Experiments with a logrank test \citep{TerschurePLG21}, $t$-test (in the vignette of $R$ package \citep{LyT20}) and  $2 \times 2$ tables \citep{TurnerLG21} all confirm the picture that arises from Figure~\ref{fig:samplesizes}: with $\S$-variables based on optional stopping one needs on average {\em less\/} data to achieve a certain desired power, but one needs to prepare for {\em more\/} data in the worst-case. 
Taking stock, we can conclude that if current standard null hypothesis tests were replaced by $\S$-value-based tests, and the standard practice to determine study sizes were replaced by the one above, and the percentage of studies in which the alternative is true is not too small, the world would need on average  about the same or even a bit less data than it does now, to reach substantially more robust conclusions and better meta-analyses. Yet --- at least as long as scientists insist on power and significance requirements  ---  each individual study would have to {\em plan\/}  for substantially more data, giving researchers an incentive not to adapt these new methods. We see this {\em Tragedy of the Commons\/} as one of the biggest obstacles for uptake of $\S$-variables in practical settings.

\section{Earlier and Related Work}
\label{sec:related}
\paragraph{$\S$-Variables, Test Martingales, Information Projections, General Novelty}
As seen in Section~\ref{sec:os}, $\S$-variables are the building blocks of {\em test (super-) martingales}, 
which go back to \cite{Ville39}.  E-variables themselves have
probably been originally introduced by Levin (of $P$ vs $NP$ fame)
(\citeyear{Levin76}) (see also \citep{gacs2005uniform}) under the name {\em
	test of randomness}, but Levin's abstract context is quite different
from ours. Independently discovered by \cite{zhang2011asymptotically} (under the name {\em PBR (prediction-based ratio)\/})
they were later analyzed by
\cite{shafer2011test} (calling them, with hindsight confusingly, {\em Bayes factors\/}); \cite{ShaferV19,VovkW21} ({\em e-variables/values\/}) and \cite{Shafer19} ({\em bets/betting scores\/}) --- we originally called them {\em $S$-values\/} ourselves. The literature seems to converge to $\S$-variables and -values. Here the $\S$  may either stand for evidence or for expectation.

Test martingales themselves have been thoroughly investigated by 
\cite{shafer2011test,ShaferV19}. They themselves underlie AV
(anytime-valid) $\p$-values \citep{johari2015always}, AV tests (which we call `tests that are safe for optional stopping')
and AV confidence sequences. The latter were recently developed in
great generality by A.\ Ramdas and collaborators; see
e.g.\ \citep{balsubramani2016sequential,howard2018uniform}.
Both AV tests and confidence sequences have first been developed by
H.\ Robbins and his students
\citep{darling1967confidence,lai1976confidence,robbins1970statistical}.
Like we do for $\S$-variables, Ramdas et al.\ (and also
e.g.\ \cite{pace2019likelihood}) stress the promise of  AV notions
for a safer kind of statistics that is significantly more robust
than standard tests and confidence intervals.

Just like regular tests can be turned into confidence intervals by
varying the null and `inverting' the resulting tests, AV confidence sequences can be created by starting with a collection of test
martingales, one for each null, and then varying the null (it is sometimes claimed that problematic aspects of  null hypothesis testing are mostly due to the very idea of a `null hypothesis' or a significance level 
	\citep{Cumming12,McShane19}. Without wanting to take sides in this issue, we note that standard confidence intervals are just as unsafe under optional continuation as standard 
	Neyman-Pearson hypothesis tests). The work on AV tests and confidence
sequences is therefore very similar in spirit to ours, with our work stressing
analysis at the level of batches of data rather than individual data points, and with the AV work bringing out the difference to Bayesian approaches more explicitly (AV $1-\alpha$-confidence intervals are typically wider than Bayesian $1-\alpha$-credible sets).  
In fact we do not claim any real novelty for
the `safe' or `AV' setting per se:  the real novelty of this paper is in the four versions of 
Theorem~\ref{thm:maina}. As far as we know, these results are new, with the exception of a  special case of the simplest version of Theorem~\ref{thm:maina} (Section~\ref{sec:GRO}): the case of discrete outcome spaces, simple $\cH_1$ and convex $\cH_0$ was already formulated and
proved
by \cite{zhang2011asymptotically}. Theorem~\ref{thm:maina} heavily builds on properties of standard- and reverse and joint information projections about which there is a rich literature, key references being \cite{Csiszar75,Topsoe79,CsiszarT84,csiszar2003information}. Both the standard and the reverse information projection are special cases of {\em $F$-divergence projections}, investigated in great detail in both the information-theoretic community \citep{Rueschendorf84} and the mathematical finance community \citep{Gundel05,FollmerG06}; the robust optimization problems in the latter two papers, when instantiated to logarithmic utility, bear some resemblance to our GROW criterion. While the details and the motivation are quite different, it would be of interest to study the connections further.

\paragraph{Relation to Sequential Testing}
{\em Sequential testing\/} \citep{Lai09,siegmund2013sequential}, pioneered by \cite{Wald47} and
Barnard, is
mathematically very similar to, but conceptually quite different from, testing based on test martingales and
(therefore) $\S$-variables.  Sequential tests are made for streams of data $Y_i, Y_2,\ldots$ as in Example~\ref{ex:oc} and Section~\ref{sec:GROanalysis} and are
based on random
processes $(S_i)_{i \in \naturals}$ such that, for each $i$,  $S_i$ is a conditional $\S$-variable given $Y_1, \ldots, Y_{i-1}$ under $\cH_0$, and $1/S_i$ is a conditional $\S$-variable given $Y_1, \ldots, Y_{i-1}$ under $\cH_1$.  
Of course, this two-sided $\S$-variable property only holds in quite special cases --- roughly under the same conditions as Proposition~\ref{prop:dominance} (monotone likelihood ratio), i.e.\ in our Example~\ref{ex:ponential} and for the $t$-test with point prior on $\delta^+, \delta^-$. 
In such a setting, the sequential test
based on $S_1, S_2, \ldots$ with prespecified parameters
$\alpha,\beta$ proceeds by calculating $S_1, S_2, \ldots$ and stopping
at $\tau^*$, the smallest $\tau$ at which either $S_{\tau} \geq
(1-\beta)/\alpha$ (`accept') or $S_{\tau} \leq (1-\alpha)/\beta$
(`reject'). Wald showed that this test has Type I error probability
bounded by $\alpha$ and Type II error bounded by $\beta$. The reason one
can stop at a smaller threshold ($(1-\beta)/\alpha$ rather than $1/\alpha$) is that one {\em has\/} to stop at
$\tau^*$, Thus, the method does not allow for optional stopping in our
sense: conceptually, sequential tests were designed for special, pre-specified stopping times. 
Still, much work in sequential testing
can be re-cycled to obtain test martingales and $\S$-values --- but not always vice versa since  $\S$-variables are often not `two-sided'.  

\paragraph{Related Work on Relating $\p$-values and $\S$-variables}
\cite{shafer2011test} and \cite{ShaferV19}  give a general formula for {\em
	calibrators\/} $f$ (see also \cite{VovkW21} and  \cite{vovk93,sellke2001calibration} for early work in this direction). These are decreasing functions $f: [0,1]
\rightarrow [0,\infty]$ so that for any $\p$-value $\pval$, $\eval\coloneqq f(\pval)$ is an
$\S$-variable. The choice of any such calibrator is essentially arbitrary, but, following \cite{Shafer19}, let us consider one that is especially simple: $f(p) = 1/\sqrt{p} - 1$. For example, for any calibrator $f$ suggested for practice, 
rejection under the $\S$-variable based test with significance level $\alpha = 0.05$,
so that $\eval \geq 20$, would then correspond to reject only if $\pval
\leq f^{-1}(20) = 1/441=0.0023$, 
requiring a substantial amount of additional data for rejection under
a given alternative.  Note that the $\S$-variables we developed for
{\em given\/} models in previous sections are more sensitive than such
generic calibrators though. For example, consider the normal location family of the previous section.
With the calibrator above, we would reject if $\hat\mu \geq  z_{0.0023}/\sqrt{n} \approx 2.8/\sqrt{n}$. The amount of data to plan for to obtain power $80\%$ would then be $\approx (2.8/1.65)^2 n_{\textsc{np}} \approx 3.0 n_{\textsc{np}}$, whereas for the $\S$-value based on the normal likelihood ratio we would need $\approx 2.2 n_{\textsc{np}}$, and even significantly less under optional stopping.
\section{GRO: Discussion and Open Problems}
\label{sec:GROdiscussion}
In this paper we provided several motivations for our various GRO criteria as we introduced them, in Section~\ref{sec:GRO}--\ref{sec:regrow}. Here we reflect on the strength of our arguments in sequence-of-studies settings when taking the running products of the per-study GRO e-variables.
Let us first consider a simple alternative $Q$ as in Section~\ref{sec:grosimple}, so that GROW and REGROW criteria coincide with GRO.
In an optional continuation (OC) context, 
the justification of GRO is strongest if all study outcomes ${\bf Y}_{(1)}, {\bf Y}_{(2)}, \ldots$ are independent
and the variances of $\log E_{(m)}$ under $Q$ for the GRO e-variables $E_{(m)}$ are not too large.  As pointed out by a referee, if these variances are large, there could be prolonged periods of `draw-downs'  --- in the gambling interpretation, independence ensures, by the law of large numbers, that asymptotically the GRO e-variables maximize our capital; but if there is high variance, there may be several studies in a row during which the product of all e-values so far remains low, a fact well-known among economists. 
Moreover, i.i.d.\ data and the seqdec property are required for another central justification of GRO in the optional stopping setting, namely Breiman's insight that the expected stopping time before one can reject is minimized by maximizing expectation of the log capital (\cite{TerschurePLG21} explain this in detail).
%
This issue may perhaps in some cases be resolved by instead of adopting the `global' GRO e-variable (among all e-variables for the null), taking an alternative e-variable that is GRO among a subset of e-variables with sufficiently low variance under $Q$. Such low variance e-variables would presumably look quite different from likelihood ratios and Bayes factors though, since for given $\Theta_0$, the only e-variable that can be obtained by projecting on $\cW(\Theta_0)$ is the unconstrained GRO e-variable:
e-variables that are optimal in a non-GRO sense cannot be obtained by projection, even if KL is replaced by another distance or divergence $D'$: Corollary~\ref{cor:justone} implies that, with $W^*_0 = \arg \min_{W) \in \cW(\Theta_0)} D'(Q; P_W)$, we have that $E'= q(\bY)/p_{W^*_0}(\bY)$  {\em does not\/} give an E-variable unless $W^*_0$ also minimizes the KL divergence, i.e.\ if $E'$ is an E-variable at all, it must be GRO.

In case $\cH_1$ is composite, then even if the ${\bf Y}_{(j)}$
are independent and the variances are low, there may be alternatives to GROW and REGROW that are sometimes preferable. For example, if we are in an OS setting and we use our tests as the basis for always-valid confidence intervals (see the previous section), we may want to aim for the conditional e-variables that lead to the  confidence intervals  that shrink to $0$ at the fastest possible rate, which for regular 1-dimensional parameters of interest is usually $O((\log \log n)/n)$ \citep{howard2018uniform}.  
These are not obtained by REGROW e-variables (which, by extending the reasoning of Example~\ref{ex:jeffreys}, in parametric problems achieve a width of $O((\log n)/n)$). The narrower  $O((\log \log n)/n)$ can be achieved by the {\em switching\/} strategies of
 \cite{erven2012catching} or the {\em stitching\/} method used to design test martingales by \cite{howard2018uniform}. 
A more precise understanding of whether such methods can also be re-understood as optimizing a variation of GRO, and more generally what meta-GRO criteria are reasonable at all, and in what situations, is needed. This includes the question of when any variation of GRO automatically provides seqdec, or close-to-seqdec specifications, allowing us to engage in OS (Example~
\ref{ex:regrow}). Answering these questions is a major avenue for future research: it ultimately determines how widely applicable  GRO criteria really are. 
\section{Could {F}isher, {J}effreys and {N}eyman Have Agreed on a {\em Curren-} {\em cy\/} for Testing?}
\label{sec:theory}
The three main approaches towards null hypothesis testing are
Jeffreys' Bayes factors, Fisher's $\p$-value-based testing and
the Neyman-Pearson method. In the paper {\em Could Fisher, Jeffreys and Neyman Have Agreed on Testing?}, \cite{Berger03} noted that, while these
methodologies seem superficially highly contradictory, there
exist methods that have a place within all three. The developments in this paper lead to the conclusion that $\S$-variable based testing --- although it differs in some technical
respects from Berger's proposals --- is very much in the same spirit:

Concerning the {\em Neyman-Pearson approach\/}: $\S$-variables lead to
tests with Type-I error guarantees at any fixed significance level
$\alpha$, which is also the first requirement of a Neyman-Pearson test --- requiring safety under optional continuation or optional stopping simply enforces the requirement to hold over a non-pre-specified sequence of studies, which is a natural requirement in scientific applications.
Since there is then no single study any more, the concept of `power' looses its centrality (and may be upgraded to requiring power one), and growth-rate optimality is a natural quantitative refinement. The fact that a high growth-rate corresponds to a high value of $\Exp_{P_1}[\log \eval]$ under $\cH_1$ whereas a high power corresponds to a high probability that $P_1(\log \eval \geq - \log \alpha)$ also shows that GRO and power remain intimately connected in $\S$-variable theory as well.

Concerning the {\em Fisherian approach\/}: here, $\p$-values are interpreted as indicating amounts of evidence against the null, and their definition does not need to refer to any specific alternative $\cH_1$.
Exactly the same holds for $\S$-values: the basic interpretation `a large $\S$-value provides evidence against $\cH_0$' holds no matter how the  $\S$-variable is defined, as long as it satisfies (\ref{eq:basic}). 
If they are defined relative to $\cH_1$ that is close to the actual process generating the data they will grow fast and provide a lot of evidence, but the basic interpretation holds regardless. In contrast to evidence based on standard $\p$-values however, (a) $\S$-based evidence has a  concrete additional interpretation in terms of money (the higher $\eval$, the more money one has gained in a game that is not favourable under the null); (b) it remains valid under optional continuation, and (c) unlike the $\p$-value, it is compatible with (provides the same evidence as) likelihood ratios do in simple-vs.-simple testing --- the one case where the use of likelihood ratio as evidence is standard.

Concerning the {\em Bayesian approach\/}: despite their monetary
interpretation, {\em all\/} $\S$-variables that we encountered can also
be written as Bayes
factors, and Theorem~\ref{thm:maina} strongly suggests that this
is a very general phenomenon. Subjective prior knowledge can be accounted for using the $W_1$-GRO $\S$-variable (Section~\ref{sec:BayesGRO}), whereas maximin optimal GROW and REGROW $\S$-variables sometimes correspond to `objective' Bayes approaches based on Jeffreys' and/or right-Haar priors. Still, there seem to be two fundamental differences: first, in a standard Bayesian analysis, one would require error guarantees and safety under OC \emph{ under the prior\/} instead of under all $P \in \cH_0$, and second, one would insist on using full, standard likelihoods --- whereas $\S$-variables may also be based on partial \citep{TerschurePLG21} or Dawid's (\citeyear{Dawid97}) prequential \citep{waudbysmith2021estimating} likelihoods rather than full likelihoods --- which then however may be combined with priors (on $\cH_1$) after all. Even though we emphasise Type-I error safety throughout, because of this generic freedom in using priors on $\cH_1$ the link to Bayesian methods remains close.

\paragraph{The Dream} With the massive criticisms of $\p$-values in
recent years, there seems to be growing consensus that, in the context of hypothesis testing, $\p$-values should either not be used at all, or at least, with utter care
\citep{wasserstein2016asa,benjamin2018redefine}. Yet otherwise, the
disputes among adherents of the three schools continue.
For example, some
highly accomplished statisticians reject the idea of testing without a
clear alternative outright; others say that such goodness-of-fit tests
are an essential part of data analysis. Some insist that significance
testing (with binary decisions) should be abolished altogether \citep{amrhein2019scientists},
others (perhaps slightly cynically) acknowledge that significance may
be silly in principle, yet maintain that journals and conferences will
always require a significance-style `bar' in practice and that therefore such
bars should be made as meaningful as possible. Finally, within the
Bayesian community, the Bayes factor is sometimes presented as a
panacea for most testing ills, while others  warn against its use, protesting, for example, against claims that Bayes factors can `handle optional stopping' \citep{HendriksenHG18}.
\textit{ Wouldn't it be nice if all these accomplished
	but disagreeing people could continue to go their way, yet would
	have a common language or `currency' to express amounts of evidence,
	and would be able to combine their results in a meaningful way?}
This is what $\S$-variables can provide: consider three tests with the
same null hypothesis $\cH_0$, based on samples $\Ydata_{(1)}$, $\Ydata_{(2)}$
and $\Ydata_{(3)}$ respectively.  The results of a GROW $\S$-variable test
aimed to optimise power on sample $\Ydata_{(1)}$ for $\delta \geq \delta^+$, an $\S$-variable test for
sample $\Ydata_{(2)}$ based on a Bayesian prior $W_1$ on $\cH_1$ and a
Fisherian $\S$-variable test in which the alternative $\cH_1$ is not
explicitly formulated, can all be multiplied --- and the result will
be meaningful, both in terms of monetary gain and in terms of error probability.

\paragraph{Acknowledgments} Many thanks to  
A.\ Barron, J.\ Berger, 
R.\ Frongillo, P.\ Harremo\"es, A.\ Hendriksen, A.\ Henzi,
A.\ Ly, R.\ Meester, M.\ Perez, A.\ Ramdas, J.\ ter Schure, 
G.\ Shafer, R.\ Turner, V.\ Vovk, B.\ Wagonner,  J.\ Ziegel who all made helpful remarks.



\FloatBarrier

\DeclareRobustCommand{\VANDER}[3]{#3}
\bibliography{safetestbib}

\begin{thebibliography}{72}
\providecommand{\natexlab}[1]{#1}
\providecommand{\url}[1]{\texttt{#1}}
\expandafter\ifx\csname urlstyle\endcsname\relax
  \providecommand{\doi}[1]{doi: #1}\else
  \providecommand{\doi}{doi: \begingroup \urlstyle{rm}\Url}\fi

\bibitem[Adams(2020)]{Adams20}
Reuben Adams.
\newblock Safe hypothesis tests for 2 x 2 contingency table.
\newblock Master's thesis, Delft Technical University, 2020.

\bibitem[Amrhein et~al.(2019)Amrhein, Greenland, and
  McShane]{amrhein2019scientists}
Valentin Amrhein, Sander Greenland, and Blake McShane.
\newblock Scientists rise up against statistical significance, 2019.

\bibitem[Balsubramani and Ramdas(2016)]{balsubramani2016sequential}
Akshay Balsubramani and Aaditya Ramdas.
\newblock Sequential nonparametric testing with the law of the iterated
  logarithm.
\newblock In \emph{Proceedings of the Thirty-Second Conference on Uncertainty
  in Artificial Intelligence}, pages 42--51, 2016.

\bibitem[Barron et~al.(1998)Barron, Rissanen, and Yu]{BarronRY98}
A.~Barron, J.~Rissanen, and B.~Yu.
\newblock The minimum description length principle in coding and modeling.
\newblock \emph{IEEE Transactions on Information Theory}, 44\penalty0
  (6):\penalty0 2743--2760, 1998.
\newblock Special Commemorative Issue: Information Theory: 1948-1998.

\bibitem[Bayarri et~al.(2012)Bayarri, Berger, Forte, Garc{\'\i}a-Donato,
  et~al.]{bayarri2012criteria}
Maria~J Bayarri, James~O Berger, Anabel Forte, G~Garc{\'\i}a-Donato, et~al.
\newblock Criteria for {B}ayesian model choice with application to variable
  selection.
\newblock \emph{The Annals of statistics}, 40\penalty0 (3):\penalty0
  1550--1577, 2012.

\bibitem[Benjamin et~al.(2018)Benjamin, Berger, Johannesson, Nosek,
  Wagenmakers, Berk, Bollen, Brembs, Brown, Camerer,
  et~al.]{benjamin2018redefine}
Daniel~J Benjamin, James~O Berger, Magnus Johannesson, Brian~A Nosek, E-J
  Wagenmakers, Richard Berk, Kenneth~A Bollen, Bj{\"o}rn Brembs, Lawrence
  Brown, Colin Camerer, et~al.
\newblock Redefine statistical significance.
\newblock \emph{Nature Human Behaviour}, 2\penalty0 (1):\penalty0 6, 2018.

\bibitem[Berger(2003)]{Berger03}
J.~Berger.
\newblock Could {F}isher, {J}effreys and {N}eyman have agreed on testing?
\newblock \emph{Statistical Science}, 18\penalty0 (1):\penalty0 1--12, 2003.

\bibitem[Berger et~al.(1998)Berger, Pericchi, and Varshavsky]{berger1998bayes}
James~O Berger, Luis~R Pericchi, and Julia~A Varshavsky.
\newblock Bayes factors and marginal distributions in invariant situations.
\newblock \emph{Sankhy{\=a}: The Indian Journal of Statistics, Series A}, pages
  307--321, 1998.

\bibitem[Berger(1985)]{Berger85}
J.O. Berger.
\newblock \emph{Statistical Decision Theory and {B}ayesian Analysis}.
\newblock Springer Series in Statistics. Springer-Verlag, New York, revised and
  expanded 2nd edition, 1985.

\bibitem[Bhattacharya and Waymire(2021)]{BhattacharyaW21}
R.~Bhattacharya and E.~C. Waymire.
\newblock \emph{Random Walk, Brownian Motion, and Martingales}, volume 292 of
  \emph{Springer Graduate Texts in Mathematics}.
\newblock 2021.

\bibitem[Breiman(1961)]{breiman1961optimal}
Leo Breiman.
\newblock Optimal gambling systems for favorable games.
\newblock \emph{Fourth Berkeley Symposium}, 1961.

\bibitem[Clarke and Barron(1994)]{ClarkeB94}
B.S. Clarke and A.R. Barron.
\newblock Jeffreys' prior is asymptotically least favorable under entropy risk.
\newblock \emph{Journal of Statistical Planning and Inference}, 41:\penalty0
  37--60, 1994.

\bibitem[Cover and Thomas(1991)]{CoverT91}
T.M. Cover and J.A. Thomas.
\newblock \emph{Elements of Information Theory}.
\newblock Wiley-Interscience, New York, 1991.

\bibitem[Csisz\'ar(1975)]{Csiszar75}
I.~Csisz\'ar.
\newblock {$I$}-divergence geometry of probability distributions and
  minimization problems.
\newblock \emph{Annals of Probability}, 3\penalty0 (1):\penalty0 146--158,
  1975.

\bibitem[Csisz{\'{a}}r and Tusn{\'{a}}dy(1984)]{CsiszarT84}
I.~Csisz{\'{a}}r and G.~Tusn{\'{a}}dy.
\newblock Information geometry and alternating minimization procedures.
\newblock \emph{Statistics and Decisions, Supplemental Issue}, 1:\penalty0
  205--237, 1984.

\bibitem[Csisz{\'a}r and Matus(2003)]{csiszar2003information}
Imre Csisz{\'a}r and Frantisek Matus.
\newblock Information projections revisited.
\newblock \emph{IEEE Transactions on Information Theory}, 49\penalty0
  (6):\penalty0 1474--1490, 2003.

\bibitem[Cumming(2012)]{Cumming12}
G.~Cumming.
\newblock \emph{Understanding the New Statistics: Effect Sizes, Confidence and
  Meta-Analysis}.
\newblock Routledge, 2012.

\bibitem[Darling and Robbins(1967)]{darling1967confidence}
D.A. Darling and H.~Robbins.
\newblock Confidence sequences for mean, variance, and median.
\newblock \emph{Proceedings of the National Academy of Sciences of the {USA}},
  58\penalty0 (1):\penalty0 66, 1967.

\bibitem[Dass and Berger(2003)]{dass-2003-unified-condit}
Sarat~C. Dass and James~O. Berger.
\newblock Unified conditional frequentist and {B}ayesian testing of composite
  hypotheses.
\newblock \emph{Scandinavian Journal of Statistics}, 30\penalty0 (1):\penalty0
  193--210, 2003.

\bibitem[Dawid(1997)]{Dawid97}
A.P Dawid.
\newblock Prequential analysis.
\newblock In S.~Kotz, C.B. Read, and D.~Banks, editors, \emph{Encyclopedia of
  Statistical Sciences}, volume 1 (Update), pages 464--470. Wiley-Interscience,
  New York, 1997.

\bibitem[Eaton(1989)]{Eaton1989}
M.L. Eaton.
\newblock \emph{Group Invariance Applications in Statistics}.
\newblock Institute of Mathematical Statistics and American Statistical
  Association, 1989.

\bibitem[Erven et~al.(2012)Erven, Gr{\"u}nwald, and
  de~Rooij]{erven2012catching}
Tim~van Erven, Peter Gr{\"u}nwald, and Steven de~Rooij.
\newblock Catching up faster by switching sooner: a predictive approach to
  adaptive estimation with an application to the {AIC--BIC} dilemma.
\newblock \emph{Journal of the Royal Statistical Society: Series B (Statistical
  Methodology)}, 74\penalty0 (3):\penalty0 361--417, 2012.
\newblock With discussion, pp. 399--417.

\bibitem[F\"ollmer and Gundel(2006)]{FollmerG06}
H.~F\"ollmer and A.~Gundel.
\newblock Robust projections in the class of martingale measures.
\newblock \emph{Illinois Journal of Mathematics}, 50\penalty0 (1--4):\penalty0
  439--472, 2006.

\bibitem[G{\'a}cs(2005)]{gacs2005uniform}
Peter G{\'a}cs.
\newblock Uniform test of algorithmic randomness over a general space.
\newblock \emph{Theoretical Computer Science}, 341\penalty0 (1-3):\penalty0
  91--137, 2005.

\bibitem[Gr{\"u}nwald and Roos(2020)]{GrunwaldR20}
P.~Gr{\"u}nwald and T.~Roos.
\newblock {M}inimum {D}escription {L}ength revisited.
\newblock \emph{International Journal of Mathematics for Industry}, 11\penalty0
  (1), 2020.

\bibitem[Gr\"unwald and Dawid(2004)]{GrunwaldD04}
P.~D. Gr\"unwald and A.~P. Dawid.
\newblock Game theory, maximum entropy, minimum discrepancy, and robust
  {Bayesian} decision theory.
\newblock \emph{Annals of Statistics}, 32\penalty0 (4):\penalty0 1367--1433,
  2004.

\bibitem[Gr{\"u}nwald and Mehta(2020)]{GrunwaldM19}
Peter Gr{\"u}nwald and Nishant Mehta.
\newblock Fast rates for general unbounded loss functions: from {ERM} to
  generalized {B}ayes.
\newblock \emph{Journal of Machine Learning Research}, 2020.

\bibitem[Gundel(2005)]{Gundel05}
A.~Gundel.
\newblock Robust utility maximization for complete and incomplete market
  models.
\newblock \emph{Finance and Stochastics}, 9:\penalty0 151--176, 2005.

\bibitem[Harremo\"es(2021)]{Harremoes21}
Peter Harremo\"es.
\newblock Personal communication, 2021.

\bibitem[Hendriksen et~al.(2021)Hendriksen, {\VANDER{Heide}{De}{de}}~Heide, and
  Gr{\"u}nwald]{HendriksenHG18}
A.~Hendriksen, R.~{\VANDER{Heide}{De}{de}}~Heide, and P.D. Gr{\"u}nwald.
\newblock Optional stopping with {B}ayes factors: a categorization and
  extension of folklore results, with an application to invariant situations.
\newblock \emph{Bayesian Analysis}, 16\penalty0 (3):\penalty0 961--989, 2021.

\bibitem[Henzi and Ziegel(2022)]{henzi2021valid}
Alexander Henzi and Johanna~F. Ziegel.
\newblock Valid sequential inference on probability forecast performance.
\newblock \emph{Biometrika}, 2022.

\bibitem[Howard et~al.(2021)Howard, Ramdas, McAuliffe, and
  Sekhon]{howard2018uniform}
Steven~R Howard, Aaditya Ramdas, Jon McAuliffe, and Jasjeet Sekhon.
\newblock Uniform, nonparametric, non-asymptotic confidence sequences.
\newblock \emph{Annals of Statistics}, 2021.

\bibitem[Jeffreys(1961)]{Jeffreys61}
H.~Jeffreys.
\newblock \emph{Theory of Probability}.
\newblock Oxford University Press, London, 3rd edition, 1961.

\bibitem[Johari et~al.(2021)Johari, Koomen, Pekelis, and
  Walsh]{johari2015always}
Ramesh Johari, Pete Koomen, Leonid Pekelis, and David Walsh.
\newblock Always valid inference: Continuous monitoring of a/b tests.
\newblock \emph{Operations Research}, 2021.

\bibitem[Kelly(1956)]{Kelly56}
J.L. Kelly.
\newblock A new interpretation of information rate.
\newblock \emph{Bell System Technical Journal}, pages 917--926, 1956.

\bibitem[Koolen and Gr{\"u}nwald(2021)]{KoolenG21}
W.~Koolen and P.~Gr{\"u}nwald.
\newblock Log-optimal anytime-valid e-values.
\newblock \emph{International Journal of Approximate Reasoning}, 2021.
\newblock Festschrift for G. Shafer's 75th Birthday.

\bibitem[Lai(2009)]{Lai09}
T.L. Lai.
\newblock Martingales in sequential analysis and time series, 1945–1985.
\newblock \emph{Electronic Journal for History of Probability and Statistics},
  5\penalty0 (1), 2009.

\bibitem[Lai(1976)]{lai1976confidence}
Tze~Leung Lai.
\newblock On confidence sequences.
\newblock \emph{The Annals of Statistics}, 4\penalty0 (2):\penalty0 265--280,
  1976.

\bibitem[Lehmann et~al.(2005)Lehmann, Romano, and Casella]{lehmann2005testing}
Erich~Leo Lehmann, Joseph~P Romano, and George Casella.
\newblock \emph{Testing statistical hypotheses}, volume~3.
\newblock Springer, 2005.

\bibitem[Levin(1976)]{Levin76}
Leonid~A. Levin.
\newblock Uniform tests of randomness.
\newblock \emph{Soviet Mathematics Doklady}, 17\penalty0 (2):\penalty0
  337--340, 1976.

\bibitem[Li(1999)]{Li99}
J.Q. Li.
\newblock \emph{Estimation of Mixture Models}.
\newblock PhD thesis, Yale University, New Haven, CT, 1999.

\bibitem[Liang et~al.(2008)Liang, Paulo, Molina, Clyde, and
  Berger]{liang2008mixtures}
Feng Liang, Rui Paulo, German Molina, Merlise~A Clyde, and Jim~O Berger.
\newblock Mixtures of {g-priors} for {B}ayesian variable selection.
\newblock \emph{Journal of the American Statistical Association}, 103\penalty0
  (481), 2008.

\bibitem[Ly et~al.(2020)Ly, Turner, {\VANDER{Schure}{Ter}{ter}}~Schure,
  Perez-Ortiz, and Gr{\"u}nwald]{LyT20}
A.~Ly, R.~Turner, J.~{\VANDER{Schure}{Ter}{ter}}~Schure, M.F. Perez-Ortiz, and
  P.~Gr{\"u}nwald.
\newblock {R}-package {\tt safestats}, 2020.
\newblock install in {R} by {\tt
  devtools::install\_github("AlexanderLyNL/safestats", ref = "logrank",
  build\_vignettes = TRUE)}.

\bibitem[McShane et~al.(2019)McShane, Gal, Gelman, Robert, and
  Tackett]{McShane19}
Blakeley~B. McShane, David Gal, Andrew Gelman, Christian Robert, and
  Jennifer~L. Tackett.
\newblock Abandon statistical significance.
\newblock \emph{The American Statistician}, 73\penalty0 (sup1):\penalty0
  235--245, 2019.

\bibitem[Orabona and Jun(2021)]{OrabonaJ21}
F.~Orabona and K.S. Jun.
\newblock Tight concentrations and confidence sequences from the regret of
  universal portfolio.
\newblock \emph{arXiv preprint arXiv:2110.14099}, 2021.

\bibitem[Pace and Salvan(2019)]{pace2019likelihood}
Luigi Pace and Alessandra Salvan.
\newblock Likelihood, replicability and robbins' confidence sequences.
\newblock \emph{International Statistical Review}, 2019.

\bibitem[Perez et~al.(2022)Perez, Lardy, {D}e Heide, and
  Gr{\"u}nwald]{perez2022estatistics}
Muriel~Felipe Perez, Tyron Lardy, Rianne {D}e Heide, and Peter Gr{\"u}nwald.
\newblock E-statistics, group invariance and anytime-valid testing.
\newblock \emph{arXiv:2208.07610}, 2022.

\bibitem[Posner(1975)]{Posner75}
E~Posner.
\newblock Random coding strategies for minimum entropy.
\newblock \emph{IEEE Transactions on Information Theory}, 21\penalty0
  (4):\penalty0 388--391, 1975.

\bibitem[Ramdas et~al.(2021)Ramdas, Ruf, Larsson, and Koolen]{fork-convex}
Aaditya Ramdas, Johannes Ruf, Martin Larsson, and Wouter~M. Koolen.
\newblock Testing exchangeability: Fork-convexity, supermartingales and
  e-processes.
\newblock \emph{International Journal of Approximate Reasoning}, 2021.
\newblock ISSN 0888-613X.
\newblock \doi{https://doi.org/10.1016/j.ijar.2021.06.017}.
\newblock URL
  \url{https://www.sciencedirect.com/science/article/pii/S0888613X21000980}.

\bibitem[Robbins(1970)]{robbins1970statistical}
Herbert Robbins.
\newblock Statistical methods related to the law of the iterated logarithm.
\newblock \emph{The Annals of Mathematical Statistics}, 41\penalty0
  (5):\penalty0 1397--1409, 1970.

\bibitem[Rouder et~al.(2009)Rouder, Speckman, Sun, Morey, and
  Iverson]{rouder-2009-bayes}
Jeffrey~N. Rouder, Paul~L. Speckman, Dongchu Sun, Richard~D. Morey, and
  Geoffrey Iverson.
\newblock Bayesian t-tests for accepting and rejecting the null hypothesis.
\newblock \emph{Psychonomic Bulletin \& Review}, 16\penalty0 (2):\penalty0
  225--237, Apr 2009.
\newblock ISSN 1531-5320.

\bibitem[Royall(2000)]{Royall00}
R.~Royall.
\newblock On the probability of observing misleading statistical evidence.
\newblock \emph{Journal of the American Statistical Association}, 2000.

\bibitem[Royall(1997)]{royall1997statistical}
Richard Royall.
\newblock \emph{Statistical evidence: a likelihood paradigm}.
\newblock Chapman and Hall, 1997.

\bibitem[R{\"u}schendorf(1984)]{Rueschendorf84}
L.~R{\"u}schendorf.
\newblock On the minimum discrimination information theorem.
\newblock \emph{Statistics and Decisions, Supplemental Issue}, 1:\penalty0
  263--283, 1984.

\bibitem[{\VANDER{Schure}{Ter}{ter}}~Schure and
  Gr{\"u}nwald(2019)]{terschure2019accumulation}
J.~{\VANDER{Schure}{Ter}{ter}}~Schure and Peter Gr{\"u}nwald.
\newblock Accumulation bias in meta-analysis: the need to consider time in
  error control.
\newblock \emph{F1000Research}, 8, 2019.

\bibitem[{\VANDER{Schure}{Ter}{ter}}~Schure
  et~al.(2021){\VANDER{Schure}{Ter}{ter}}~Schure, Perez-Ortiz, Ly, and
  Gr{\"u}nwald]{TerschurePLG21}
J.~{\VANDER{Schure}{Ter}{ter}}~Schure, M.F. Perez-Ortiz, A.~Ly, and
  P.~Gr{\"u}nwald.
\newblock The safe log rank test: Error control under continuous monitoring
  with unlimited horizon.
\newblock \emph{arXiv preprint arXiv:1906.07801}, 2021.

\bibitem[Sellke et~al.(2001)Sellke, Bayarri, and Berger]{sellke2001calibration}
Thomas Sellke, MJ~Bayarri, and James~O Berger.
\newblock Calibration of $p$-values for testing precise null hypotheses.
\newblock \emph{The American Statistician}, 55\penalty0 (1):\penalty0 62--71,
  2001.

\bibitem[Shafer(2021)]{Shafer19}
G.~Shafer.
\newblock Testing by betting: A strategy for statistical and scientific
  communication.
\newblock \emph{Journal of the Royal Statistical Society, Series A}, 2021.
\newblock With Discussion.

\bibitem[Shafer and Vovk(2019)]{ShaferV19}
G.~Shafer and V.~Vovk.
\newblock \emph{Game-Theoretic Probability: Theory and Applications to
  Prediction, Science and Finance}.
\newblock Wiley, 2019.

\bibitem[Shafer et~al.(2011)Shafer, Shen, Vereshchagin, and
  Vovk]{shafer2011test}
Glenn Shafer, Alexander Shen, Nikolai Vereshchagin, and Vladimir Vovk.
\newblock Test martingales, {B}ayes factors and p-values.
\newblock \emph{Statistical Science}, pages 84--101, 2011.

\bibitem[Siegmund(2013)]{siegmund2013sequential}
David Siegmund.
\newblock \emph{Sequential analysis: tests and confidence intervals}.
\newblock Springer Science \& Business Media, 2013.

\bibitem[Tops{\o}e(1979)]{Topsoe79}
F.~Tops{\o}e.
\newblock Information-theoretical optimization techniques.
\newblock \emph{Kybernetika}, 15\penalty0 (1):\penalty0 8--27, 1979.

\bibitem[Turner et~al.(2021)Turner, Ly, and Gr{\"u}nwald]{TurnerLG21}
Rosanne Turner, Alexander Ly, and Peter Gr{\"u}nwald.
\newblock Safe tests and always-valid confidence intervals for contingency
  tables and beyond.
\newblock \emph{arXiv Preprint 2106.02693}, 2021.

\bibitem[Ville(1939)]{Ville39}
J.~Ville.
\newblock {\'E}tude critique de la notion de collectif.
\newblock \emph{Monographies des Probabilit\'es}, 3, 1939.

\bibitem[Vovk(1993)]{vovk93}
V.G. Vovk.
\newblock A logic of probability, with application to the foundations of
  statistics.
\newblock \emph{Journal of the Royal Statistical Society, series B},
  55:\penalty0 317--351, 1993.
\newblock (with discussion).

\bibitem[Vovk and Wang(2021)]{VovkW21}
Vladimir Vovk and Ruodu Wang.
\newblock E-values: Calibration, combination, and applications.
\newblock \emph{Annals of Statistics}, 2021.

\bibitem[Wald(1947)]{Wald47}
Abraham Wald.
\newblock \emph{Sequential Analysis}.
\newblock Wiley, New York, 1947.

\bibitem[Wang and Ramdas(2020)]{WangR20}
Ruodu Wang and Aaditya Ramdas.
\newblock False discovery rate control with e-values.
\newblock \emph{arXiv preprint arXiv:2009.02824}, 2020.

\bibitem[Wasserstein et~al.(2016)Wasserstein, Lazar,
  et~al.]{wasserstein2016asa}
Ronald~L Wasserstein, Nicole~A Lazar, et~al.
\newblock The {A}{S}{A}s statement on p-values: context, process, and purpose.
\newblock \emph{The American Statistician}, 70\penalty0 (2):\penalty0 129--133,
  2016.

\bibitem[Waudby-Smith and Ramdas(2021)]{waudbysmith2021estimating}
Ian Waudby-Smith and Aaditya Ramdas.
\newblock Estimating means of bounded random variables by betting, 2021.

\bibitem[Williams(1991)]{Williams91}
D.~Williams.
\newblock \emph{Probability with Martingales}.
\newblock Cambridge Mathematical Textbooks, 1991.

\bibitem[Zhang et~al.(2011)Zhang, Glancy, and Knill]{zhang2011asymptotically}
Yanbao Zhang, Scott Glancy, and Emanuel Knill.
\newblock Asymptotically optimal data analysis for rejecting local realism.
\newblock \emph{Physical Review A}, 84\penalty0 (6):\penalty0 062118, 2011.

\end{thebibliography}



\newpage
\appendix
\DeclareRobustCommand{\VANDER}[3]{#2}

{\bf \LARGE Appendix}

\section{Theorem~\ref{thm:maina}, Corollaries and Conditions}

Here we prove Theorem~\ref{thm:maina} and its Corollary~\ref{cor:justone}. We also, in Appendix~\ref{app:regularityconditions}, 
discuss the required regularity conditions for Theorem~\ref{thm:maina} and we prove that they are applicable in all our examples.
\subsection{Proof of Theorem~\ref{thm:maina}, Simplest Version, and Corollary~\ref{cor:justone}}
\label{app:theorem1simpleproof}
The proof of the first version of Theorem~\ref{thm:maina} relies on quite technical results from \cite{Li99}, but if the minimum in (\ref{eq:firstgro}) is achieved by some prior $W^*_0$, and under the further condition that we can exchange differentiation and expectation, then the partial and crucial result that (\ref{eq:firsteval}) is an $\S$-variable has a very simple proof, which we provide first as a `warm-up': evaluate the derivative $f(\alpha) = (d/\alpha) D(Q \| (1-\alpha) P_{W^*_0} + \alpha P_{\theta})$ at $\alpha = 0$ for arbitrary $\theta \in \Theta_0$ and note that it is $\geq 0$ iff $\Exp_{\bf Y \sim Q}[p_{\theta}({\bf Y})/p_{W^*_0}({\bf Y})] = \Exp_{\bf Y \sim P_{\theta}}[q({\bf Y})/p_{W^*_0}({\bf Y})] \leq 1$. Differentiating again gives that $f(\alpha)$ is convex, and the result follows from convexity of $\cP_0 := \{ P_W: W \in \cW(\Theta_0) \}$. 
\newcommand{\Qdist}{\ensuremath{Q}}
\newcommand{\Qdens}{\ensuremath{q}}

We proceed to give the complete and fully general proof. 
Note that $\cP_0$ is convex, and (by assumption of the theorem) every distribution in $\cP_0$ as well as $Q$ has a density relative to $\mu$ and $\inf_{P \in \cP_0} D(Q \| P) < \infty$. These three givens allow us to use a range of results about the reverse information projection (RIPr) established in the Ph.D.\ thesis \citep{Li99} (additional proofs of   (extensions of) all of Li's results we need below can be found in the refereed paper \cite{GrunwaldM19}).

First, the existence and uniqueness of a measure $P^*_0$ (not
necessarily a probability measure) with density $p^*_0$ that
satisfies $D(Q \| P^*_0) = \inf_{P \in \cP_0} D(Q \| P)$ (i.e.\ it is the RIPr), and furthermore has the property
\begin{equation}\label{eq:notyetreversed}
\text{for all $p$ that are densities of some $P \in  \cP_0$:}\ \  
\Exp_{\Ydata \sim Q} 
\left[\frac{p(\Ydata)}{p^*_0(\Ydata)}\right] \leq 1,
\end{equation}
follows directly from \cite[Theorem 4.3]{Li99}.  But by writing
out the integral in the expectation explicitly we immediately see
that we can
rewrite (\ref{eq:notyetreversed})  as:
$$
\text{for all $P \in \cP_0$:}\ \  
\Exp_{\Ydata \sim P} 
\left[\frac{q(\Ydata)}{p^*_0(\Ydata)}\right] \leq 1.
$$ Li's Theorem 4.3 still allows for the possibility that $\int
p^*_0(y) \dif \mu (y) > 1$.  To see that in fact this is
impossible, i.e.\ $p^*_0$ defines a (sub-) probability
density, use Lemma 4.5 of \cite{Li99}. This shows
that $\eval^*=q(\bY)/p^*_0(\bY)$ is an $\S$-variable, and (using that $P^*_0$ is the RIPr)
the second and third equality of (\ref{eq:firstgro}).  The final line of the result (`if $\ldots$ then $P^*_0 = P_{W^*_0}$') follows directly from
Lemma 4.1 of \cite{Li99}.

It remains to show the first equality of (\ref{eq:firstgro}) and essential uniqueness of $\eval^*$. 
For the former, it is sufficient to show that for all $\S$-variables, i.e.\ all $\eval \in \cE(\Theta_0)$, 
\begin{align}
\label{eq:onlythingnotinli}
\Exp_{\Ydata \sim Q} \left[ \log \eval \right] \leq \Exp_{Q} \left[ \log \eval^* \right].
\end{align}
To show this, fix any $\S$-variable $\eval = e(\Ydata)$ in
$\cE(\Theta_0)$. Now further fix $\epsilon > 0$ and fix a  $W_{(\epsilon)} \in
\priorset(\Theta_0)$ with $D(Q \| P_{W_{(\epsilon)}}) \leq \inf_{W_0 \in \cW(\Theta_0)}  D(Q \| P_{W_0}) + \epsilon$. We must have, with $q'(y) \coloneqq e(y)
p_{W_{(\epsilon)}}(y)$, that $\int q'(y) \dif \mu = \Exp_{\Ydata \sim P_{W_{(\epsilon)}}} [\eval] \leq
1$, so $q'$ is a sub-probability density, and by properness of the log scoring rule, 
\begin{equation}
\Exp_{Q} [ \log \eval] = \Exp_{Q} \left[ \log \frac{q'(\Ydata)}{p_{W_{(\epsilon)}}(\Ydata)} \right]
\leq \\ \Exp_{Q} \left[ \log \frac{q(\Ydata)}{p_{W_{(\epsilon)}}(\Ydata)} \right] = D(Q \| P_{W_{(\epsilon)}}) \leq \inf_{W_0 \in \cW(\Theta_0)}  D(Q \| P_{W_0}) + \epsilon.  \nonumber \end{equation}
Since we can take $\epsilon$ to be arbitrarily close to $0$, (\ref{eq:onlythingnotinli}) follows.

To show essential uniqueness of $\eval^*$, let $\eval$ be any $\S$-variable with
$\Exp_{Q} [ \log \eval] = \Exp_{Q} [ \log \eval^*]$.  By
linearity of expectation, $\eval'= (1/2) \eval^* + (1/2)\eval$ is then also
an $\S$-variable, and by Jensen's inequality applied to the logarithm we
must have $\Exp_{Q} [ \log \eval'] > \Exp_{Q} [ \log
\eval^*]$ unless $Q(\eval = \eval^*) = 1$. Since we have
already shown that for any $\S$-variable $\eval'$, $\Exp_{Q} [ \log
\eval'] \leq \Exp_{Q} [ \log \eval^*]$, it follows that
$Q(\eval \neq \eval^*) = 0$. But then, by our assumption that $Q$ has full support, i.e.\ $q({\bf Y}) > 0$ hold $\mu$-almost everywhere, we must have that $P_{\theta}(\eval \neq \eval^*) = 0$ for all $\theta \in \Theta$, so $\eval^*$ is
essentially unique. 

\paragraph{Proof of Corollary~\ref{cor:justone}}
Let $W_0$ be as in the corollary statement. By definition of $\eval^*$ as in Theorem~\ref{thm:maina}, simplest version, and then using strict convexity of the KL divergence in its second argument \citep{CoverT91} and the fact that $D(Q \| P)$ is minimised, over $P\in \{P_W : W \in \cW(\Theta_0) \}$, we have: 
$$
\Exp_{Q}\left[ 
\log \eval^* \right] = 
D(Q \| P^*_0) <
D(Q \| P_{W_0}) = \Exp_{Q}\left[ 
\log \frac{q}{p_{W_0}}\right]
$$
so that, if ${q}/{p_{W_0}}$ were an $\S$-variable, we would have a contradiction with the first equality in  (\ref{eq:firstgro}). 

\subsection{Proof of full version of Theorem~\ref{thm:maina}}
\commentout{
	that the minimum is achieved uniquely.  Suppose that the minimum is achieved for some $(W^*_0,W^*_1)$ and some $(W'_0,W'_1)$. From \cite[Theorem 11]{ErvenH14} we have that 
	the KL divergence 
	$$D(P_{(1-\alpha){W^*_1} + \alpha W'_1} \|
	P_{(1-\alpha) W^*_0 + \alpha W'_0}) = 
	D(( 1-\alpha)P_{W^*_1} + \alpha P_{W'_1} \|
	(1-\alpha) P_{W^*_0} + \alpha P_{W'_0})
	$$ 
	is strictly convex on $0 \leq \alpha \leq 1$ unless
	$p_{W^*1}/p_{W^*_0} = p_{W'_1}/p_{W'_0}$ almost surely under the underlying measure. If this is the case we have essential uniqueness so we are done. If not, then we may assume  strict convexity of $D$ and we will show that this leads to a contradiction: by  strict convexity of $D$ and linearity of expectation $\Exp_{\theta}$,  
	\begin{align*}
	& D(P_{(1-\alpha){W^*_1} + \alpha W'_1} \|
	P_{(1-\alpha) W^*_0 + \alpha W'_0}) + \Exp_{\theta \sim 
		(1-\alpha){W^*_1} + \alpha W'_1}[f(\theta)]
	<  \\
	& (1-\alpha)
	D(P_{W^*_1}  \|
	P_{W^*_0}) + \alpha D(P_{W'_1}  \|
	P_{W'_0}) +  (1-\alpha) \Exp_{\theta \sim 
		W^*_1} [f(\theta)] + \alpha \Exp_{\theta \sim W'_1}[f(\theta)],
	\end{align*}
	which implies that $(W^*_1,W^*_0)$ is 
	not a minimum on $\cW_1 \times \cW_0$ which are both assumed convex --- so we have arrived at the desired contradiction. 
}
The proof consists of two sub-parts, Part (a) relying on the simple version of the theorem presented in Section~\ref{sec:GRO} and proven above (henceforth called `the simple theorem'), and 
Part (b) relying on a nonstandard minimax/saddle-point theorem from \cite{GrunwaldD04} (GD from now on), itself relying heavily on an earlier result from \cite{Topsoe79}.  

\textit{ Part (a).\/}
We first show that $\eval^f$ as in (\ref{eq:finale}) is an $\S$-variable. This follows by the simple theorem, with $\bV$ in the role of $\bY$, $Q$ in the theorem statement substituted by $P^{[\bV]}_{W_1^*}$ and $P_W$ for $W \in \cW(\Theta_0)$ replaced by $P^{[\bV]}_{W}$ and using that (\ref{eq:tuinslang}) implies that 
$\inf_{W_0 \in \cW_0(\Theta_0)} D(P^{[\bV]}_{W^*_1} \|
P^{[\bV]}_{W_0}) < \infty$. 
Next, we show that {\em if\/} (\ref{eq:jipr}) holds for $\eval^f$ as in (\ref{eq:finale}), then all $\S$-variables $\eval$ for which it holds must be essentially equal to $\eval^f$. To see this, suppose that $\eval$ is another $\S$-variable satisfying (\ref{eq:jipr}). Then we have
$$
\inf_{W \in \priorset_1} (\ \Exp_{\Ydata \sim P_W}[\log \eval] - f(W) 
\ ) = D(P^{[\bV]}_{W^*_1} \| P^{*[\bV]}_0)  - f(W^*_1) =  \Exp_{\Ydata \sim P_{W^*_1}}\left[ \log \eval^f \right]- f(W^*_1)
$$
as follows by writing out the definition of $D(\cdot \| \cdot)$. On the other hand, using the definition of $\inf$, we must have
$$
\inf_{W \in \priorset_1} (\ \Exp_{\Ydata \sim P_W}[\log \eval] - f(W) 
\ ) \leq \Exp_{\Ydata \sim P_{W^*_1}}\left[\log  \eval \right]- f(W^*_1).
$$
The only way these two displays can be reconciled is if 
$ \Exp_{\Ydata \sim P_{W^*_1}}\left[ \log \eval \right] \geq
\Exp_{\Ydata \sim P_{W^*_1}}\left[ \log \eval^f \right]$. But since $\eval^f$ is the RIPr of $P_{W_1^*}^{[\bV]}$, we can use the simple theorem again, applied with $P_{W_1^*}$ in the role of $Q$ and $\cH_1 = \{P_{W_1^*}\}$, to conclude that $\eval^f$ must be essentially equal to $\eval$. The final line of the fully general Theorem~\ref{thm:maina} follows again by reduction to the analogous statement of the simple theorem. 

Finally, we will use the simple theorem  to show the following one-sided version of (\ref{eq:jipr}): 
\begin{equation}\label{eq:jipronesided}
\inf_{W \in \priorset_1} (\ \Exp_{\Ydata \sim P_W}[\log \eval^f] - f(W) 
\ ) \leq 
\sup_{\eval\in \cE(\Theta_0)} \inf_{W \in \priorset_1} (\ \Exp_{\Ydata \sim P_W}[\log \eval] - f(W) \ ) \leq  \\   
D(P^{[\bV]}_{W^*_1} \| P^{*[\bV]}_0)  - f(W^*_1).
\end{equation}   
The first inequality is trivial since $\eval^f \in \cE(\Theta_0)$. The second follows if we can show that
\begin{equation}\label{eq:wringe}
\sup_{\eval\in \cE(\Theta_0)} \inf_{W \in \priorset_1} \Exp_{P_W}[\log \eval - f(W)] \leq 
\inf_{W \in \priorset_1} \inf_{W_0 \in \cW_0} (\ D(P_{W} \| P_{W_0}) - f(W_1) \ ) 
\end{equation}
and recognizing that by assumptions of the theorem, the right-hand side of (\ref{eq:wringe}) coincides with the right-hand side of (\ref{eq:jipronesided}). 
To prove (\ref{eq:wringe}), note that by the simple version of the theorem we already  have for each fixed $W_1 \in \priorset_1$ that 
$$\inf_{W_0 \in \cW_0} D(P_{W_1} \| P_{W_0}) = \sup_{\eval \in
	\cE(\Theta_0)} \Exp_{P_{W_1}} [ \log \eval]$$ and this directly implies
the inequality by adding $-f(W_1)$ to both sides and using a standard ``$\inf \sup \geq \sup \inf$'' argument
(the trivial side of the minimax theorem). 
The equality follows by assumption of the Theorem.

\textit{ Part (b).\/}
Taking stock, we see that the only thing that is left to prove is (\ref{eq:jipronesided}) with the reversed inequalities. For this, it suffices to show 
that 
\begin{equation}\label{eq:virus}
D(P^{[\bV]}_{W^*_1} \| P^{*[\bV]}_{0}) - f(W^*_1) \leq \inf_{W \in \priorset_1} \Exp_{P_W}[\log \eval^f - f(W) ].
\end{equation}
Since all
distributions occurring in (\ref{eq:virus}) are marginals on $\bV$, and
$\eval^f$ can be written as a function of $\bV$, we will from now on simply
refer to the marginal densities on $\bV$ corresponding to $P_{W}$ as
$p_W$ (rather than $p'_W$ as in the main text), and we will omit the
superscripts $[\bV]$ from $P$; thus we take as our basic outcome now $\bV$
rather than $\Ydata$.

\newcommand{\Wn}{\ensuremath{W_{\text{\sc n}}}}
\newcommand{\Wu}{\ensuremath{W_{\text{\sc u}}}}
\newcommand{\Wfin}{\ensuremath{\cW^{< \infty}}}
\newcommand{\pluspart}[1]{\ensuremath{\left[#1\right]_{+}}}
\newcommand{\minpart}[1]{\ensuremath{\left[#1\right]_{-}}}
We will show the stronger statement that (\ref{eq:virus}) holds with equality, by using a minimax/saddle point result that holds for general functions $L: \Theta_1 \times \cW_1 \rightarrow \reals \cup \{ \infty \}$ such that $L(\Wn,\Wu) := \Exp_{\theta \sim \Wn}[L(\theta,\Wu)]$ 
is well-defined for all $\Wn \in \cW_1$ (the condition `well-defined' is necessary since the expectation is over a function that may neither be bounded from below nor from above; see for example Section 3.1 of GD  for the (standard) definitions). These $L$ are interpreted as loss functions, with $\theta_1 \in \Theta_1$ denoting a state of nature and $\cW_1$ an arbitrary convex set of distributions on $\Theta_1$, each $W \in \cW_1$ being interpreted as an action.
Following  GD, we can associate a {\em decision-theoretic entropy\/} $H(\Wn) := \inf_{\Wu \in \cW} L_0(\Wn,\Wu) = L(\Wn,\Wn)$ with any such $L$. The following result holds for all $\Theta_1$, $\cW_1$ and $L$ as defined above but we will apply it to the instantiation of $\Theta_1$ and $\cW_1$ in Theorem~\ref{thm:maina}. 

\paragraph{GD's Theorem 6.3} Assume that (a) $L$ is a proper scoring rule, i.e.\ for all $\Wn \in \cW_1$, $H(\Wn) = L(\Wn,\Wn)$.  Suppose that (b)  $W_1^* \in \cW_1$ is `maximum entropy' i.e.\ $\sup_{\Wn \in \cW_1} H(\Wn) = H(W^*_1) < \infty$ and (c) the lower semi-continuity condition below holds. Then $(W_1^*,W_1^*)$ is a saddle-point relative to $L$, i.e.\
\begin{equation}\label{eq:entrophy}
H(W^*_1) = L(W^*_1,W^*_1) = \sup_{W \in \cW_1} L(W,W_1^*).
\end{equation}
\paragraph{Lower Semicontinuity Condition (GD's Condition 6.1)}
Let $(W_n)$ be a sequence of distributions in $\cW_1$ such that $H(W_n)$ is bounded below and such that $(W_n)$ converges weakly to some distribution $W^{\circ}$ on $\Theta_1$. Then $L(W^{\circ},\Wu)$ is well-defined for all $\Wu \in \cW$ and for all $W'\in \cW$, $L(W', W^{\circ}) \leq \lim \inf_{n \rightarrow \infty} L(W',W_n)$. 

We  now define the specific loss function to which we will apply the above theorem. 
$\cW_1$ and $\Theta_1$ are defined as in the statement of Theorem~\ref{thm:maina}. (\ref{eq:tuinslang}) implies that $D(P_{W_1^*} \| P^*_0) < \infty$
for some $P^*_0$ with density $p^*_0$. 
Similarly, $P_{W}$ must have some density $p_W$ under all $W \in \cW_1$. We can therefore define, 
using these densities,
\begin{align*}
L(\theta,\Wu) &  = 
\Exp_{\bV \sim P_{\theta}} \left[
- \log \frac{p_{\Wu}(\bV)}{p^*_0(\bV)}
- f(\theta)\right] \\
L(\Wn,\Wu) &  = 
\Exp_{\theta \sim \Wn}
\Exp_{\bV \sim P_{\theta}} \left[
- \log \frac{p_{\Wu}(\bV)}{p^*_0(\bV)}
- f(\theta)\right]. 
\end{align*}
Since $P_{\theta}$ has full support for all $\theta \in \Theta_1$, $P_{W_1^*}$ has full support and so $p^*_0(\bV) > 0$ a.s.\ under $P_{W_1^*}$ and hence under all $P_{\theta}$ with $\theta \in \Theta_1$. Similarly $p_W(\bV) > 0$ a.s.\ under all $P_\theta$ with $\theta \in \Theta_1$. Thus, the quantity inside the expectation is almost surely well-defined.
To see that the expectations are themselves well-defined (using standard definitions, see again Section 3.1 of GD), note that we can write
\begin{align*}
L(\Wn,\Wu) = & + \Exp_{\Wn} \Exp_{P_{\theta}} 
\pluspart{- \log \frac{p_{\Wu}}{p_{\Wn}}} + 
\Exp_{\Wn} \Exp_{P_{\theta}} 
\pluspart{- \log \frac{p_{\Wn}}{p^*_0}} -
\Exp_{\Wn} \Exp_{P_{\theta}} 
\pluspart{f(\theta)} \\
& 
- \Exp_{\Wn} \Exp_{P_{\theta}} 
\minpart{- \log \frac{p_{\Wu}}{p_{\Wn}}} - 
\Exp_{\Wn} \Exp_{P_{\theta}} 
\minpart{- \log \frac{p_{\Wn}}{p^*_0}} +
\Exp_{\Wn} \Exp_{P_{\theta}} 
\minpart{f(\theta)}
\end{align*}
with $\pluspart{x} := \max \{x,0\}$ and $\minpart{x} = \max\{ - x, 0 \}$. The expectation would be undefined iff there is both a term equal to $\infty$ and a term equal to $- \infty$ on the right. We will show that this is not the case. We assume 
$f(\theta)$ bounded, and, under our finite KL condition,   $D(P_{\Wn} \| P_{\Wu}) = \Exp_{\Wn} \Exp_{P_{\theta}} 
\pluspart{- \log \frac{p_{\Wu}}{p_{\Wn}}} - 
\Exp_{\Wn} \Exp_{P_{\theta}} 
\minpart{- \log \frac{p_{\Wu}}{p_{\Wn}}} < \infty$, so we only need to worry about the terms involving $p^*_0$. But these are also the positive and negative parts of a minus KL divergence, so  the full expectation is well-defined as a number in $\reals \cup \{-\infty\}$. 
The expectations are therefore welldefined and we can write
\begin{align}\label{eq:distrophy}
L(\Wn,\Wu) = D(P_{\Wn} \| P_{\Wu}) - D(P_{\Wn} \| P^*_0) - \Exp_{\theta \sim \Wn} [f(\theta)]
\end{align}
and analogously for $L(\theta,\Wu)$.

\paragraph{Applying GD's theorem to $L$}
We apply GD's theorem to the loss function $L$ above with $W_1^*$ as in the statement of the theorem. From (\ref{eq:distrophy}) we see that $L(W^*_1,W)$ is minimised, over $\cW_1$, by $W = W^*_1$ and then finite, so that GD's requirements (a)  and (b) hold for loss function $L$. We can now reason as follows. If the lower semicontinuity condition (c) also holds, then the theorem applies and (\ref{eq:entrophy}) implies, taking minus on both sides,
\begin{align*}
- L(W^*_1,W^*_1) = \inf_{W \in \cW_1} - L(W,W_1^*).
\end{align*}
which, rewriting the left-hand side using (\ref{eq:distrophy}) and the right-hand side using definition of $L$, is in turn seen to be equivalent to
(\ref{eq:virus}), and the desired result follows. 

It thus only remains to show that the lower semicontinuity condition holds.  
Using (\ref{eq:distrophy}) we can write  $H(\Wn)= - D(P_{\Wn} \| P^*_0) + f(\Wn)$ for all $\Wn \in \cW_1$. 
Take a sequence $(W_n)_n$ and $W^{\circ}$ as in the condition. Then $(P_{W_n})_n$ converges weakly to $P_{W^{\circ}}$ (this is easy to see but see the proof of Lemma 9.2.\ of GD for an explicit proof). Since also  $f$ is bounded, $f(W_n)$ converges to $f(W^{\circ})$.  Also, for some $K \in \reals$, for all $n$, we have $H(W_n) \geq K$ so for some $K' \in \reals$, by boundedness of $f$, we have, for all $n$, $D(P_{W_n} \| P^*_0) \leq K'< \infty$.  By Posner's (\citeyear{Posner75}) theorem, $D(P_W \| P^*_0)$ is lower semi-continuous in its first argument.
Posner only proves the result for $P^*_0$ a probability measure; but it still holds even if $P^*_0$ is a strict sub-probability measure, since then $p'(\bY) = p^*_0(\bY)/\int p^*_0(\bY) d \mu = 1$ represents a distribution $P'$ and the result follows by applying Posner's result to $P'$ and noting that $D(P_W \| P^*_0) = D(P_W \| P') + C$ for some constant $C$ not depending on $W$.

The lower-semicontinuity in the first argument implies $D(P_{W^{\circ}} \| P^*_0) \leq \lim \inf_{n \rightarrow \infty} D(P_{W_n} \| P^*_0) \leq K'< \infty$. Following an argument exactly parallel to the proof of well-definedness for $L(\Wn,\Wu)$ given above, it now follows that $L(W^{\circ}, \Wu)$ is well-defined for all $\Wu \in \cW_1$ as required. 
Next,  using (\ref{eq:distrophy}), we see that it is sufficient to show that for all $W \in \cW_1$:
$$
D(P_{W} \| P_{W^{\circ}}) \leq  
\lim \inf_{n \rightarrow \infty} D(P_W \| P_{W_n}).$$ 
But this again follows directly from Posner's theorem, which also says that KL divergence is lower semi-continuous in its {\em second\/} argument. 
\subsection{Remarks on and Checking of Conditions for Theorem~\ref{thm:maina}}
\label{app:regularityconditions}
\paragraph{The Full Support and Finite KL Condition}
Requiring full support in the simplest version of the theorem ensures that $\eval^*$ is a.s.\ well-defined: without it, there may be an outcome ${\bf y}$ such that for some $\theta \in \Theta_0$, $P_{\theta}({\bf y}) > 0$ whereas $P_{W_0}({\bf y})= Q({\bf y}) = 0$. Then $\eval^*$ is undefined with positive probability under this $\theta$.  The finite KL condition $D(P_{\theta} \| P_{\theta'}) < \infty$ imposed in the generalised versions of the theorem  is just slightly stronger than the full support condition. It is required to make sure that all expectations in the proof are well-defined.

For standard parametric models in standard parameterisations (e.g.\ all multivariate exponential families in their mean-value parameterisation), both conditions will hold automatically as long as one excludes  points at the boundary of the parameter space, if those exist. For example, in the $2\times 2$ setting without a pre-specified effect size we restrict $\Theta_1$ to $(0,1)^2$, requiring the Bernoulli probabilities $\mu_{1|a}$ and $\mu_{1|b}$ to be non-degenerate.  But, since the condition only refers to $\Theta_1$, it is o.k.\ to set $\Theta_0= [0,1]$ to include the boundary points in the null.

\paragraph{Additional Condition: Existence of $W^*_1$}
The requirement for composite $\cH_1$ 
that a ${W^*_1}$ exists achieving the minimum in (\ref{eq:jipr}) is strong in general,  but it holds
in all our examples with composite $\cH_1$: Example~\ref{ex:ponential} ($W_1^*$ is shown to be a point prior in the example), 
Example~\ref{ex:jeffreys} (since there we restrict $\Theta_1$ to be compact) and Example~\ref{ex:2x2c} and~\ref{ex:2x2b} (here verifying the condition requires some work, see below). 
It also holds in  the $t$-test setting underneath Theorem~\ref{thm:particular} with effect sizes $\delta^+$ and $\delta^-$ ($W_1^*$ reduces to a point prior on $\delta^+$). 
By allowing $\S$-variables to be functions of 
$\bV$ that are themselves functions of  $\Ydata$ (i.e.\ $\sigma(\Ydata)$-measurable) as in the latter example, we make the condition considerably
weaker.

\paragraph{Applicability of Theorem~\ref{thm:maina} and existence of minimizing $W^*_1$ and $W^*_0$ in Example~\ref{ex:2x2c} and \ref{ex:2x2b}}
We have  $\cH_1 = \{P_{\mu_{1|a},\mu_{1|b}} : (\mu_{1|a},\mu_{1|b}) \in\Theta_1 \}$ and $\cH_0 = \{P_{\mu}: \mu \in \Theta_0 \}$, $\Theta_0 = [0,1]$ with definitions as in Example~\ref{ex:2x2}. In Example~\ref{ex:2x2c}, we take  $\Theta_1 = (0,1)^2$ and we can take $\Theta_0 = [0,1]$ or $\Theta_0 = (0,1)$ (the same minima will be achieved in both cases). In Example~\ref{ex:2x2b} we take $\Theta_1 = \{ (\mu_{1|a},\mu_{1|a}+ \delta): 0 < \mu_{1,a} < 1- \delta \}$ and $\Theta_0= (0,1)$.   

We only give the proof for Example~\ref{ex:2x2c}; the proof for Example~\ref{ex:2x2b} is entirely analogous. 

The requirement for applying Theorem~\ref{thm:maina} that $P_{\theta}, P_{\theta'}$ with $\theta,\theta' \in \Theta_1$ satisfy $D(P_{\theta} \| P_{\theta'}) < \infty$, and have full support trivially holds by our exclusion of the boundary points in $\Theta_1$. The only remaining condition for applying Theorem~\ref{thm:maina}  is the existence of a KL minimizing prior $W^*_1$. We will show the stronger result that there exists a pair of minimizing priors  $(W^*_1, W^*_0)$ with $W^*_1 \in \cW(\Theta_1)$ and $W^*_0 \in \cW(\Theta_0)$ such that 
\begin{equation}\label{eq:klparty}
\inf_{W_1 \in \cW(\Theta_1), W_0 \in \cW(\Theta_0)}
(D(P_{W_1} \| P_{W_0}) - f(W_1) )=  D(P_{W^*_1} \|
P_{W^*_0}) - f(W^*_1)  < \infty,
\end{equation}
with $f(W) = \Exp_{(\mu_{1|a},\mu_{1|b}) \sim W} [f(\mu_{1|a},\mu_{1|b})]$ and $f(\mu_{1|a},\mu_{1|b}) =D(P_{\mu_{1|a},\mu_{1|b}} \| P_{\mu^{\circ}})$ with $\mu^{\circ}$ as in (\ref{eq:eval2x2}). 
We do this by first, in Part (a), showing that there exists such a pair with $W^*_1 \in [0,1]^2$, i.e.\ with $\Theta_1$ extended to include its boundary points. We then, in Part (b), show that the resulting $W^*_1$ puts no mass on these boundary points, so that it also achieves the minimum on $\cW(\Theta_1)$. 

{\em Part (a)} The sets $\cW([0,1]^2)$ and  $\cW(\Theta_0)$ are convex and compact in the weak topology; by 
Posner's (\citeyear{Posner75}) theorem, 
$D(P_{W_1} \| P_{W_0})$ is lower-semicontinuous in its second argument in the weak topology on $\{P_{W_0} : W_0 \in \cW(\Theta_0) \}$ and hence on $\cW(\Theta_0)$ itself (see Section 9 of GD) and $f(W)$ is linear and bounded on $\cW(\Theta_0)$; this shows that for each $W_1$, the corresponding minimizing $W^*_0$ is achieved; since $D(P_{W_1} \| P_{W^*_0}) \leq D(P_{W_1} \| P_{1/2}) < \infty$ (with $P_{1/2} \in \cH_0$ representing Bernoulli$(1/2)$) and $f$ is bounded, the finiteness in (\ref{eq:klparty}) is guaranteed as well. To see that the minimum $W_1$ is achieved as well, note that, again by Posner's theorem, $D(P_{W_1} \| P_{W_0})$ is also lower-semicontinuous in its first argument in  the weak topology on $\{P_{W_1} : W_1 \in \cW([0,1]^2) \}$. The same argument as before now gives that the minimum $W^*_1$ is achieved. 

{\em Part (b)} It now suffices to show that $P_{W^*_1}$ has full support, for this implies that $W^*_1$ assigns mass $1$ on $\Theta_1 = (0,1)^2$ and then $W^*_0$ must assign mass $1$ on $(0,1)$ (otherwise the KL divergence in (\ref{eq:klparty}) would be infinite, and we already established it is not).  
To show full support of $P_{W^*_1}$, note that by symmetry considerations, we must have that, with our choice of $f$, 
\begin{equation}
\label{eq:symmetry}
M:= D(P_{W^*_1} \| P_{W^*_0} ) - f(W^*_1) = 
D(P_{W^{\circ}_1} \| P_{W^{\circ}_0} ) - f(W^{\circ}_1) 
\end{equation}
for a prior $W^{\circ}_1$ such that, for all ${\bf y} \in \{0,1\}^n$,  $p_{W^{\circ}_1}({\bf y} \mid {\bf x})=
p_{W^{*}_1}({\bf y}' \mid {\bf x})$ with ${\bf y}'$ is the modification of ${\bf y}$ with all $0$s and $1$s interchanged, and similarly for $W^{\circ}_0$.  
Now if $P_{W^*_1}$ would not have full support, we have $P_{W^*_1}(Y_1 = y_1,Y_2 = y_2 \mid X_1= a,X_2=b) = 0$  for some   $(y_1,y_2) \in \{0,1\}^2$. Then $P_{W^{\circ}_1}(Y_1 = \bar{y}_1,Y_2 = \bar{y}_2 \mid X_1= a,X_2=b) = 0$  for $\bar{y}_j = 1- y_j$. 
But  $\inf_{W_0 \in \cW(\Theta_0)} D(P_{W^*_1} \| P_{W_0} ) - f(W^*_1)$ as a function of $W^*_1$ is easily checked to be strictly convex on $\cW(\Theta_1)$, so by (\ref{eq:symmetry}) we must have that, for $W' = (1/2) W^*_1 + (1/2) W^{\circ}_1$, it holds that 
$\inf_{W_0 \in \cW(\Theta_0)} D(P_{W'}\| P_{W_0}) + f(W') < 
M$. 
But this contradicts that $M$ is the minimum. 
This shows that $P_{W^*_1}$ has full support.

\paragraph{Finiteness of Support in Example~\ref{ex:2x2c} and~\ref{ex:2x2b}}
We claimed that the supports of the priors $W^*_1$ and $W^*_0$ in Example~\ref{ex:2x2b} (restricted $\Theta_1$) are finite. In fact they are finite also in Example~\ref{ex:2x2c} (unrestricted $\Theta_1$).
We verify this for $W^*_1$, the case for $W^*_0$ is analogous. Note that for given sample size $n$, the probability
distribution $P_W$ is completely determined by the probabilities
assigned to the sufficient statistics $N_{1|a}, N_{1|b}$. This means
that for each prior $W \in \cW(\Theta_1)$, the Bayes marginal $P_W$
can be identified with a vector of $M_n \coloneqq (n_a +1) \cdot (n_b +1)$
real-valued components. Every such $P_W$ can also be written as a
mixture of $P_{\theta}$'s for
$\theta = (\mu_{a|1},\mu_{b|1}) \in \Theta_1$, a convex set. By Carath\'eodory's
theorem we need at most $M_n$ mixture  components to describe an
arbitrary $P_W$ as a mixture of the $P_{\theta}$'s; this proves the claim. 
\section{Additional Clarifications and Proofs}
\subsection{Section~\ref{sec:os}}
\label{app:first}
We discuss two extensions of the filtration to be used in Definition~\ref{def:cond.safe}. 
For concreteness and simplicity we consider the sequential set-up of 
Section~\ref{sec:GROanalysis} in which, for each study $m$,  there is an underlying data stream $Y_{m,1}, Y_{m.2}, \ldots$. 
\commentout{\paragraph{Filtration: Subtleties}
The natural filtration to use if there is no additional nonstochastic side-information, is $(\sigma(\bY^{(j)}))_{(j)}$, i.e.\ when specifying a new $\S$-variable $\eval_{(j+1)}$ we are allowed to make use of all data we observed in the past.
However, the $\S$-variable assignment based on the $t$-test leads to a filtration  that is coarser than  $\sigma(\bY^{(j)})= \sigma(Y^{T_{(j)}})$, as can be seen from (\ref{eq:condy}) in Section~\ref{sec:GROanalysis}. A similar phenomenon may ocur in other settings with nuisance parameters expressing group invariance. 
To illustrate, consider a fully sequential assignment as defined in that section. Then  $U_i = V_i = Y_i / |Y_1|$, so if one knows all the $U_i$ one still does not know the $Y_i$. The only repercussion of this is that the stopping times $T_{(j)}$ of the individual studies are now not allowed to directly depend directly on the $Y_i$, since this is not measurable. Under standard (e.g.\ fixed, or stopping when the likelihood overshoots some threshold --- the likelihood can be written as a function of the $V_i$ and thus {\em is\/} measurable) stopping times all $\eval_{(j)}$ will still be $\S$-variables and Proposition~\ref{prop:optional.stopping} can still be applied and Ville-Robbins still holds --- so the restriction is quite harmless.
}
\paragraph{Filtration: Conditional Distributions}
In the $2 \times 2$ setting the $\Theta_1$ represent conditional distributions of $Y \mid X$. 
While neither Section~\ref{sec:os} nor Section~\ref{sec:GROanalysis} formally allowed for that setting, the extension is straightforward. We simply assume the underlying streams are of the form $(X_{m,1},Y_{m,1}), (X_{m,2}, Y_{m,2}), \ldots$ 
(with, in the $2 \times 2$ example, $X_{m,i} \in \{a,b\}$). The distributions in $\cH_1$ are now extended to define a random process of independent outcomes with the same conditional distribution for a single such stream $(X_1,Y_1), (X_2,Y_2), \ldots$, i.e.\ for all $\theta_1 \in \Theta_1$, we set $p_{\theta_1}(y^n \mid x^n) = 
P_{\theta_1}(Y^n =y^n \mid X^n = x^n) := \prod_{i=1}^n  p_{\theta_1}(y_i \mid x_i)$. We then only need to extend Definition~\ref{def:cond.safe} of conditional $\S$-variables to deal with this extension. This is achieved by setting $\cF_{(m-1)}$ in the definition to $\sigma(\bY^{(m-1)},\bX^{(m)})$.
\paragraph{Filtration: Side Information}
Now we consider how the set-up can be extended to deal with side-information that may be used e.g.\ after $j$ studies to decide whether to start a new, $j+1$st study at all, and if so, what the sample size of that study will be.
%
For this we again need to extend $\cH_0$ and $\cH_1$ so that its elements define a conditional random process, with at the time that the $j$-th study has just been observed, also the additional variables $\bR_{(1)}, \ldots, \bR_{(j)}$ observed. Even if there are underlying streams of data $Y_{j,1}, Y_{j,2}, \ldots$ so that the $\bY_{(j)}$ have an internal structure, the $\bR_{(j)}$ are not required to have such a structure.
To make all desired probabilities well-defined and at the same time make sure that the side-information is really external, we impose a conditional independence on the underlying stream: 
under each $P \in \cH_0$, for each $m,i \in \naturals$, 
the conditional distribution of $Y_{m,i}$ given $Y_{m}^{i-1}$ and $\bR^{(m-1)}$ is defined to be the same as the distribution under $P$ of $Y_{m,i}$ given $Y_m^{i-1}$, which is already well-defined once $\cH_0$ is specified. With this definition, we can set $\cF_{(j)} = \sigma(\bY^{(j)}, \bR^{(j)})$ (or $\cF_{(j)} = \sigma({\bf Y}^{(j)}, \bR^{(j)}, \bX^{(j+1)})$ if the $\cH_j$ already contain conditional distributions, or $\cF_{(j)} = \sigma({\bf U}^{(j)}, \bR^{(j)}, \bX^{(j+1)})$ with ${\bf U}^{(j)}$ a coarsening of $\bY^{(j)}$ if required, such as in the $t$-test setting). Focusing on the simplest case with $\cF_{(j)} = \sigma(\bY^{(j)}, \bR^{(j)})$, the construction ensures that 
$\Exp_P[\eval_{(j)} \mid \bY^{(j-1)},\bR^{(j-1)}] \leq 1$ 
(and hence we have safety under OC) 
in both cases discussed in Section~\ref{sec:GROanalysis}: we can either use the plug-in application of any $\S$-variable specification, or, if the specification is {\em seqdec}, we can also use the sequential application.  Note that with this construction, the $\bR^{(j)}$ are allowed to depend on $\bY_{(1)}, \ldots, \bY_{(j)}$ in unspecified ways. This is unproblematic because (other than in the case with conditional distributions and $\bX^{(j+1)}$) the $\eval_{(j+1)}$ cannot depend on $\bR^{(j+1)}$. For example, based on what she sees in the $\bY_{(j)}$, your boss may decide to announce `we have money to do an additional study with 100 patients', which can be encoded as a particular outcome of $\bR_{(j)}$. This may then be used to decide to continue (i.e.\ set $\tau$ in Proposition~\ref{prop:optional.stopping} to be larger than  $j$) and set $N_{(j+1)}$ to be $100$.


\subsection{Section~\ref{sec:grow}}
\label{app:remainingproofsGROW}
\paragraph{Proof of Proposition~\ref{prop:dominance}}
The monotone likelihood ratio property implies stochastic dominance \citep{lehmann2005testing}, i.e.\ with $P{[T]}$ denoting the distribution of the statistic $T$,
we must have $\Exp_{P_{\delta}{[T]}}[f(T)] \geq
\Exp_{P_{\delta'}}[f(T)]$ for $\delta \geq \delta'$ and every increasing function $f$. This implies that
\begin{equation}\label{eq:painters}
D(P_{\delta^+} \| P_{\delta^-}) =     \Exp_{\Ydata \sim P_{\delta^+}} \left[ \log
\frac{p_{\delta^+}(\Ydata)}{p_{\delta^-}(\Ydata)} \right] =
\inf_{\delta \geq \delta^+} \Exp_{\Ydata \sim P_{\delta}} \left[ \log
\frac{p_{\delta^+}(\Ydata)}{p_{\delta^-}(\Ydata)} \right].
\end{equation}
We also have, by the same stochastic dominance result, 
for $\delta \leq \delta^-$,
\begin{align*}
\Exp_{\bY \sim P_{\delta}}\left[ 
\frac{p_{\delta^+}(\bY)}{p_{\delta^-}(\bY)}\right]
\leq \Exp_{\bY \sim P_{\delta^-}}\left[ 
\frac{p_{\delta^+}(\bY)}{p_{\delta^-}(\bY)}\right] = 1,
\end{align*}    
so that $\eval^* = p_{\delta^+}(\bY)/p_{\delta^-}(\bY)$ is an $\S$-variable, which directly leads to the first inequality in the chain of (in)equalities (\ref{eq:colours}) below. The first equality follows by (\ref{eq:painters}), the second because, since $\eval^*$ is of form $p_{\delta^+}/p_{W_0}$, with $W_0 \in \cE(\Theta_0)$ (namely, $W_0$ puts all mass on $\delta^-$), it must, by Corollary~\ref{cor:justone}, be the GRO-$\S$-variable for testing between  $\Theta'_1 = \{\delta^+\}$ and $\Theta_0$. The final two inequalities are immediate: 
\begin{align}\label{eq:colours}
& \sup_{\eval\in \cE(\Theta_0)} \inf_{\delta \geq \delta^+} \Exp_{\Ydata \sim P_{\delta}} \left[ 
\log \eval \right]   \geq  \inf_{\delta  \geq \delta^+} \Exp_{\Ydata \sim P_{\delta}} \left[  \log
\frac{p_{\delta^+}(\Ydata)}{p_{\delta^-}(\Ydata)} \right]
=   \Exp_{\Ydata \sim P_{\delta^+}} \left[ \log
\frac{p_{\delta^+}(\Ydata)}{p_{\delta^-}(\Ydata)} \right] = \nonumber
\\   &  \sup_{\eval\in \cE(\Theta_0)} \Exp_{\Ydata \sim P_{\delta^+}} \left[ 
\log \eval \right] \geq \inf_{\delta \geq \delta^+}
\sup_{\eval\in \cE(\Theta_0)} \Exp_{\Ydata \sim P_{\delta}} \left[ 
\log \eval \right] \geq \sup_{\eval\in \cE(\Theta_0)} \inf_{\delta \geq \delta^+} \Exp_{\Ydata \sim P_{\delta}} \left[ 
\log \eval \right].
\end{align}
This chain of inequalities implying that all its parts are equal, the result follows. 
\subsection{Section~\ref{sec:group}}
\label{app:remainingproofsttest}
\paragraph{Proof of (\ref{eq:suffu})}
\eqref{eq:suffu} follows from \cite[page 273]{lai1976confidence} or as special case of Theorem 2.1.\ of \cite{berger1998bayes}, but the first proof leaves out details and the second is very abstract, so for convenience we give a direct proof. For simplicity restrict to the case with $W$ putting all its mass on a particular $\delta$. Fix arbitrary $\sigma > 0$ and $n \geq 2$ and  note that $V_1 \in \{-1,1\}$ and  $P_{\delta}(V_1 = 1)= P_{\delta,\sigma}(Y_1 > 0)$
(note that $p'_{W[\delta]}(V_1)$ is a probability mass function, whereas $p'_{W[\delta]}(V_i \mid V^{i-1})$ is defined as density relative to Lebesgue measure for $i> 1$). We must then have:
\begin{align*}
& P_{\delta}(V_1=1) \cdot p'_{\delta}(v_2, \ldots, v_n \mid V_1 = 1) 
= P_{\delta}(V_1=1) \cdot \int_0^\infty 
p_{\delta/|y_1|,\sigma/|y_1|}(v^n \mid Y_1 = y_1) 
p_{\delta,\sigma}(y_1 \mid Y_1 > 0) d y_1 \\
& = \int_0^{\infty} \prod_{i=2}^n  \left( 
\frac{1}{\sqrt{2 \pi} } \cdot \frac{|y_1|}{\sigma} \cdot e^{- \frac{1}{2}
	\left( \frac{v_i}{\sigma/|y_1|} - \delta \right)^2} \right) \cdot 
\left( \frac{1}{\sqrt{2 \pi \sigma^2}} e^{- \frac{1}{2} \left(\frac{y_1}{\sigma} - \delta \right)^2 }\right)   d y_1   \\
&= \int_0^{\infty} \prod_{i=1}^n  \left( 
\frac{1}{\sqrt{2 \pi} } \cdot \frac{|y_1|}{\sigma} \cdot e^{- \frac{1}{2}
	\left( \frac{v_i}{\sigma/|y_1|} - \delta \right)^2} \right) \cdot \frac{1}{|y_1|} d y_1 \\ 
&= \int_0^{\infty} \prod_{i=1}^n  \left( 
\frac{1}{\sqrt{2 \pi} } \cdot \frac{1}{\tau} \cdot e^{- \frac{1}{2}
	\left( \frac{v_i}{\tau} - \delta \right)^2} \right) \cdot \frac{\tau}{\sigma } \left|\frac{d y_1}{ d \tau}\right|  d \tau   = 
\int_0^{\infty} \prod_{i=1}^n  \left( 
\frac{1}{\sqrt{2 \pi} } \cdot \frac{1}{\tau} \cdot e^{- \frac{1}{2}
	\left( \frac{v_i}{\tau} - \delta \right)^2} \right) \cdot \frac{1}{\tau} d \tau 
\end{align*}
Here in the first equality we used that, given $Y_1= y_1$, the $V_i$ are independent Gaussian,  with variance $\sigma/|y_1|$ and mean $\delta \sigma/|y_1|$ hence effect size $\delta/|y_1|$ and hence density $p_{\delta/|y_1|,\sigma/|y_1|}$. The second equality replaces the conditional density of $Y_1$ by the marginal (so that $P_{\delta}(V_1=1)$ cancels) and exploits that $y_1 = v_1 |y_1|$, the third is a standard change-of-variable from  $\sigma/|y_1|$ to $\tau$ using the Jacobian transformation and the fourth is immediate.  

This shows the desired result if $V_1=1$. The case for $V_1 = -1$ is analogous.
\commentout{goes analogously as follows:
	\begin{align*}
	& P_{\delta}(V_1= - 1) \cdot p'_{\delta}(v_2, \ldots, v_n \mid V_1 = - 1) \\ & = P_{\delta}(V_1= - 1) \cdot \int_{- \infty}^0 
	p_{\delta/|y_1|,\sigma/|y_1|}(v^n \mid Y_1 = y_1) 
	p_{\delta,\sigma}(y_1 \mid Y_1 < 0) d y_1 \\
	& = \int_{-\infty}^0 \prod_{i=2}^n  \left( 
	\frac{1}{\sqrt{2 \pi} } \cdot \frac{|y_1|}{\sigma} \cdot e^{- \frac{1}{2}
		\left( \frac{v_i}{\sigma/|y_1|} - \delta \right)^2} \right) \cdot 
	\left( \frac{1}{\sqrt{2 \pi \sigma^2}} e^{- \frac{1}{2} \left(\frac{y_1}{\sigma} - \delta \right)^2 }\right)   d y_1   \\
	& = \int_{-\infty}^0 \prod_{i=2}^n  \left( 
	\frac{1}{\sqrt{2 \pi} } \cdot \frac{|y_1|}{\sigma} \cdot e^{- \frac{1}{2}
		\left( \frac{v_i}{\sigma/|y_1|} - \delta \right)^2} \right) \cdot 
	\left( \frac{1}{\sqrt{2 \pi \sigma^2}} e^{- \frac{1}{2} \left(\frac{v_1 |y_1|}{\sigma} - \delta \right)^2 }\right)   d y_1   \\
	&= \int_0^{\infty} \prod_{i=1}^n  \left( 
	\frac{1}{\sqrt{2 \pi} } \cdot \frac{|y_1|}{\sigma} \cdot e^{- \frac{1}{2}
		\left( \frac{v_i}{\sigma/|y_1|} - \delta \right)^2} \right) \cdot \frac{1}{|y_1|} d y_1 \\ 
	&= \int_0^{\infty} \prod_{i=1}^n  \left( 
	\frac{1}{\sqrt{2 \pi} } \cdot \frac{1}{\tau} \cdot e^{- \frac{1}{2}
		\left( \frac{v_i}{\tau} - \delta \right)^2} \right) \cdot \frac{\tau}{\sigma } \left|\frac{d y_1}{ d \tau}\right|  d \tau   = 
	\int_0^{\infty} \prod_{i=1}^n  \left( 
	\frac{1}{\sqrt{2 \pi} } \cdot \frac{1}{\tau} \cdot e^{- \frac{1}{2}
		\left( \frac{v_i}{\tau} - \delta \right)^2} \right) \cdot \frac{1}{\tau} d \tau 
	\end{align*}}
\paragraph{Applicability of Proposition~\ref{prop:dominance} to the $t$-test Setting}
(Section~\ref{sec:group} underneath Theorem~\ref{thm:particular})
A simple calculation gives that 
$p'_{\delta^+}(\bV)/p'_{\delta^-}(\bV)$ can be re-expressed as a density ratio of the $t$-statistic $T = t(\bf Y)$, i.e.\ 
$p'_{\delta^+}(\bV)/p'_{\delta^-}(\bV)= p''_{\delta^+}(t(\bY))/p'_{\delta^-}(t(\bY))$, where $p''_{\delta}$ is the density of a noncentral $t$-distribution with $\nu:= n - 1$ degrees of freedom and noncentrality parameter $\mu=
\sqrt{n} \delta$. But these densities are well-known to form a monotone likelihood ratio family in the $T$ statistic, so that we can apply Proposition~\ref{prop:dominance} to $\{ p''_{\delta} : \delta \in \Delta \}$.

\subsection{Section~\ref{sec:competitive}}
\label{app:brown}
\paragraph{Determining Sample Size for a Desired Power}
Consider a 1-sided test for the normal location family with variance $1$ which rejects if $\hat\mu \geq f(n)/\sqrt{n}$ where $\hat\mu$ is the MLE at sample size $n$ and $f$ is some increasing function of $n$. We want to find the smallest $n$ at which we achieve power $1-\beta$ under mean $\delta$, i.e.\ such that $$P_{\delta}(\sqrt{n}(\hat\mu - \delta) \geq f(n) - \delta \sqrt{n})\geq 1-\beta,$$ where under  $P_{\delta}$, the $Y_1, \ldots,Y_n$  are i.i.d.\ $N(\delta,1)$. This is the smallest $n$ at which 
$f(n) - \delta \sqrt{n} \geq - z_{\beta}$, i.e.\ 
$\sqrt{n} \geq (z_{\beta}+ f(n))/\delta$. 
The standard result for $n_{\textsc{np}}$ now follows by taking $f(n) = z_{\alpha}$. For the Bayesian test, to very good approximation, $f(n) = \sqrt{6 + \log n}$ (Example~\ref{ex:normal}). 
Since, $n_{\textsc{Bayes}}$, the smallest $n$ for the Bayesian test must be larger than $n_{\textsc{np}}$, it satisfies $n_{\textsc{Bayes}}/n_{\textsc{np}} \geq (z_{\beta} + \sqrt{6 + \log n_{\textsc{Bayes}}})^2 /(z_{\beta} + z_{\alpha})^2 \geq (z_{\beta} + \sqrt{6 + \log n_{\textsc{np}}})^2 /(z_{\beta} + z_{\alpha})^2$, giving a logarithmic factor as claimed. 
\paragraph{Brownian Motion}
Fix $a > 0$ and, for a standard Brownian motion $X_t = B_t + c t $ with drift $c$, define $S = \min \{t > 0: X_t \geq a \}$. The distribution of $S$ is well-known (see e.g.\ \citep{BhattacharyaW21}) and has density  given by
\begin{equation}\label{eq:hittingtime}
f_{a,c}(s) =  \frac{a \exp \left(- \frac{(a- cs)^2}{2s} \right)}{\sqrt{2 \pi s^3}}.
\end{equation}
Fix $\delta$ and let  $n_t = t/\delta^2$ and  $\cT = \{ \delta^2, 2 \delta^2, 3 \delta^2, \ldots\}$. Consider, for $t \in \cT$, the discrete time process 
$$W_t := \log \frac{ p_{\delta}(Y^{n_t})}{p_0(Y^{n_t})}$$
where $p_{\delta}$ is the density of $Y^n$ under $N(\delta,1)$. 
Writing out the definition and re-arranging, we find that if $Y_1, Y_2, \ldots$ are i.i.d.\ $\sim N(\delta,1)$, then for all $\cT' \subset \cT$
the conditional distribution of $W_t$ given $\{W_t : t \in \cT'\}$ (in particular, this includes the marginal distribution of $W_t$ if we take $\cT'$ empty) agrees with the conditional distribution of $X_t = B_t + (1/2) t$ and we can thus approximate the distribution of the first time when $W_t$ exceeds $- \log \alpha$ by the distribution with density (\ref{eq:hittingtime}) with $c=1/2$ and $a = -\log \alpha$ --- the distribution of the first hitting time of $B_t$ will be shifted slightly to the left, because when stopping $W_t$ we are only checking the process at intervals of size $\delta^2$. Intuitively, as $\delta \downarrow 0$, we expect the distribution functions to converge. To make this concrete, let $q_{\beta}$ be the $\beta$-quantile of the distribution given by $f_{-\log\alpha ,1/2}(s)$. 
We want to calculate $n_{\max}(\beta)$, the smallest $n$ such that $P_{\delta}(\tau_1 \leq n) > 1- \beta$, i.e.
we want to find the smallest $n$ such that, with $t^* = \delta^2  n$, we have 
$P_{\delta}(\delta^2 \tau_1 > t^*) < \beta$. The correspondence between $W_t$ and $X_t$ suggests that in the  
limit for $\delta \downarrow 0$, $t^*$ converges to 
$q_{\beta}$, giving that $n_{\max}(\beta) \sim  C_{\beta,0}/\delta^2$ with $C_{\beta,0} = q_{\beta}$ and $\Exp_{P_{\delta}}[\tau_{\beta}]\sim C'_{\beta,0} /\delta^2$ with 
$C'_{\beta,0} = \int_0^{q(\beta)} t f(t) d t + \beta q_{\beta}$.  

\end{document}